\documentclass{article}
\pdfoutput=1

\usepackage[utf8]{inputenc}
\usepackage[english]{babel}
\usepackage[T1]{fontenc}
\usepackage{amsmath}
\usepackage{amsfonts}
\usepackage{amssymb}
\usepackage{amsthm}
\usepackage{tikz}
\usepackage{enumitem}
\usetikzlibrary{calc, intersections, fpu}
\usetikzlibrary{arrows.meta} 
\usetikzlibrary{decorations.text}
\usepackage{tikz-cd}
\usepackage[colorlinks=true, linkcolor=blue, hyperindex=true, citecolor = green!60!purple]{hyperref}
\usepackage{comment}
\usepackage{geometry}
\usepackage{dsfont}
\usepackage{cleveref}

\newcommand{\R}{\mathbb{R}}
\newcommand{\To}{\mathbb{T}}
\newcommand{\Sp}{\mathbb{S}}
\newcommand{\T}{\mathcal{T}}

\newcommand{\control}[2]{ ($#2+#1$) .. #1}
\newcommand\tikzmark[5][0]{\tikz[overlay,remember picture,baseline] \node [rotate=#1,anchor=base,xshift=#4,yshift=#5] (#2) {$\scriptstyle#3$};}
\newcommand\modulo[3]{\pgfmathparse{mod (#1,#2)}\pgfmathtruncatemacro{#3}{\pgfmathresult}}
\newcommand\euclidiv[3]{\pgfmathparse{int(#1/#2)}\pgfmathtruncatemacro{#3}{\pgfmathresult}}

\newcommand{\cross}{\operatorname{cr}}
\newcommand{\addcross}{\operatorname{Top}}
\newcommand{\add}{\operatorname{Add}}
\newcommand{\unlink}{\operatorname{CC}}

\newtheorem{theorem}{Theorem}[section]
\newtheorem{definition}[theorem]{Definition}
\newtheorem{proposition}[theorem]{Proposition}

\newtheorem{remark}[theorem]{Remark}
\newtheorem{lemma}[theorem]{Lemma}
\newtheorem{conjecture}[theorem]{Conjecture}

\title{Hard diagrams of split links} 

\author{Corentin Lunel\thanks{INRIA Université Côte d'Azur, Montpellier, France, corentin.lunel@inria.fr} \and Arnaud de Mesmay\thanks{LIGM, CNRS, Univ. Gustave Eiffel, ESIEE Paris, F-77454 Marne-la-Vall\'ee, France, arnaud.de-mesmay@univ-eiffel.fr} \and Jonathan Spreer\thanks{School of Mathematics and Statistics, University of Sydney, Australia, jonathan.spreer@sydney.edu.au}}

\date{}

\begin{document}
\maketitle

\begin{abstract}
Deformations of knots and links in ambient space can be studied combinatorially on their diagrams via local modifications called Reidemeister moves. While it is well-known that, in order to move between equivalent diagrams with Reidemeister moves, one sometimes needs to insert excess crossings, there are significant gaps between the best known lower and upper bounds on the required number of these added crossings. In this article, we study the problem of turning a diagram of a split link into a split diagram, and we show that there exist split links with diagrams requiring an arbitrarily large number of such additional crossings. More precisely, we provide a family of diagrams of split links, so that any sequence of Reidemeister moves transforming a diagram with $c$ crossings into a split diagram requires going through a diagram with $\Omega(\sqrt{c})$ extra crossings. Our proof relies on the framework of bubble tangles, as introduced by the first two authors, and a technique of Chambers and Liokumovitch to turn homotopies into isotopies in the context of Riemannian geometry.
\end{abstract}


\smallskip


\section{Introduction}\label{sec_intro}

The Reidemeister theorem~\cite{Reidemeister_reidemeister_th} is a fundamental and powerful result in low-dimensional topology. It states that any two link diagrams represent equivalent links\footnote{A link is an embedding of a collection of circles into $\R^3$ and two links are considered equivalent if they are ambient isotopic. Link diagrams are projections of links into a plane as shown, for instance, in \Cref{pic_goeritz}. We refer to~\Cref{sec_prelim} for standard definitions in knot theory.} if and only if they can be related by a sequence of planar isotopies and local moves, called \textbf{Reidemeister moves}, pictured in \Cref{pic_Reide_move}. This theorem is at the heart of many theoretical results as well as computational applications. Indeed, many knot invariants, from the most basic ones such as tri-colourability~\cite[Section~1.5]{Adams_knotbook} to more advanced ones such as the Jones polynomial~\cite{Kauffman_Jones_poly} or Khovanov homology~\cite{bar2002khovanov}, can be shown to be invariants by demonstrating that they are not modified by Reidemeister moves. From a computational point of view, the Reidemeister theorem allows for a discretisation of the space of possible transformations to consider when testing for knot equivalence. This fact is at the root of a straightforward algorithm to study algorithmic problems in knot theory: at the level of diagrams, apply Reidemeister moves in a random or brute-force manner until a desired property is verified. 

\begin{figure}[ht]
\begin{center}
\begin{tikzpicture}[scale=0.95]
\clip (-0.1,-0.1) rectangle (14.4,2.1);
\begin{scope}[very thick]
\draw (0,0) -- +(0,2);
\draw (1.5, 0) .. controls +(90:0.7) and \control{($(-135:0.1)+(1.75,1)$)}{(-135:0.3)}; 
\draw ($(45:0.1)+(1.75,1)$) .. controls+(45:0.7) and \control{($(-45:0.1)+(1.75,1)$)}{(-45:0.7)} -- +(135:0.2) .. controls +(135:0.3) and \control{(1.5,2)}{(-90:0.7)};
\draw [Stealth-Stealth] (0.25,1) -- (1.25,1) node [midway, above] {RI};

\begin{scope}[xshift=3.75cm]
\draw (0,0) -- ++(0,2);
\draw (1,0) -- ++(0,2);

\draw [Stealth-Stealth] (1.25,1) -- (2.25,1) node [midway, above] {RII};

\draw (2.5,0) .. controls +(60:0.75) and \control{(3.5,1)}{(-90:0.3)} .. controls +(90:0.3) and \control{(2.5,2)}{(-60:0.75)}; 
\draw [line width = 0.2 cm, white] (3.5,0) .. controls +(120:0.75) and \control{(2.5,1)}{(-90:0.3)} .. controls +(90:0.3) and \control{(3.5,2)}{(-120:0.75)}; 
\draw (3.5,0) .. controls +(120:0.75) and \control{(2.5,1)}{(-90:0.3)} .. controls +(90:0.3) and \control{(3.5,2)}{(-120:0.75)}; 
\end{scope}

\begin{scope}[xshift=8.75cm]
\draw (0,0) -- ++(2,2);
\draw [line width = 0.2 cm, white] (2,0) -- ++(-2,2);
\draw (2,0) -- ++(-2,2);
\draw [line width = 0.2 cm, white] (0,0.6) -- ++(2,0);
\draw (0,0.6) -- ++(2,0);

\draw [Stealth-Stealth] (2.25,1) -- (3.25,1) node [midway, above] {RIII};

\draw (3.5,0) -- ++(2,2);
\draw [line width = 0.2 cm, white] (5.5,0) -- ++(-2,2);
\draw (5.5,0) -- ++(-2,2);

\draw [line width = 0.2 cm, white] (3.5,1.4) -- ++(2,0);
\draw (3.5,1.4) -- ++(2,0);
\end{scope}
\end{scope}
\end{tikzpicture}
\caption{The three Reidemeister moves RI, RII, and RIII.}
\label{pic_Reide_move}
\end{center}
\end{figure}
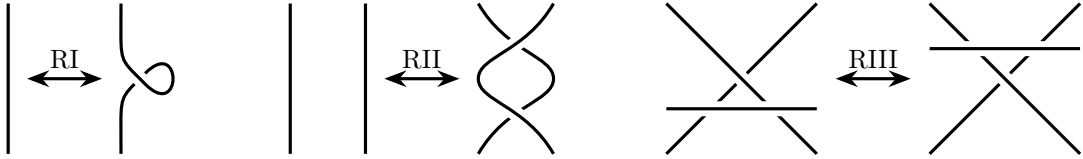

A primary example of such a knot theory problem is recognising the \textbf{unknot}, that is, the unique knot admitting a diagram with no crossings. This is a first instance of the fundamental problem of knot theory of deciding whether two knots are equivalent. It turns out that some unknot diagrams, called \textbf{hard unknot diagrams} or \textbf{culprits}~\cite{Mesmay_hard_unknot,Kauffman_untangling}, exhibit an unwanted behaviour for the above algorithm. Namely, given an initial diagram $\mathcal{D}$ of the unknot, the largest number of crossings of a diagram in any sequence of Reidemeister moves from $\mathcal{D}$ to the $0$-crossing diagram must be larger than in $\mathcal{D}$. This means one first needs to \emph{add} crossings before being able to reach the $0$-crossing diagram. The existence of such diagrams implies that it is not possible to untangle an unknot diagram by only applying Reidemeister moves that do not increase the number of crossings. An example of such a hard unknot diagram, called the \textbf{Goeritz culprit}, is shown in \Cref{pic_goeritz}. In~\cite{Mesmay_hard_unknot}, it is shown that, when diagrams in $\mathbb{S}^2$ are considered, at least one extra crossing is required to untangle this unknot. We do not know of systematic techniques or methods to prove easily that a given diagram of the unknot is hard. Instead, all proofs known to the authors resort to an exhaustive search in the graph of Reidemeister moves, which quickly becomes infeasible. Recently, new techniques based on reinforcement learning have been applied to find large numbers of hard unknot diagrams~\cite{applebaum2024unknotting}, but the proofs of hardness still involved exhaustive enumerations.

\begin{figure}[ht]
\begin{center}
\begin{tikzpicture}[scale=0.8]
\def\e{0.2}
\def\a{45}
\def\ln{0.45}
\def\il{1}
\def\ol{1}
\def\ll{2.5}
\def\lm{3}
\clip (-7,-2.5) rectangle (6,2.5);

\coordinate (l0) at (-6,0);
\coordinate (l1) at (-5,0);
\coordinate (l2) at (-4,0);
\coordinate (l3) at (-3,0);
\coordinate (u0) at (-0.5,1.5);
\coordinate (u1) at (0.5,1.5);
\coordinate (b0) at (-0.5,-1.5);
\coordinate (b1) at (0.5,-1.5);
\coordinate (r0) at (3,0);
\coordinate (r1) at (4,0);
\coordinate (r2) at (5,0);

\begin{scope}[thick]
\begin{scope}[even odd rule]
\clip(-7,-2.5) rectangle (6,2.5) (l1) circle (\e cm);
\clip(-7,-2.5) rectangle (6,2.5) (u1) circle (\e cm);
\draw (l1) .. controls +(\a:-\ln) and \control{(l0)}{(-\a:\ln)} .. controls +(-\a:-\lm) and \control{(u0)}{(-\a:-\ll)} .. controls +(-\a:\ln) and \control{(u1)}{(\a:-\ln)};
\end{scope}
\begin{scope}[even odd rule]
\clip(-7,-2.5) rectangle (6,2.5) (l0) circle (\e cm);
\clip(-7,-2.5) rectangle (6,2.5) (l2) circle (\e cm);
\draw (l0) .. controls +(\a:\ln) and \control{(l1)}{(-\a:-\ln)} .. controls +(-\a:\ln) and \control{(l2)}{(\a:-\ln)};
\end{scope}
\begin{scope}[even odd rule]
\clip(-7,-2.5) rectangle (6,2.5) (l1) circle (\e cm);
\clip(-7,-2.5) rectangle (6,2.5) (l3) circle (\e cm);
\draw (l1) .. controls +(\a:\ln) and \control{(l2)}{(-\a:-\ln)} .. controls +(-\a:\ln) and \control{(l3)}{(\a:-\ln)};
\end{scope}
\begin{scope}[even odd rule]
\clip(-7,-2.5) rectangle (6,2.5) (b0) circle (\e cm);
\clip(-7,-2.5) rectangle (6,2.5) (l2) circle (\e cm);
\draw (l2) .. controls +(\a:\ln) and \control{(l3)}{(-\a:-\ln)} .. controls +(-\a:\il) and \control{(b0)}{(-\a:-\ol)};
\end{scope}
\begin{scope}[even odd rule]
\clip(-7,-2.5) rectangle (6,2.5) (u0) circle (\e cm);
\clip(-7,-2.5) rectangle (6,2.5) (l3) circle (\e cm);
\draw (l3) .. controls +(\a:\il) and \control{(u0)}{(\a:-\ol)};
\end{scope}
\begin{scope}[even odd rule]
\clip(-7,-2.5) rectangle (6,2.5) (l0) circle (\e cm);
\clip(-7,-2.5) rectangle (6,2.5) (b1) circle (\e cm);
\draw (l0) .. controls +(\a:-\lm) and \control{(b0)}{(\a:-\ll)} .. controls +(\a:\ln) and \control{(b1)}{(-\a:-\ln)};
\end{scope}
\begin{scope}[even odd rule]
\clip(-7,-2.5) rectangle (6,2.5) (b0) circle (\e cm);
\clip(-7,-2.5) rectangle (6,2.5) (r1) circle (\e cm);
\draw (b0) .. controls +(-\a:\ln) and \control{(b1)}{(\a:-\ln)} .. controls +(\a:\ol) and \control{(r0)}{(\a:-\il	)}.. controls +(\a:\ln) and \control{(r1)}{(-\a:-\ln)};
\end{scope}
\begin{scope}[even odd rule]
\clip(-7,-2.5) rectangle (6,2.5) (r0) circle (\e cm);
\clip(-7,-2.5) rectangle (6,2.5) (r2) circle (\e cm);
\draw (r0) .. controls +(-\a:\ln) and \control{(r1)}{(\a:-\ln)} .. controls +(\a:\ln) and \control{(r2)}{(-\a:-\ln)};
\end{scope}
\begin{scope}[even odd rule]
\clip(-7,-2.5) rectangle (6,2.5) (u0) circle (\e cm);
\clip(-7,-2.5) rectangle (6,2.5) (r0) circle (\e cm);
\draw (u0) .. controls +(\a:\ln) and \control{(u1)}{(-\a:-\ln)} .. controls +(-\a:\ol) and \control{(r0)}{(-\a:-\il)};
\end{scope}
\begin{scope}[even odd rule]
\clip(-7,-2.5) rectangle (6,2.5) (u1) circle (\e cm);
\clip(-7,-2.5) rectangle (6,2.5) (r1) circle (\e cm);
\draw (u1) .. controls +(\a:\ll) and \control{(r2)}{(\a:\lm)} .. controls +(\a:-\ln) and \control{(r1)}{(-\a:\ln)};
\end{scope}
\begin{scope}[even odd rule]
\clip(-7,-2.5) rectangle (6,2.5) (b1) circle (\e cm);
\clip(-7,-2.5) rectangle (6,2.5) (r2) circle (\e cm);
\draw (r2) .. controls +(-\a:\lm) and \control{(b1)}{(-\a:\ll)};
\end{scope}
\end{scope}
\end{tikzpicture}
\caption{The Goeritz culprit: using Reidemeister moves in $\mathbb{S}^2$, one must add at least one extra crossing to untangle this unknot diagram.}
\label{pic_goeritz}
\end{center}
\end{figure}
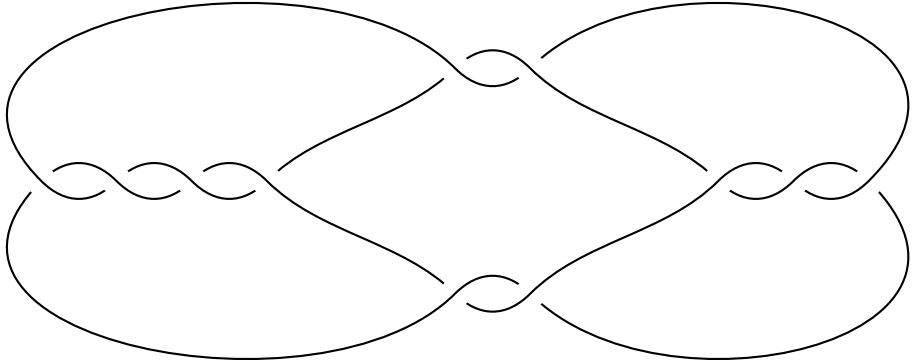

Following the notations of \cite{Mesmay_hard_unknot, Kauffman_untangling} we denote the number of crossings in a diagram $\mathcal{D}$ by $\cross(\mathcal{D})$. For diagrams $\mathcal{D}_1$ and $\mathcal{D}_2$ of equivalent links, and a sequence $R$ of $r$ Reidemeister moves transforming $\mathcal{D}_1$ into $\mathcal{D}_2$, we define \textbf{$\text{Top}(\mathcal{D}_1,R)$}$ = \max_{i\in \{0, 1, \ldots , r\}} \{\cross(\mathcal{D}^i) - \cross(\mathcal{D}_1)\}$ where $\mathcal{D}^i$, $0\leq i \leq r$, is the diagram after performing the first $i$ moves of the sequence. The minimal number of extra crossings to pass from $\mathcal{D}_1$ to $\mathcal{D}_2$ is denoted by \textbf{$\text{Add}(\mathcal{D}_1,\mathcal{D}_2)$}, which, formally, is the minimum of $\addcross(\mathcal{D}_1,R)$ taken over all sequences of Reidemeister moves $R$ transforming $\mathcal{D}_1$ into $\mathcal{D}_2$. Hence, $\add(\mathcal{D}_1,\mathcal{D}_2)$ is a lower bound on the number of crossings we must add in any sequence of Reidemeister moves performed on $\mathcal{D}_1$ to reach $\mathcal{D}_2$. 

Let $\mathcal{D}$ be a diagram of the unknot and $\hat{\mathcal{D}}$ be its $0$-crossing diagram. Naturally, $\mathcal{D}$ is a hard unknot diagram if and only if $\add(\mathcal{D},\hat{\mathcal{D}})$ is positive. This measure of complexity is called $m$ in \cite{Mesmay_hard_unknot} (see also the \emph{recalcitrance} in~\cite{Kauffman_untangling}). Studying these complexity measures and hard unknot diagrams is trickier than one might think. In fact, one of the purposes of \cite{Mesmay_hard_unknot} is to confirm or invalidate claims about previously known hard unknot diagrams using an exhaustive computer search over all possible sequences of Reidemeister moves. Still, the ``hardest'' known diagrams of the unknot have only been verified to require at least three extra crossings before they can be untangled. In contrast, the following conjecture is folklore.

\begin{conjecture}
  \label{conjm}
  Let $m$ be an integer and let $\hat{\mathcal{D}}$ be the $0$-crossing diagram of the unknot. Then there exists a diagram of the unknot $\mathcal{D}$ with $n$ crossings such that any sequence of Reidemeister moves from $\mathcal{D}$ to $\hat{\mathcal{D}}$ passes through a diagram with at least $n+m$ crossings.
\end{conjecture}

Note that a proof of \Cref{conjm}, formulated in terms of recalcitrance, is claimed in \cite{Kauffman_untangling}, but concerns about the correctness of this proof are raised in \cite{Mesmay_hard_unknot}.

\medskip

In this article, as a possible step towards a proof of \Cref{conjm}, we shift our focus to \textbf{split links}.
A link $L$ is said to be split if there exists a sphere disjoint from $L$ separating at least two link components of $L$. If such a sphere exists, there exists a link diagram in which two sublinks of $L$ are disjoint: they are separated by a circle in the plane. Such a diagram is called a \textbf{split diagram}. By capping off the aforementioned circle with one disc above and one disc below the plane of projection, we can verify that conversely a split diagram witnesses a split link.
Determining if a link is split is known as the \textbf{splitting problem}. 

Given a split link $L$ with a diagram $\mathcal{D}_1$, we study $\add(\mathcal{D}_1,\mathcal{D}_2)$ where $\mathcal{D}_2$ is a split diagram of $L$. If the minimum of $\add(\mathcal{D}_1,\mathcal{D}_2)$ over all split diagrams $\mathcal{D}_2$ of $L$, called the \textbf{crossing-complexity} and denoted by $\textbf{\text{CC}}(D_1)$, is positive, we call $\mathcal{D}_1$ a \textbf{hard split link}. 

\subparagraph*{Our results.} We exhibit a family of link diagrams $\mathcal{D}(p,q)$ of split links $L(p,q)$ with two unlinked sublinks. The first sublink $M=M(p,q)$ is made of two linked torus knots $T_{p,q}$, and the second is an unknot $U$ surrounding one of the torus knots (see \Cref{pic_def_diag} for an illustration). 
For any split diagram $\mathcal{D}'(p,q)$ of $L(p,q)$, we prove \Cref{th_unlink_analysed}, implying that $\add(\mathcal{D}(p,q),\mathcal{D'}(p,q)) = \Omega (\min (p,q))$. More precisely, we have the following main theorem.

\begin{theorem}\label{th_unlink_analysed}
For all $n \geq 2$, any sequence of Reidemeister moves transforming diagram $\mathcal{D}(n,n+1)$ of the link $L(n,n+1)$ with $2n^2 +2$ crossings into a split diagram passes through a diagram with at least $2n^2 +\frac{2}{3} n$ crossings. In particular, there exist hard split links of arbitrarily large crossing-complexity. 
\end{theorem}

\begin{figure}[ht]
\begin{center}
\includegraphics[scale=0.415]{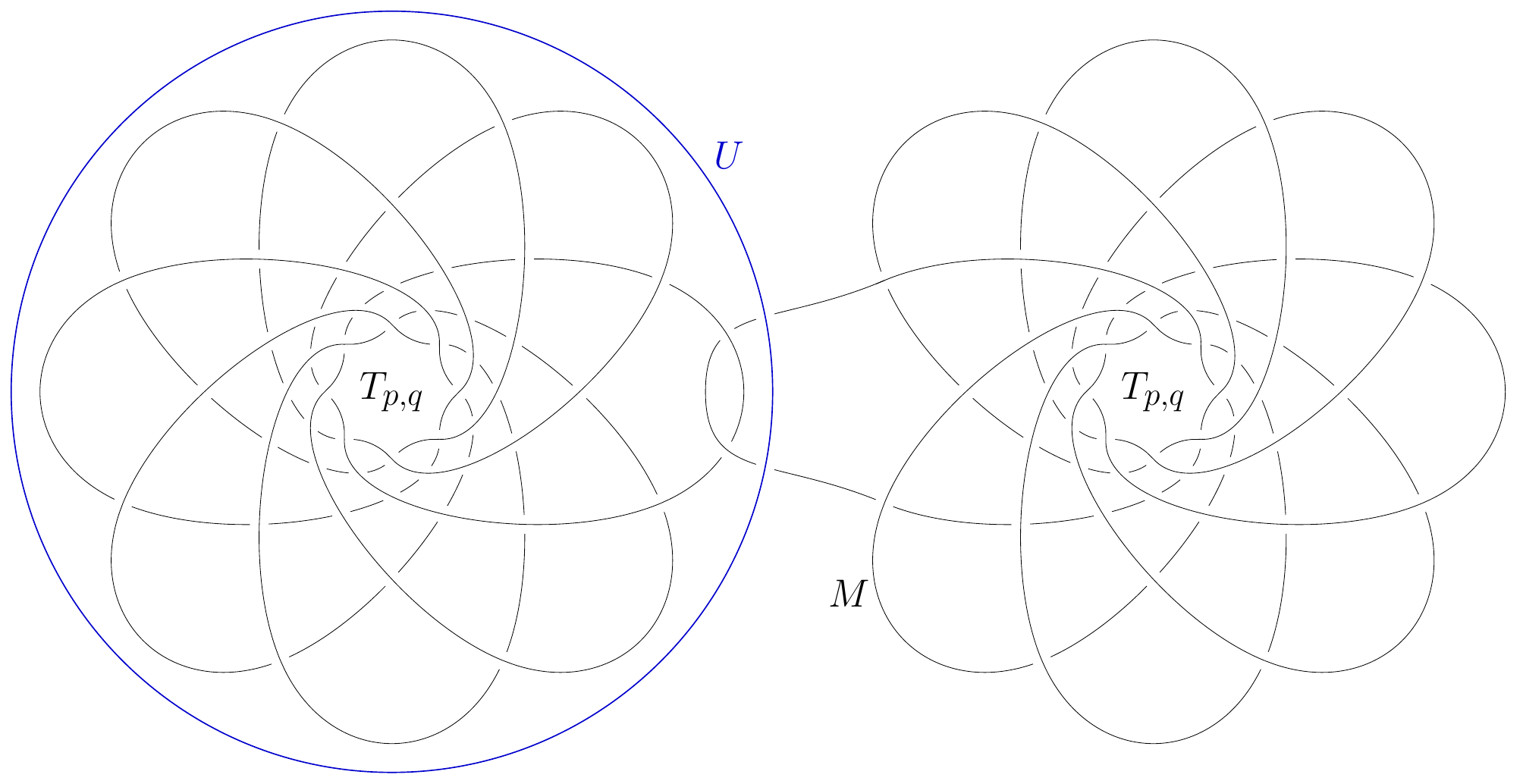}
\caption{The link diagram $\mathcal{D} (p,q)$, $(p,q) = (7,8)$: two linked torus knots $T_{7,8}$ and an unknot $U$.}
\label{pic_def_diag}
\end{center}
\end{figure}

This is to our knowledge the first construction of split links with super-constant crossing-complexity, and we cannot rely on exhaustive search methods to prove \Cref{th_unlink_analysed}. Instead, we develop new techniques to lower-bound the crossing-complexity. It was proved by Dynnikov~\cite[Theorem~2]{Dynnikov_arc} that $\unlink(D)= \mathcal{O}(|D|^2)$, where $|D|$ denotes the number of crossings in a diagram $D$, so there is still a significant gap between this upper bound and the $\Omega(\sqrt{|D|})$ lower bound provided by our theorem.

\begin{remark}
\label{rem:caution}
\Cref{th_unlink_analysed} is one of the many knot-theoretical statements that seems intuitively clear but is surprisingly delicate to prove: inspecting \Cref{pic_def_diag}, ``clearly'' the only way to split $U$ from the two components of $M$ is through an overlay of $U$ on top of one of these components, yielding the claimed increase in the number of crossings. In fact, the following idea to prove \Cref{th_unlink_analysed} may seem straightforward: 
if there is a sequence of Reidemeister moves splitting $U$ from $M$ without adding too many crossings, then we can find a continuous family of planes $(P_t)_{t\in \{0,1\}}$ in $\mathbb{R}^3$ sweeping one of the torus knots in $M$ such that none of the planes $P_t$ intersects $M$ in too many points. Known results on \emph{bridge position}~\cite{schubert1954numerische,schultens2007bridge} and \emph{thin position}~\cite{gabai1987foliations,scharlemann2005thin} of torus knots then imply that such a family cannot exist\footnote{This argument is the basis of the proof in~\cite{Mesmay_treewidth} showing that torus knot diagrams have high tree-width.}. 

However, this proof idea has two issues. First, during a sequence of Reidemeister moves, the unknot $U$ might come to intersect itself, making it difficult to define a continuous family of planes that mimics its intersections with $M$.  Second, the Reidemeister moves might lead $U$ to move back and forth, leading to a non-monotone behaviour of the planes $P_t$, which is not allowed in the aforementioned positions commonly used in knot theory. This leads us to rely on more advanced tools to prove \Cref{th_unlink_analysed}. 
\end{remark}

\subparagraph*{Overview of the proof.} The proof of \Cref{th_unlink_analysed} works by contradiction. That is, we start by assuming that there exists a sequence of Reidemeister moves $R$ transforming $\mathcal{D}(n,n+1)$ into a split diagram and such that $\addcross(\mathcal{D}(n,n+1),R)$ remains small. This implies that, throughout performing $R$, the number of crossings involving $U$ and $M$ always remains small. 

The main tool that we rely on is the framework of bubble tangles introduced by the first and second authors in~\cite{Lunel_spherewidth_out}. More precisely, we use the evolution of $U$ throughout $R$ to define a collection of $2$-dimensional spheres that continuously sweep $\Sp^3$ (a \textit{sweep-out}) and that -- according to our assumptions -- all have a small number of intersections with $M$. This sweep-out resembles the sphere decompositions of~\cite{Lunel_spherewidth_out}, but presents two notable differences. On the one hand it is simpler: it is linear and features no double bubbles. On the other hand, it is not monotone: a sphere involved in this sweep-out may go back-and-forth, while this behaviour is not allowed in the sphere decompositions from \cite{Lunel_spherewidth_out}. Despite this last difference, the bubble tangles, which are obstructions to thin sphere decompositions developed in \cite{Lunel_spherewidth_out}, are versatile enough to show that the existence of this sweep-out leads to a contradiction.

As mentioned in \Cref{rem:caution}, building the sweep-out is not straightforward. Intuitively, one would like to lift each unknot $U$ to a sphere by capping it off above and below the diagram. However, the projection of the unknot $U$ may intersect itself during the sequence of Reidemeister moves, which complicates the process.
We alleviate this problem using results and methods from Chambers and Liokumovich~\cite{Chambers_homo-isotopy} to transform homotopies of curves on a Riemannian surface into isotopies of similar length. The connection is the following: we think of the projection of the link $M$ as a discrete metric for curves in the projection plane. In this discrete model, the \textbf{length} of a curve is given by its number of intersections with $M$, similarly to the cross-metric model commonly used in computational topology of surfaces (see for example~\cite{de2010tightening}). Our assumptions imply that there exists a homotopy of the projection of $U$ where intermediate curves all have small length. The techniques of \cite{Chambers_homo-isotopy} show that this implies that there also exists an isotopy of the projection of $U$ with the same bounds on lengths. Since such an isotopy must consist of simple curves, we can then lift it into a sweep-out of $\Sp^3$ with $2$-spheres, which all have a controlled number of intersections with $M$. However, a subtlety is that we cannot use the results of \cite{Chambers_homo-isotopy} out of the box. Indeed, the discrete metric defined by the link $M$ is not fixed but evolves with the Reidemeister moves in $R$. In \Cref{sec_homo_iso} we explain why the proof techniques in \cite{Chambers_homo-isotopy} can be adapted to deal with this issue.

The tools of the first and second authors~\cite{Lunel_spherewidth_out} are then used to find a contradiction between the existence of the sweep-out and the topological properties of our link $M$. Namely, $L(n,n+1)$ consists of two torus knots $T_{n,n+1}$, which are embedded on tori in a specific way (they both have high \emph{compression-representativity}). Hence, they can be used to define two \textbf{bubble tangles} using \cite[Theorem~1.2]{Lunel_spherewidth_out} (In a nutshell, a bubble tangle is a way of choosing a ``small side'' for each sphere of the sweep-out, where the intuition is that the small side should be easy to sweep, see \Cref{sec_prelim} for a formal definition). We then prove that, since in the initial diagram $U$ lies in-between the two tori, different small sides are chosen for the corresponding sphere by each of the two bubble tangles. But the existence of the sweep-out with a small number of intersections with $M$ forces the small sides of both bubble tangles to agree, leading to a contradiction.

\subparagraph*{Related work.}
The splitting problem has been studied several times as a useful and easier problem for understanding the unknot recognition problem~\cite{Dynnikov_arc, Lackenby_Poly_bound}. In 1961, Haken used normal surface theory to show that it is decidable \cite{haken_normal_surface}. Later, it was determined to be in \textbf{NP}~\cite{Hass_trivial_NP} and also in co-\textbf{NP}~\cite[Theorem~1.6]{Lackenby_co-NP}. 

Several decision problems related to the splitting problem have been studied as well. For instance, the problem of deciding whether changing at most $k$ crossings can transform a link diagram into the diagram of a split link is known to be \textbf{NP}-hard \cite{Koenig_hardness}, and Lackenby provided an algorithm to detect links which can be split using exactly one crossing change~\cite{lackenby2021links}. 

Another natural question is to ask for the minimal number of Reidemeister moves needed to split a diagram. An exponential bound for this number was first provided by~\cite{Hayashi_upper_unlink}, and this bound was later greatly improved by Lackenby in~\cite{Lackenby_Poly_bound}, where he provided a polynomial bound using a combination of normal surface theory~\cite{Hass_trivial_NP} and Dynnikov's work on grid diagrams~\cite{Dynnikov_arc}. There is a quadratic lower bound on the number of moves needed to untangle a specific unknot diagram in~\cite{hass2010unknot}, and it was shown in~\cite{unbearable_SoCG, unbearable} that finding the shortest sequence of Reidemeister moves to untangle an unknot is \textbf{NP}-hard.

\subparagraph*{Organisation of this paper.}
After going through our setup in \Cref{sec_prelim}, we explain how to use the results of \cite{Chambers_homo-isotopy} in \Cref{sec_homo_iso} while providing a detailed proof in \Cref{sec_chambers_appendix}. This step is crucial for our definition of sweep-outs. Then, we exploit obstructions from \cite{Lunel_spherewidth_out} to prove \Cref{th_unlink_analysed} in \Cref{sec_tangle}.

\section{Setup and definitions}\label{sec_prelim}

\subparagraph*{Knots and links.} Many concepts from this paper come from knot theory: while we strive to be as self-contained as possible, we refer to standard textbooks~\cite{burde2002knots,Rolfsen_Knots} for an introduction to this topic and to Hatcher~\cite{Hatcher_Algebraic_Topology} for background on algebraic topology. Throughout this article, we work in the piecewise-linear (PL) category, which means that all the objects and functions that we consider are piecewise-linear with respect to a fixed polyhedral decomposition of the ambient space (generally $\Sp^3$). Two embeddings $i_1$ and $i_2$ in a topological space $S$ are \textbf{(ambient) isotopic} if there exists a continuous family of homeomorphisms $h: S \times[0,1] \rightarrow S$ such that $h (i_1,0) =i_2$ and $h(\cdot,1)$ is the identity. A \textbf{knot}, resp. a \textbf{link}, is an embedding of the circle $\Sp^1$, resp. of a disjoint union of circles, into $\Sp^3$. In the following, we introduce the main definitions for knots for simplicity, but they apply identically to links. Two knots $K_1$ and $K_2$ are considered to be equivalent if they are isotopic. Since every knot misses at least one point of $\Sp^3$, via stereographic projection we can equivalently consider knots to be embedded in $\mathbb{R}^3$, and we freely switch between these two perspectives. The \textbf{unknot} is, up to equivalence, the unique embedding of $S^1$ in $\mathbb{R}^3$ with image a triangle. A \textbf{torus knot} $T_{p,q}$ is a knot embedded on a surface of an unknotted torus $\To$ in $\mathbb{S}^3$, for example a standard torus of revolution. It winds $p$ times around the revolution axis, and $q$ times around the core of the torus. We refer to Figure~\ref{pic_def_diag} for an illustration of two torus knots and an unknot.

A \textbf{knot diagram} $\mathcal{D}$ is a plane four-regular graph (i.e., a planar graph with an explicit embedding in the plane), where each vertex, i.e., crossing of the knot, is decorated to indicate which strands are above and below. From such a diagram, one can easily obtain the data of a knot $K \hookrightarrow \mathbb{R}^3$ and a linear \textbf{projection} map $p: \mathbb{R}^3 \rightarrow P \simeq \mathbb{R}^2$ so that $p(K)=\mathcal{D}$ (respecting the decorations), where $P$ is a plane of $\mathbb{R}^3$ called the \textbf{projection plane}. The \textbf{crossing number} of a knot is the minimal number of crossings among all of its diagrams. The Reidemeister theorem~\cite{Reidemeister_reidemeister_th} shows that two diagrams represent equivalent knots if and only if they can be connected by a sequence of planar isotopies and local moves called \textbf{Reidemeister moves}, which are pictured in \Cref{pic_Reide_move}. 

Finally, recall that a \textbf{homotopy} between two simple closed curves of a surface $\Sigma$, $\gamma_0:\mathbb{S}^1 \hookrightarrow \Sigma$ and $\gamma_1:\mathbb{S}^1 \hookrightarrow \Sigma$ is a continuous map $\gamma:\mathbb{S}^1 \times [0,1] \rightarrow \Sigma$ such that $\gamma (\mathbb{S}^1,0)=\gamma_0$ and $\gamma (\mathbb{S}^1,1)=\gamma_1$. In particular, closed curves are allowed to self-intersect in a homotopy, while this is disallowed in an isotopy. Throughout this paper, to distinguish between links and their projections, we use calligraphic letters to designate diagrams, and capital letters to designate links so that a diagram of the link $M$ is denoted by $\mathcal{M}$.

\begin{remark}
We work with the standard setting of knot diagrams in $\mathbb{R}^2$, but it is also possible to work with diagrams in $\Sp^2$, which therefore allow for more Reidemeister moves (this is the perspective taken in~\cite{Mesmay_hard_unknot}). Our results also hold in that setting, since a knot diagram in $\Sp^2$ lifts to a knot $K \hookrightarrow \Sp^2 \times [-\varepsilon,\varepsilon] \subseteq \mathbb{R}^3$, with the natural projection map $p:\Sp^2 \times [-\varepsilon,\varepsilon], (s,t) \mapsto (s,0)$. The definitions of the spheres obtained from the diagrams in \Cref{sec_tangle} can be directly adapted to this setting, and the rest of the proof is identical.
\end{remark}

\subparagraph*{Our link diagrams.} Throughout this article, we write $L(p,q)$ for the split link consisting of two linked torus knots $T_{p,q}$, denoted by $M(p,q)$ and shown in \Cref{pic_def_diag}, and an unlinked unknot component $U$. We consider two diagrams of $L(p,q)$. The first, denoted by $\mathcal{D}(p,q)$, is shown in \Cref{pic_def_diag}, and the second diagram $\mathcal{D}'(p,q)$ is any split link diagram of $L(p,q)$.

Let $R$ be a sequence of Reidemeister moves turning $\mathcal{D}(p,q)$ into $\mathcal{D}'(p,q)$, such that we have $\addcross(\mathcal{D}(p,q),R) \leq k$ for $k \geq 0$. Our goal is to prove that $k$ cannot be smaller than a function depending only on $p$ and $q$. Since the values of $p$ and $q$ are mostly fixed and have little influence on our arguments, we mostly omit the parameters $(p,q)$ from $L$, $M$, $\mathcal{D}$, and $\mathcal{D}'$.

The following lemma directly follows from known results on torus knots.

\begin{lemma}\label{lem_min_cross_diag}
Let $n\geq 2$, and let $\mathcal{D}'$ be a link diagram equivalent to $\mathcal{D} (n,n+1)$ by Reidemeister moves. Then $\cross(\mathcal{D}') \geq 2n^2$, that is, $\mathcal{D}'$ has at least $2 n^2$ crossings.
\end{lemma}

\begin{proof}
Let $\mathcal{D}'$ be a diagram of $L=L(n,n+1)$. First note that, since $n\geq 2$, it follows from a theorem of Murasugi~\cite[Proposition 7.5]{Murasugi_crossings_torus} that each of the torus knot components $T_{n,n+1}$ of $L$ has at least $(n+1) \times (n-1) = n^2 -1$ crossings. Since the two torus knots are linked, they share at least $2$ crossings in each diagram of $L$, and it follows that $\cross(\mathcal{D}) \geq 2n^2$.
\end{proof}

\subparagraph*{From a sequence of Reidemeister moves to continuous operations.} As detailed above, we work with a sequence of Reidemeister moves $R$ turning the link diagram $\mathcal{D}$ of $L$ into the split diagram $\mathcal{D}'$. From this sequence of Reidemeister moves, we can obtain an ambient isotopy $\Phi_R : \R^3 \times [0,1] \to \R^3$ and a projection $p: \R^3 \rightarrow \R^2$ so that the diagrams $p(\Phi_R(L,t))$ follow the evolution of $\mathcal{D}$ under the moves $R$, and in particular the combinatorial types of the diagrams $p(\Phi_R(L,t))$ only change for a finite number of values of $t$, one for each Reidemeister move. The projection $p$ is regular except at the \textbf{critical times} of $R$, which are times where the projection $p \circ \Phi_R(L,t)$ displays a tangency or a triple point, see \Cref{pic_def_critic} for an illustration. 
For any non-critical time $t$, we write $L_t = \Phi_R(L,t)$ and we denote the diagram defined by $p(L_t) = p(\Phi_R(L,t))$ by $\mathcal{D}_t$. Naturally, we have $\mathcal{D}_0 = \mathcal{D}$ and $\mathcal{D}_1 = \mathcal{D}'$.

\begin{figure}[ht]
\begin{center}
\begin{tikzpicture}
\begin{scope}[very thick]

\begin{scope}[xshift=2.25cm]

\node at  ($(1.25,1)!0.5!(2.25,1)$) {RII};

\draw (2.5,0) .. controls +(90:0.75) and \control{(3,1)}{(-90:0.3)} .. controls +(90:0.3) and \control{(2.5,2)}{(-90:0.75)}; 
\draw (3.5,0) .. controls +(90:0.75) and \control{(3,1)}{(-90:0.3)} .. controls +(90:0.3) and \control{(3.5,2)}{(-90:0.75)}; 
\end{scope}

\begin{scope}[xshift=4.75cm]

\node at ($(2.25,1)!0.5!(3.25,1)$) {RIII};

\draw (3.5,0) -- ++(2,2);
\draw (5.5,0) -- ++(-2,2);

\draw (3.5,1) -- ++(2,0);\
\end{scope}
\end{scope}
\end{tikzpicture}
\caption{Critical times corresponding to Reidemeister moves RII (left) and RIII (right).}
\label{pic_def_critic}
\end{center}
\end{figure}
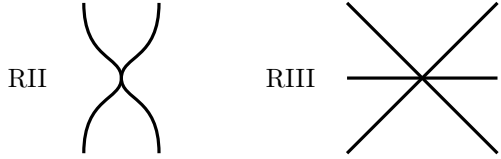

Note that the definition of $\addcross(\mathcal{D},R)$ naturally coincides with $\sup_{t\in [0,1]} \{\cross(\mathcal{D}_t) - \cross(\mathcal{D}_0)\} $. Indeed, the diagram $\mathcal{D}_t$ at critical times has fewer intersections than one of $\mathcal{D}_{t+ \epsilon }$ or $\mathcal{D}_{t- \epsilon }$ for $\epsilon$ small enough. Furthermore, $\cross(\mathcal{D}_t)$ is constant for all $t$ between two critical times.

In \Cref{sec_homo_iso} we consider the movements of $U$ and of $M$ under the Reidemeister moves separately. 
We use $\mathcal{M}_t = p(\Phi_R(M,t))$ as a shorthand for the diagram of the sublink $M\subset L$ at time $t$. For $\mathcal{U}$ we consider the homotopy $\phi_U: \mathbb{S}^1 \times [0,1] \rightarrow \R^2$ in the plane induced by the projection $p(\Phi_R(U,t))$. We denote the corresponding curves by $\mathcal{U}_t=\phi_U(\Sp^1,t)$ and emphasise that we consider these as immersions of closed curves in the plane. That is, we forget which strand is over which at each self-crossing of $\mathcal{U}_t$. See \Cref{pic_commut_op_def} for a summary of this setup.

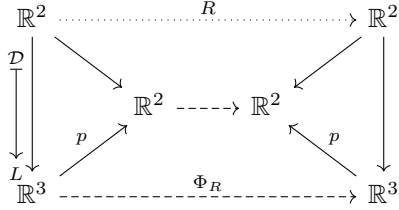
\begin{figure}[ht]
\begin{center}
\begin{tikzcd}
\R^2  \arrow[rrr, "R", dotted] \arrow[dr] \tikzmark{s}{\mathcal{D}}{-12pt}{-14pt}{} \tikzmark{s}{L}{-12pt}{-58pt} \tikz[overlay,remember picture,baseline] \draw [|->] (-0.42,-0.56) -- ++(0,-1.2); & & & \R^2 \arrow[dl, swap] \\
& \R^2 \arrow[r, dashed] & \R^2 & \\
\R^3 \arrow[rrr, "\Phi_R", dashed] \arrow[uu, <-] \arrow[ur, "p"]   & & & \R^3 \arrow[uu, <-, swap ] \arrow[ul, "p", swap]
\end{tikzcd}
\caption{Definition of our homotopies from the sequence of Reidemeister moves $R$.}
\label{pic_commut_op_def}
\end{center}
\end{figure}

\section{From homotopies to isotopies}\label{sec_homo_iso}

We work with the definitions of \Cref{sec_prelim}. We start with a sequence of Reidemeister moves $R$ turning $\mathcal{D}$ into $\mathcal{D}'$, and consider the induced homotopy $\phi_U$ of the link component $\mathcal{U}$ in the plane of projection. The goal of this section is to use results from~\cite{Chambers_homo-isotopy} in order to locally alter the images $\mathcal{U}_t=\phi_U(\Sp^1,t)$ into simple closed curves, and hence to obtain an isotopy taking $\mathcal{U}_0$ to $\mathcal{U}_1$. 
After these modifications, the altered simple closed curves representing $U$ still sweep over the remainder $\mathcal{M}_t$ of the diagrams $\mathcal{D}_t$, but this sweep no longer necessarily corresponds to a sequence of Reidemeister moves on the original link $L$.

Nevertheless, in the next section we use the isotopy from $\mathcal{U}_0$ to $\mathcal{U}_1$ to define a family of $2$-spheres in $\Sp^3$ sweeping through $M$. This setup then allows us to use bubble tangles to obstruct small values of $\addcross(\mathcal{D},R)$ in the initial sequence of Reidemeister moves $R$.

\medskip

The key points of this process are given by the following statement.

\begin{proposition}\label{prop_iso_U}
Let $R$ be a sequence of Reidemeister moves turning $\mathcal{D}$ into the split diagram $\mathcal{D}'$ such that for all $t \in [0,1]$,  $\cross(\mathcal{D}_t)\leq m$ for some integer $m \geq 0$. Then there exists an ambient isotopy $\Phi' : \Sp^3 \times [0,1] \to \Sp^3$ and an isotopy $h:\Sp^1 \times [0,1] \rightarrow \R^2$ such that

\begin{enumerate}
    \item $h(\Sp^1,0)=\mathcal{U}_0$ and $h(\Sp^1,1)=\mathcal{U}_1$; and
    \item for all $t \in [0,1]$, the total number of crossings in the overlay of $p(\Phi'(M,t))$ and $h(\Sp^1,t)$ in $\R^2$ is at most $m$.
\end{enumerate}
\end{proposition}

We emphasise that in the second item of \Cref{prop_iso_U}, we only consider the sublink $M$ in $\Phi'(M,t)$, and not the entire link. That is, the proposition provides an ambient isotopy for $M$ in $\R^3$ and an isotopy for $p(U)$ in $\R^2$, while preserving a bound on the total number of intersections of $p(\Phi'(M,t))$ and $h(\Sp^1,t)$ in the plane of projection. This proposition follows from the techniques of Chambers and Liokumovich developed in Chapter $2$ of \cite{Chambers_homo-isotopy}, the proof is detailed in \Cref{sec_chambers_appendix} for completeness. We first state one of their main results. For $\alpha$ a curve on a surface, we denote by $-\alpha$ the same curve with the opposite orientation.

\begin{definition}[Chambers, Liokumovich {\cite[Definition 1.3]{Chambers_homo-isotopy}}]
\label{def:Chambers_homo-isotopy}
Two curves $\alpha$ and $\beta$ on a Riemannian $2$-manifold $M$ are \textbf{$\epsilon$-image equivalent}, $\alpha \sim_\epsilon \beta$, if (i) there exists a finite collection of disjoint intervals $\bigsqcup_{i=1}^{n} I_i \subset \Sp^1$ such that $\|\alpha (\Sp^1 \smallsetminus \bigsqcup I_i)\| + \| \beta (\Sp^1 \smallsetminus \bigsqcup I_i) \| < \epsilon$, and (ii) there exists a permutation $\sigma$ of $\{1, \ldots, n\}$ and a map $f : \{1, \ldots, n\} \rightarrow \{0,1\}$, such that $\alpha |_{I_i} = {(-1)}^{f(i)} \beta|_{I_{\sigma_{(i)}}}$ for all $i$. Here, $\| \alpha \|$ denotes the length of the curve $\alpha$.
\end{definition}

The main result we use for the proof of \Cref{prop_iso_U} is given by the following statement.

\begin{theorem}[Chambers, Liokumovich {\cite[Theorem 1.1']{Chambers_homo-isotopy}}]\label{Chambers_homo-isotopy}
Suppose that $\gamma$ is a smooth homotopy of closed curves on a Riemannian $2$-manifold $M$ and that $\gamma_0$ is a simple closed curve. Then, for every $\epsilon > 0$, there exists an isotopy $\overline{\gamma}$ such that $\overline{\gamma_0} = \gamma_0$ and $\overline{\gamma_1}$ is $\epsilon$-image equivalent to a small perturbation of $\gamma_1$. Additionally, for every $t$, there exists a $t'$ such that $\overline{\gamma_t}$ is $\epsilon$-image equivalent to a small perturbation of $\gamma_{t'}$. If $\gamma_1$ is simple, or is a point, then this homotopy also ends at $\gamma_1$, up to a change in orientation.
\end{theorem}

Informally, we think of \Cref{def:Chambers_homo-isotopy} and \Cref{Chambers_homo-isotopy} as follows. When a curve $\alpha$ self-intersects in the plane, each crossing point can be \textbf{resolved} in one of two ways by reconnecting the endpoints in a small ball around the crossing point, see \Cref{pic_resolution} on the right. When all the crossings have been resolved in some way and we obtain a simple closed curve $\alpha'$, we say that $\alpha'$ is a \textbf{resolution} of $\alpha$. The theorem states that if we have a homotopy $\gamma$ on a surface between two simple curves $\gamma_0$ and $\gamma_1$, one can obtain an \emph{isotopy} $\overline{\gamma}$ between $\gamma_0$ and $\gamma_1$\footnote{Or $\gamma_1$ with its orientation reversed. Since it is safe to disregard orientations for our purpose, we always consider curves up to orientation reversal.} where each intermediate curve $\overline{\gamma_t}$ is a resolution of some intermediate curve $\gamma_{t'}$. Note that, here, the times $t$ and $t'$ may not coincide. 

If $M$ is endowed with a metric (for example, a Riemannian one), the lengths of $\gamma_t$ and of $\overline{\gamma_{t'}}$ differ by an arbitrarily small quantity. Hence, \Cref{Chambers_homo-isotopy} immediately implies that, for all $\varepsilon>0$, if there exists a homotopy between two simple curves $\gamma_0$ and $\gamma_1$ where each intermediate curve has length at most $\ell$, then there also exists an isotopy between $\gamma_0$ and $\gamma_1$ where each intermediate curve has length at most $\ell+\varepsilon$.

\begin{figure}[ht]
\begin{center}
\begin{tikzpicture}
\def\ec{1.5}
\clip (-0.8,-2.6) rectangle (11.8,3);

\begin{scope}[xshift = 0cm]
\coordinate (a) at (1,2);
\coordinate	(b) at (0,1);
\coordinate	(c) at (1,0);
\coordinate	(d) at (2,1);
\coordinate	(e) at (0.5,-1);
\coordinate	(f) at (1.5,-1);
\node [white,circle, inner sep = \ec pt] (na) at (a) {};
\node [white,circle, inner sep = \ec pt] (nb) at (b) {};
\node [white,circle, inner sep = \ec pt] (nc) at (c) {};
\node [white,circle, inner sep = \ec pt] (nd) at (d) {};
\node [white,circle, inner sep = \ec pt] (ne) at (e) {};
\node [white,circle, inner sep = \ec pt] (nf) at (f) {};

\begin{scope}[thick]
\draw (na.south east) -- (na.north west);
\draw (nb.south west) -- (nb.north east);
\draw (nc.south west) -- (nc.north east);
\draw (nd.north west) -- (nd.south east);
\draw (ne.south) -- (ne.north);
\draw (nf.east) -- (nf.west);

\draw (na.south west) -- (nb.north east);
\draw (nb.south east) -- (nc.north west);
\draw (nc.north east) -- (nd.south west);
\draw (nd.north west) -- (na.south east);
\draw (ne.east) -- (nf.west);
\draw (nc.south west) .. controls +(-0.4,-0.4) and \control{(ne.north)}{(0,0.2)};
\draw (nc.south east) .. controls +(0.4,-0.4) and \control{(nf.north)}{(0,0.2)};
\draw (na.north east) .. controls +(1.4,1.4) and \control{(nd.north east)}{(1.4,1.4)};
\draw (na.north west) .. controls +(-1.4,1.4) and \control{(nb.north west)}{(-1.4,1.4)};
\draw (ne.south) .. controls +(0,-1) and \control{(nf.south)}{(0,-1)};
\draw (nb.south west) .. controls +(-1,-1) and \control{(ne.west)}{(-0.75,0)};
\draw (nd.south east) .. controls +(1,-1) and \control{(nf.east)}{(0.75,0)};
\end{scope}
\end{scope}

\begin{scope}[xshift = 4.5cm]
\coordinate (a) at (1,2);
\coordinate	(b) at (0,1);
\coordinate	(c) at (1,0);
\coordinate	(d) at (2,1);
\coordinate	(e) at (0.5,-1);
\coordinate	(f) at (1.5,-1);
\node [white,circle, inner sep = \ec pt] (na) at (a) {};
\node [white,circle, inner sep = \ec pt] (nb) at (b) {};
\node [white,circle, inner sep = \ec pt] (nc) at (c) {};
\node [white,circle, inner sep = \ec pt] (nd) at (d) {};
\node [white,circle, inner sep = \ec pt] (ne) at (e) {};
\node [white,circle, inner sep = \ec pt] (nf) at (f) {};

\begin{scope}[thick]
\draw (na.south east) -- (na.north west);
\draw (na.south west) -- (na.north east);
\draw (nb.south west) -- (nb.north east);
\draw (nb.south east) -- (nb.north west);
\draw (nc.south east) -- (nc.north west);
\draw (nc.south west) -- (nc.north east);
\draw (nd.south west) -- (nd.north east);
\draw (nd.south east) -- (nd.north west);
\draw (ne.south) -- (ne.north);
\draw (ne.west) -- (ne.east);
\draw (nf.east) -- (nf.west);
\draw (nf.north) -- (nf.south);

\draw (na.south west) -- (nb.north east);
\draw (nb.south east) -- (nc.north west);
\draw (nc.north east) -- (nd.south west);
\draw (nd.north west) -- (na.south east);
\draw (ne.east) -- (nf.west);
\draw (nc.south west) .. controls +(-0.4,-0.4) and \control{(ne.north)}{(0,0.2)};
\draw (nc.south east) .. controls +(0.4,-0.4) and \control{(nf.north)}{(0,0.2)};
\draw (na.north east) .. controls +(1.4,1.4) and \control{(nd.north east)}{(1.4,1.4)};
\draw (na.north west) .. controls +(-1.4,1.4) and \control{(nb.north west)}{(-1.4,1.4)};
\draw (ne.south) .. controls +(0,-1) and \control{(nf.south)}{(0,-1)};
\draw (nb.south west) .. controls +(-1,-1) and \control{(ne.west)}{(-0.75,0)};
\draw (nd.south east) .. controls +(1,-1) and \control{(nf.east)}{(0.75,0)};
\end{scope}
\end{scope}

\begin{scope}[xshift = 9cm]
\def\ec{4}
\coordinate (a) at (1,2);
\coordinate	(b) at (0,1);
\coordinate	(c) at (1,0);
\coordinate	(d) at (2,1);
\coordinate	(e) at (0.5,-1);
\coordinate	(f) at (1.5,-1);
\node [white,circle, inner sep = \ec pt] (na) at (a) {};
\node [white,circle, inner sep = \ec pt] (nb) at (b) {};
\node [white,circle, inner sep = \ec pt] (nc) at (c) {};
\node [white,circle, inner sep = \ec pt] (nd) at (d) {};
\node [white,circle, inner sep = \ec pt] (ne) at (e) {};
\node [white,circle, inner sep = \ec pt] (nf) at (f) {};

\begin{scope}[thick]
\draw (na.south east) .. controls (na) and (na) .. (na.south west);
\draw (na.north east) .. controls (na) and (na) .. (na.north west);
\draw (nb.south east) .. controls (nb) and (nb) .. (nb.north east);
\draw (nb.south west) .. controls (nb) and (nb) .. (nb.north west);
\draw (nc.south east) .. controls (nc) and (nc) .. (nc.south west);
\draw (nc.north east) .. controls (nc) and (nc) .. (nc.north west);
\draw (nd.south east) .. controls (nd) and (nd) .. (nd.south west);
\draw (nd.north east) .. controls (nd) and (nd) .. (nd.north west);
\draw (ne.east) .. controls (ne) and (ne) .. (ne.south);
\draw (ne.north) .. controls (ne) and (ne) .. (ne.west);
\draw (nf.east) .. controls (nf) and (nf) .. (nf.south);
\draw (nf.north) .. controls (nf) and (nf) .. (nf.west);

\draw (na.south west) -- (nb.north east);
\draw (nb.south east) -- (nc.north west);
\draw (nc.north east) -- (nd.south west);
\draw (nd.north west) -- (na.south east);
\draw (ne.east) -- (nf.west);
\draw (nc.south west) .. controls +(-0.4,-0.4) and \control{(ne.north)}{(0,0.01)};
\draw (nc.south east) .. controls +(0.4,-0.4) and \control{(nf.north)}{(0,0.01)};
\draw (na.north east) .. controls +(1.3,1.3) and \control{(nd.north east)}{(1.3,1.3)};
\draw (na.north west) .. controls +(-1.3,1.3) and \control{(nb.north west)}{(-1.3,1.3)};
\draw (ne.south) .. controls +(0,-0.85) and \control{(nf.south)}{(0,-0.85)};
\draw (nb.south west) .. controls +(-1,-1) and \control{(ne.west)}{(-0.5,0)};
\draw (nd.south east) .. controls +(1,-1) and \control{(nf.east)}{(0.5,0)};
\end{scope}
\end{scope}
\end{tikzpicture}
\caption{Left: A diagram of the unknot $U$. Middle: A projection $\mathcal{U}_t$ of this unknot, where the crossing information has been forgotten. Right: A resolution of $\mathcal{U}_t$.}
\label{pic_resolution}
\end{center}
\end{figure}
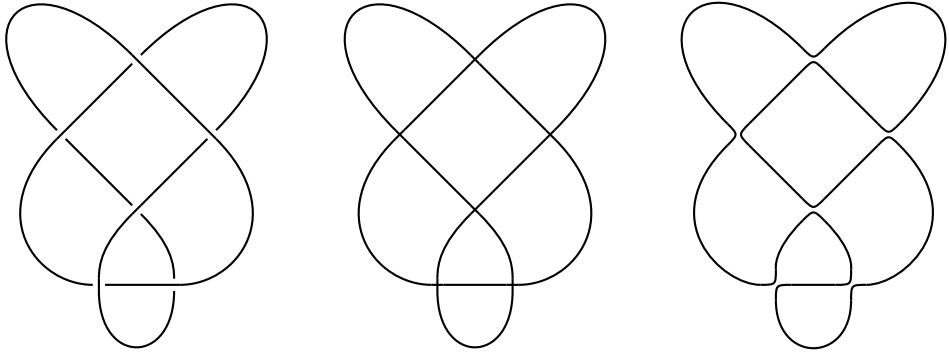

\Cref{Chambers_homo-isotopy} can be applied as a black-box to prove \Cref{prop_iso_U} in the particular case where the sublink $M$ stays invariant throughout the sequence of Reidemeister moves $R$, i.e., $\Phi_R(M,t)=\Phi_R(M,0)$ for all $t \in [0,1]$. Indeed, in this case, we can think of $\mathcal{M}_t$ as a discrete metric measuring the length of a curve $\mathcal{U}_t$ by its number of intersections with $\mathcal{M}_t$. More formally, we can take $\gamma$ to be the homotopy $\phi_U$ between $\mathcal{U}_0$ and $\mathcal{U}_1$, which are both simple closed curves by the definition of the link diagram $\mathcal{D}$. Applying \Cref{Chambers_homo-isotopy} provides us with an isotopy $h$ between $\mathcal{U}_0$ and $\mathcal{U}_1$ where all intermediate curves $\mathcal{U}'_t$ are obtained from resolving intersections of some $\mathcal{U}_{t'}$. In particular, for all $t \in [0,1]$, the number of intersections between $\mathcal{U}'_t$ and $\mathcal{M}_t$ is at most the number of intersections between $\mathcal{U}_{t'}$ and $\mathcal{M}_t$, which is at most $k$ since $\mathcal{M}_t=\mathcal{M}_0$ does not depend on $t$.

A careful reading of~\cite{Chambers_homo-isotopy} shows that the general case of \Cref{prop_iso_U}, where the diagram $\mathcal{M}_t$ of the sublink $M$ evolves during the sequence of Reidemeister moves, can also be obtained using the same proof techniques: The basic idea of the proof of \Cref{Chambers_homo-isotopy} is to first decompose the homotopy of $\gamma$ into a sequence of local moves, and then replace each curve $\gamma_t$ by one of its resolutions.

For the first step, they track discrete times where the self-intersection pattern (i.e., the homeomorphism type) of $\gamma_t$ changes. By an argument similar to the proof of the Reidemeister theorem, one can assume that this only happens at \textbf{critical events}, when a curve $\gamma_t$ undergoes a homotopy move, which is a transformation analogue to a Reidemeister move but without any crossing information. This step is formalised by their Proposition~2.1 and Lemma~2.2. Note that this step is unnecessary for us, since by construction the local transformations undergone by $\phi_U$ are projections of Reidemeister moves.

Between these critical events, any homotopy of a curve $\gamma_t$ can be applied similarly to any of its resolutions, yielding an isotopy. Hence, the idea is to replace each of these Reidemeister moves by a resolution of the move, as explained by their Figure~$2$. However, doing so in a naive way runs into discontinuity issues, a basic example of which is detailed in \cite[Example~2]{Chambers_homo-isotopy}, which is associated to their Figures $3$ and $4$. Therefore, the authors provide a more intricate workaround: the key to the proof of \Cref{Chambers_homo-isotopy} is to show how to choose the correct resolutions and connect their isotopies together. This is achieved by defining an auxiliary graph of resolutions (see their Figure~7), synthesising how they are connected by local isotopies. The precise definition for this graph follows their Figure~8. The proof is then finalised by finding a path through this graph by using the handshaking lemma~\cite{Euler_bridges}. 

In our case, the critical events are exactly the times $t$ when $\mathcal{U}_t$ undergoes a Reidemeister move. In-between these critical events, there are other Reidemeister moves involving either (i) both $\mathcal{U}_t$ and $\mathcal{M}_t$, or (ii) only $\mathcal{M}_t$. Note that in case (i), the Reidemeister move only changes the relative position of $\mathcal{U}_t$ and $\mathcal{M}_t$, and hence such a move can be treated as leaving the isotopy type of $\mathcal{M}_t$ unchanged. Therefore the homotopy of $\mathcal{U}_t$ can be used as an isotopy of any of its resolutions in this case. In case (ii), $\mathcal{M}_t$ changes, but $\mathcal{U}_t$, considered up to isotopy, does not. Therefore, by applying the same Reidemeister moves to $\mathcal{M}_t$, any motion of $\mathcal{U}_t$ between critical events can be applied to any of its resolutions while preserving the number of intersections with $\mathcal{M}_t$. These motions can then be connected using the same handshaking argument as in the proof of \Cref{Chambers_homo-isotopy}. Summarising, the proof technique directly adapts to the case of an evolving metric, as long as these evolutions are applied appropriately throughout the new sequence of isotopies.

\begin{remark}
We emphasise that in this section, we apply the framework of \Cref{Chambers_homo-isotopy} to the \emph{projection} of $U$ without crossing information, which is thus considered as a closed, non-embedded, curve in the projection plane $\mathbb{R}^2$. Instead, it might be tempting to apply these techniques \emph{directly at the level of Reidemeister moves}. This way one may hope to prove that in any sequence of Reidemeister moves splitting $\mathcal{U}$ from $\mathcal{M}$, one can assume that $\mathcal{U}$ remains simple while preserving a bound on the number of added crossings. However, this proof strategy fails because some resolutions applied to $\mathcal{U}$ can block the application of RIII moves involving $\mathcal{U}$ and $\mathcal{M}$. Such a case is pictured in Figure~\ref{pic_reidemeister_issues}$(B)$.
\end{remark}

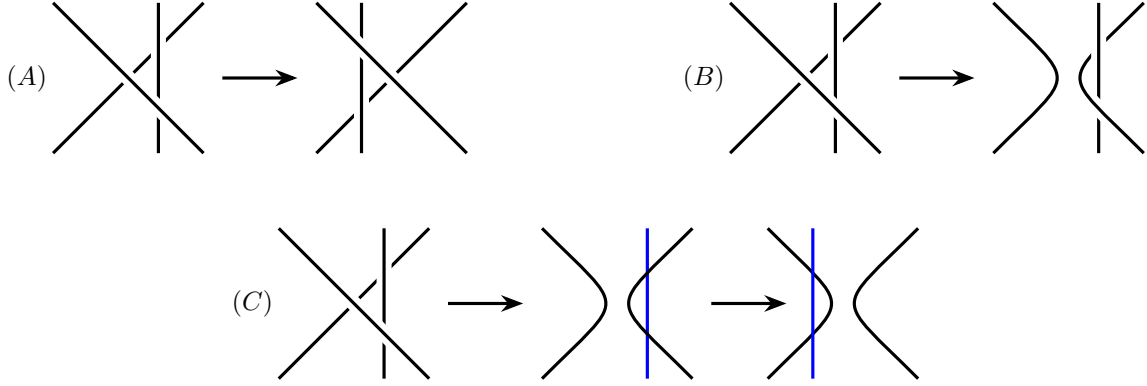
\begin{figure}
\begin{center}
\begin{tikzpicture}[scale=0.9]
\clip (-0.75,-3.1) rectangle (14.6,2.1);
\begin{scope}[very thick]
\draw (0,0) -- ++(2,2);
\draw [line width = 0.2 cm, white] (1.4,0) -- ++(0,2);
\draw (1.4,0) -- ++(0,2);
\draw [line width = 0.2 cm, white] (2,0) -- ++(-2,2);
\draw (2,0) -- ++(-2,2);

\draw [-Stealth] (2.25,1) -- (3.25,1) node [midway, above] {};

\draw (3.5,0) -- ++(2,2);
\draw [line width = 0.2 cm, white] (4.1,0) -- ++(0,2);
\draw (4.1,0) -- ++(0,2);
\draw [line width = 0.2 cm, white] (5.5,0) -- ++(-2,2);
\draw (5.5,0) -- ++(-2,2);

\begin{scope}[xshift=9cm]
\draw (0,0) -- ++(2,2);
\draw [line width = 0.2 cm, white] (1.4,0) -- ++(0,2);
\draw (1.4,0) -- ++(0,2);
\draw [line width = 0.2 cm, white] (2,0) -- ++(-2,2);
\draw (2,0) -- ++(-2,2);

\draw [-Stealth] (2.25,1) -- (3.25,1);

\draw (3.5,0) .. controls +(45:1.6) and \control{(3.5,2)}{(-45:1.6)};
\draw [line width = 0.2 cm, white] (4.9,0) -- ++(0,2);
\draw (4.9,0) -- ++(0,2);
\begin{scope}[even odd rule]
\clip (4.8,1) rectangle (5,2) (3.4,-0.1) rectangle (5.6,2.1);
\draw [line width = 0.2 cm, white] (5.5,0) .. controls +(135:1.6) and \control{(5.5,2)}{(-135:1.6)};
\draw (5.5,0) .. controls +(135:1.6) and \control{(5.5,2)}{(-135:1.6)};
\end{scope}
\end{scope}
\begin{scope}[yshift=-3cm, xshift = 3cm]
\draw (0,0) -- ++(2,2);
\draw [line width = 0.2 cm, white] (1.4,0) -- ++(0,2);
\draw (1.4,0) -- ++(0,2);
\draw [line width = 0.2 cm, white] (2,0) -- ++(-2,2);
\draw (2,0) -- ++(-2,2);

\draw [-Stealth] (2.25,1) -- (3.25,1);

\draw (3.5,0) .. controls +(45:1.6) and \control{(3.5,2)}{(-45:1.6)};
\draw [blue] (4.9,0) -- ++(0,2);
\begin{scope}[even odd rule]
\draw (5.5,0) .. controls +(135:1.6) and \control{(5.5,2)}{(-135:1.6)};
\end{scope}
\draw [-Stealth] (5.75,1) -- (6.75,1);
\begin{scope}[xshift = 3cm]
\draw (3.5,0) .. controls +(45:1.6) and \control{(3.5,2)}{(-45:1.6)};
\draw [blue] (4.1,0) -- ++(0,2);
\begin{scope}[even odd rule]
\draw (5.5,0) .. controls +(135:1.6) and \control{(5.5,2)}{(-135:1.6)};
\end{scope}
\end{scope}
\end{scope}

\node at (-0.35,1) {$(A)$};
\node at (8.65,1) {$(B)$};
\node at (2.65,-2) {$(C)$};
\end{scope}
\end{tikzpicture}
\caption{$(A)$ Reidemeister III move. $(B)$ a problematic resolution of $\mathcal{U}_t$. $(C)$ The unknot $\mathcal{U}_t$ (black) seen as a curve sweeping through $\mathcal{M}$ (blue).}
\label{pic_reidemeister_issues}
\end{center}
\end{figure}

\section{Sweep-outs, bubble tangles and proof of the lower bound}\label{sec_tangle}

For the remainder of this article, we use \Cref{prop_iso_U} to assume that the homotopy $\phi_U : \Sp^1 \times [0,1] \to \R^2$ sweeping $U$ across $M$ is an isotopy.
In other words, let $h$ be the isotopy of $\mathcal{U}$ in $\R^2$ and $\Phi'$ be the ambient isotopy of $M$ in $\Sp^3$ from \Cref{prop_iso_U}, and for all $t \in [0,1]$, we replace $\phi_U (\Sp^1, t)$ by $h (\Sp^1, t) = \mathcal{U}_t$ which is a simple closed curve.

\subsection{A sweep-out of $2$-spheres}
\label{sec:sweepout}

For each $t$, we associate a $2$-sphere $S_t$ to the simple curve $\mathcal{U}_t$, by extending $\mathcal{U}_t$ infinitely towards the direction of the projection $p$. Seen in $\Sp^3$ this produces a torus pinched at $\infty$, and cutting the surface at $\infty$ yields $S_t$ (see \Cref{pic_sphere_from_circle}). By construction, $S_t$ intersects $M_t = \Phi'(M,t)$ only in the pre-images of $\mathcal{U}_t \cap \mathcal{M}_t$ by the projection $p$. Altogether, this produces a continuous family of spheres $S_t$ sweeping a continuous family of links $M_t$. By applying the ambient isotopy $\Phi'^{-1}$ to $S_t$ and $M_t$, we can assume that $M_t$ is fixed. We slightly abuse notation and make this assumption, while still denoting the family of spheres by $S_t$. 

\begin{figure}[h]
\begin{center}
\begin{tikzpicture}[scale = 0.625]
\draw [blue, thick] (-2,2.5) arc (180:360:2 and 0.75);
\draw [blue, dotted] (-2,2.5) arc (180:0:2 and 0.75);
\node [blue, left] at (2,2.5) {$U_t$};
\fill [blue, opacity = 0.2] (-2,0) arc (180:360:2 and 0.75) -- ++(0,5) arc (0:180:2 and 0.75) -- cycle;
\draw [blue, opacity = 0.5] (-2,0) -- ++(0,5);
\draw [blue, opacity = 0.5] (2,0) -- ++(0,5);
\draw [blue, yshift = -2.5cm, opacity = 0.5] (-2,2.5) arc (180:360:2 and 0.75);
\draw [blue, dotted, yshift = -2.5cm, opacity = 0.5] (-2,2.5) arc (180:0:2 and 0.75);
\draw [blue, yshift = 2.5cm, opacity = 0.5] (-2,2.5) arc (180:360:2 and 0.75);
\draw [blue, yshift = 2.5cm, opacity = 0.5] (-2,2.5) arc (180:0:2 and 0.75);
\node [blue, opacity =0.8] at (0,5.25) {$\vdots$};
\node [blue, opacity =0.8] at (0,0) {$\vdots$};

\node at (-2,3) [left] {$\R^3$};
\draw [-Stealth] (2.5, 2.5) -- (3.5, 2.5);

\begin{scope}[xshift = 6cm]
\draw [blue, thick] (-2,2.5) arc (180:360:2 and 0.75);
\draw [blue, dotted] (-2,2.5) arc (180:0:2 and 0.75);
\node [blue, left] at (2,2.5) {$U_t$};
\filldraw [blue,fill opacity = 0.2, draw opacity = 0.5] (-2,2.5).. controls +(90:3) and \control{(4,2.5)}{(90:3)} .. controls + (90:1.5) and \control{(2,2.5)}{(90:1.5)} .. controls +(-90:1.5) and \control{(4,2.5)}{(-90:1.5)} .. controls +(-90:3) and \control{(-2,2.5)}{(-90:3)};
\node at (4,2.5) [left] {\small $\infty$};
\node at (-2.5,3) [left] {$\Sp^3$};
\draw [-Stealth] (4.5, 2.5) -- +(1,0);
\end{scope}

\begin{scope}[xshift = 14cm]
\draw [blue, thick] (-2,2.5) arc (180:360:2 and 0.75);
\draw [blue, dotted] (-2,2.5) arc (180:0:2 and 0.75);
\node [blue, left] at (2,2.5) {$U_t$};
\filldraw [blue,fill opacity = 0.2, draw opacity = 0.5] (-2,2.5).. controls +(90:2.9) and \control{(4,2.6)}{(90:2.9)} .. controls + (90:1.3) and \control{(2,2.5)}{(90:1.3)} .. controls +(-90:1.3) and \control{(4,2.4)}{(-90:1.3)} .. controls +(-90:2.9) and \control{(-2,2.5)}{(-90:2.9)};
\node at (1.75,1) [blue, left] {\small $S_t$};
\end{scope}
\end{tikzpicture}
\caption{Gluing two infinite annuli on $U$ and cutting the resulting pinched torus at $\infty$ in $\Sp^3$.}
\label{pic_sphere_from_circle}
\end{center}
\end{figure}
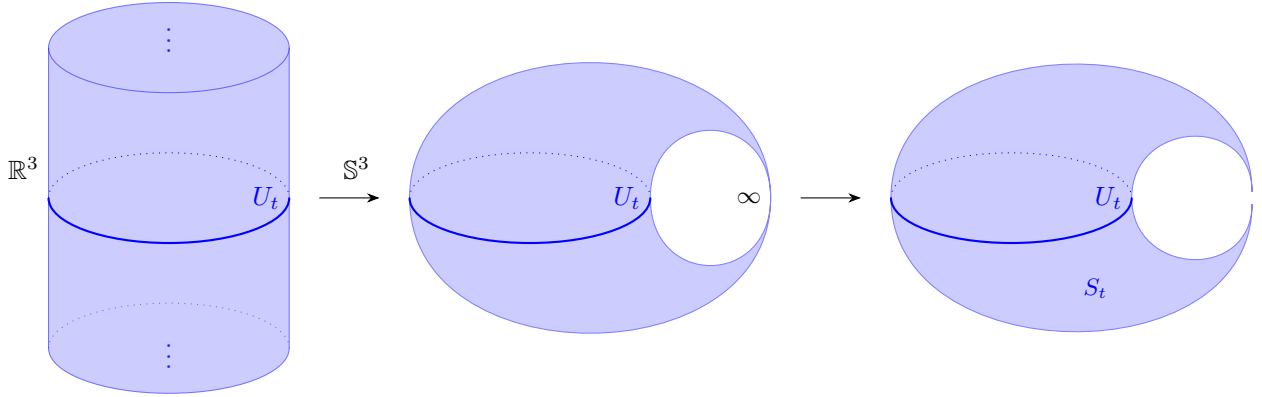

We now relate $\addcross(\mathcal{D}(n,n+1),R)$ to our continuous family of spheres $\{S_t\}_{t\in [0,1]}$.

\begin{lemma}\label{lem_inter_sphere}
We have $\sup_{t \in [0,1]} |S_t \cap M_t| -2 \leq \addcross(\mathcal{D}(n,n+1), R)$.
\end{lemma}

\begin{proof}
For $\mathcal{M}$ and $\mathcal{M}'$ two components of a link diagram, we denote by $\cross(\mathcal{M},\mathcal{M}')$ the number of crossings involving strands of $\mathcal{M}$ and $\mathcal{M}'$. If $\mathcal{M} = \mathcal{M}'$, then $\cross(\mathcal{M},\mathcal{M}')$ is the number of self-crossings of $\mathcal{M}$.

Since $\mathcal{U}_t$ is simple for all times $t$, we have $\cross(\mathcal{D}_t) = \cross(\mathcal{M}_t,\mathcal{U}_t) + \cross(\mathcal{M}_t,\mathcal{M}_t)$. By definition, $\mathcal{D}_t$ is the diagram obtained from $\mathcal{D}(n,n+1)$ at time $t$ of the sequence of Reidemeister moves $R$, and is hence equivalent to $\mathcal{D}(n,n+1)$. It follows from \Cref{lem_min_cross_diag}, that $\cross(\mathcal{M}_t,\mathcal{M}_t) \geq 2n^2$ at all times $t$. Furthermore, $\cross(\mathcal{M}_t,\mathcal{U}_t) = |S_t \cap M_t|$ by construction of $S_t$. By definition of $\mathcal{D}(n,n+1) = \mathcal{D}_0$, we have $\cross(\mathcal{D}_0)= 2n^2 +2$. Hence, $\cross(\mathcal{D}_t) - \cross(\mathcal{D}_0) \geq 2n^2 + |M_t \cap S_t| - 2n^2 -2 $.

Thus, $\sup_{t \in [0,1]} |S_t \cap M_t| -2 \leq \sup_{t \in [0,1]} \{\cross(\mathcal{D}_t) - \cross(\mathcal{D}_0)\} = \addcross(\mathcal{D}(n,n+1), R)$.
\end{proof}

\subsection{Bubble tangles}
\label{sec:bubbletangles}

The aim for the remainder of this section is to use results from \cite{Lunel_spherewidth_out} to exhibit an obstruction to $|S_t \cap M_t|$ remaining small throughout the sweep-out defined in \Cref{sec:sweepout}.

The key objects that we rely on are bubble tangles. Before introducing those, we need the following definitions. A \textbf{double bubble} consists of two spheres intersecting on a disc (see the left of \Cref{pic_double_bub_trivial} for an illustration). The complement of a double bubble consists of three balls, and we say that they are \textbf{induced} by the double bubble. A $3$-ball in $\Sp^3$ is said to be \textbf{$L$-trivial}, if it intersects $L \subset \Sp^3$ in a single unknotted segment or in the empty set, see \Cref{pic_double_bub_trivial} on the middle for an example and on the right for a non-example.

\begin{definition}[\cite{Lunel_spherewidth_out}, Definition 2.3]
\label{def:bubbletangle}
Let $L \in \Sp^3$ be a link and let $k \in \mathbb{N}$. A \textbf{bubble tangle $\T$ of order $k > 2$}, is a collection of closed balls in $\Sp^3$ such that: 
\begin{enumerate}[label = (T\arabic*)]
	\item \label{def_T1} For all balls $B \in \mathcal{T}$, we have $|\partial B \cap L| < k$.
	\item \label{def_T2} For all $2$-spheres $S \in \Sp^3$ transverse to $L$, if $|S \cap L| < k$, then exactly one of the two\footnote{By the PL Schoenflies theorem~\cite[Theorem~XIV.1]{bing}, a PL $2$-sphere in $\Sp^3$ bounds exactly two balls.} balls $B_i$, $i \in \{1, 2\}$ with boundary $S$ is in $\T$. 
	\item \label{def_T3} For all triples of balls $\{ X,Y,Z \}$, if $\{ X,Y,Z \}$ induces a double bubble transverse to $L$, then $\{ X,Y,Z \} \not \subset~\T$.  
	\item \label{def_T4} For every closed ball $B$ in $\Sp^3$, if $B$ is $L$-trivial, then $B \in \T$.
\end{enumerate}
\end{definition}

Informally, we think of a bubble tangle as a way to choose, for each sphere $S$ having less than $k$ intersections with $L$, one of the two balls that it bounds. We refer to this ball as the \textbf{small side} of $S$, with the intuition that this small side intersects $L$ in a simpler pattern than the other side. The essential property of a bubble tangle is the requirement that we cannot cover all of $\Sp^3$ using three small sides of a triple of spheres arranged in a double bubble. One of the main results of~\cite{Lunel_spherewidth_out} states that if a bubble tangle of $L$ of order $k$ exists, any sweep-outs of $\Sp^3$ by $2$-spheres contains at least one sphere with at least $k$ intersections with $L$. As discussed in the introduction, our sweep-outs differ from those in \cite{Lunel_spherewidth_out}. However, we show that despite their differences, a bubble tangle can also be used to preclude the existence of a sweep-out $S_t$ having few intersections with $M$.

\begin{figure}[ht]
\begin{center}
\begin{tikzpicture}
\def\e{0.1cm}
\def\p{1.25}
\draw [red!50!purple, thick] (5,0.75) circle (1.5);
\fill [opacity = 0.05] (5,0.75) circle (1.5);
\draw [red!50!purple] (6.5,0.75) arc (0: -180: 1.5 and 0.6);
\draw [red!50!purple, dotted] (6.5,0.75) arc (0:180: 1.5 and 0.6);
\draw (3.5, 0.75) -- (6.5,0.75);

\begin{scope}[xshift = 9cm, yshift = 0.75 cm]
\coordinate (a) at ($(90:0.35)+(0,0.35)$);
\coordinate (b) at ($(-30:0.35)+(0,0.35)$);
\coordinate (c) at ($(-150:0.35)+(0,0.35)$);
\begin{scope}[even odd rule]
\clip (0,0) circle (1.5) (c) circle (\e);
\draw (-1.5,0) .. controls +(0:0.4) and \control{(c)}{(-120:0.4)} -- (a);
\end{scope}
\begin{scope}[even odd rule]
\clip (0,0) circle (1.5) (b) circle (\e);
\draw (a) .. controls +(60:\p) and \control{(b)}{(0:\p)} -- (c);
\end{scope}
\begin{scope}[even odd rule]
\clip (0,0) circle (1.5) (a) circle (\e);
\draw (1.5,0) .. controls +(0:-0.4) and \control{(b)}{(-60:0.4)} -- (a) .. controls +(120:\p) and \control{(c)}{(180:\p)};
\end{scope} 
\draw [red!50!purple] (0,0) circle (1.5);
\draw [red!50!purple] (1.5,0) arc (0: -180: 1.5 and 0.6);
\draw [red!50!purple, dotted] (1.5,0) arc (0: 180: 1.5 and 0.6);
\fill [opacity = 0.05] (0,0) circle (1.5);
\end{scope}

\begin{scope}[xshift = 1cm, yshift = 0.75 cm]
\filldraw [red!50!purple, fill opacity = 0.2] (140:1.5) arc (140:-140:1.5) arc (-40:-320:1.5);
\fill [red!50!purple, fill opacity = 0.2] (140:1.5) .. controls +(0:0.25) and \control{(-140:1.5)}{(0:0.25)} .. controls +(180:0.25) and \control{(140:1.5)}{(180:0.25)};
\draw [red!50!purple] (140:1.5) .. controls +(0:0.25) and \control{(-140:1.5)}{(0:0.25)};
\draw [red!50!purple, dotted] (140:1.5) .. controls +(180:0.25) and \control{(-140:1.5)}{(180:0.25)};
\end{scope}
\end{tikzpicture}
\caption{Left: a double bubble. Middle: an $L$-trivial ball. Right: a ball that is not $L$-trivial.}
\label{pic_double_bub_trivial}
\end{center}
\end{figure}
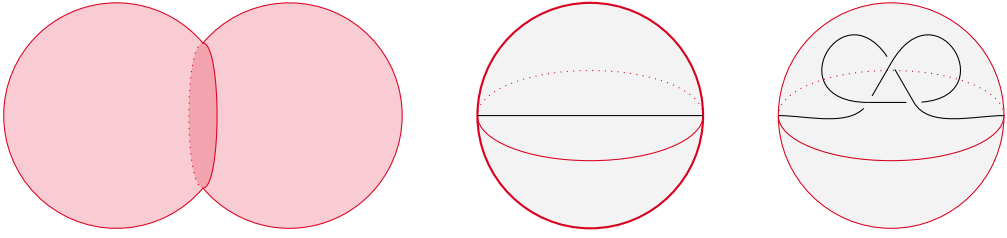

We denote by $M^1$ and $M^2$ the two torus links composing $M$.
By definition, each of them can be embedded on a standard torus $\To$ as pictured in \Cref{pic_link_embed}. A simple curve on $\Sigma \subseteq \mathbb{S}^3$ is \textbf{compressible} if it is non-contractible on $\Sigma$ and bounds a disk in $\mathbb{S}^3\setminus \Sigma$. The \textbf{compression-representativity} of a knot $K$ embedded on a surface $\Sigma$ is the minimum number of intersections between $K$ and a compressible curve in $\Sigma$. The compression-representativity of $M^1$ and $M^2$ is $n$; see \cite[Section~5]{Lunel_spherewidth_out}. Note that, since $M^1$ and $M^2$ are linked, their underlying tori intersect, but this is of no consequence for us. By \cite[Theorem~1.2]{Lunel_spherewidth_out}, there exist two bubble tangles $\T^1, \T^2$ of order $\frac{2}{3}n$, where each $\T^i$ is the bubble tangle defined by $M^i$ while completely forgetting about the other torus knot. Thus, $\T^i$ is a bubble tangle \emph{induced} by the torus on which $M^i$ is embedded. Such bubble tangles induced by a surface are called \textbf{compression bubble tangles}: given a link embedded on a surface $\Sigma$, the small sides of the bubble tangle are the balls $B$ which are either disjoint from $\Sigma$ or intersect $\Sigma$ trivially, i.e., the inclusion $i: B \cap \Sigma \rightarrow \Sigma$ is $\pi_1$-trivial, which means that it induces a trivial map on fundamental groups. For example, the annulus and disc obtained by intersecting the torus and green ball of \Cref{pic_trivial_intersections} have a $\pi_1$-trivial inclusion in the torus, but the annulus stemming from the intersection between the red ball and the torus does not.

\begin{figure}[ht]
\begin{center}
\begin{tikzpicture}
\clip (-6,-2.75) rectangle (6.2,2.75);
\coordinate (r1) at (1.75,0.25);
\coordinate (r2) at (1.5,-0.12);
\coordinate (r3) at (-1.5,-0.12);
\coordinate (r4) at (-1.75,0.25);

\fill [opacity = 0.2, blue] (0,0) circle (4 and 2.5) (r2) .. controls +(-150:1) and \control{(r3)}{(-30:1)} .. controls +(70:0.75) and \control{(r2)}{(110:0.75)};

\path [name path = torus] (0,0) circle (4 and 2.5);
\path [name path = intorus] (r2) .. controls +(-150:1) and \control{(r3)}{(-30:1)} .. controls +(70:0.75) and \control{(r2)}{(110:0.75)}; 
\path [name path = risp] (-3.1,0) circle (2 and 1.25); 
\path [name intersections={of=torus and risp,total=\tot}]
\foreach \s in {1,...,\tot}{coordinate (it\s) at (intersection-\s)};
\path [name intersections={of=intorus and risp,total=\tot}]
\foreach \s in {1,...,\tot}{coordinate (ii\s) at (intersection-\s)};

\begin{scope}[red!70!purple]
\begin{scope}[even odd rule]
\clip (-6,-2.75) rectangle (8,2.75) (-4,0) arc (-180:0:4 and 2.5) -- (r2) .. controls +(-150:1) and \control{(r3)}{(-30:1)} -- (-4,0);
\draw (-3.25,0) circle (2.5 and 1.75); 
\end{scope}
\begin{scope}
\clip (-4,0) arc (-180:0:4 and 2.5) -- (r2) .. controls +(-150:1) and \control{(r3)}{(-30:1)} -- (-4,0);
\draw [dotted](-3.25,0) circle (2.5 and 1.75); 
\end{scope}
\fill [fill opacity = 0.2] (-3.25,0) circle (2.5 and 1.75); 
\draw [very thick, dashed] (it1) .. controls +(50:0.5) and \control{(ii2)}{(100:0.5)};
\draw [dashed] (ii2) .. controls +(-150:0.5) and \control{(it1)}{(-80:0.5)};
\draw [very thick, dashed] (it2) .. controls +(90:0.5) and \control{(ii1)}{(150:0.5)};
\draw [dashed] (ii1).. controls +(-110:0.5) and \control{(it2)}{(-30:0.5)};

\begin{scope}[even odd rule]
\clip (0,0) circle (4 and 2.5) (r2) .. controls +(-150:1) and \control{(r3)}{(-30:1)} .. controls +(70:0.75) and \control{(r2)}{(110:0.75)};
\fill [opacity = 0.15] (-5.75,0) -- (it1) .. controls +(50:0.5) and \control{(ii2)}{(100:0.5)} -- (ii1) .. controls +(150:0.5) and \control{(it2)}{(90:0.5)} -- cycle;
\fill [opacity = 0.15] (-5.75,0) -- (it1) .. controls +(-80:0.5) and \control{(ii2)}{(-150:0.5)} -- (ii1) .. controls +(-110:0.5) and \control{(it2)}{(-30:0.5)} -- cycle; 
\end{scope}
\end{scope}

\coordinate (s0) at (2,0.2);
\coordinate (s1) at (2.5,1);
\coordinate (s2) at ($(s1)+(0.5,-0.5)$);
\coordinate (s3) at (2.75,-1);
\coordinate (s4) at ($(s3)+(-0.35,-0.35)$);
\coordinate (s5) at (6,1);
\coordinate (s6) at ($(s5)+(-150:0.75)$);

\begin{scope}[green!70!purple]
\draw (s0) .. controls +(135:2.5) and \control{(s5)}{(90:3)} .. controls +(-90:3) and \control{(s4)}{(-45:2)} .. controls +(135:0.5) and \control{(s3)}{(135:0.5)} .. controls +(-45:1.5) and \control{(s6)}{(-90:2)} .. controls +(90:1.75) and \control{(s0)}{(-45:2.5)};
\fill [opacity = 0.1] (s0) .. controls +(135:2.5) and \control{(s5)}{(90:3)} .. controls +(-90:3) and \control{(s4)}{(-45:2)} .. controls +(135:0.5) and \control{(s3)}{(135:0.5)} .. controls +(-45:1.5) and \control{(s6)}{(-90:2)} .. controls +(90:1.75) and \control{(s0)}{(-45:2.5)};
\fill [rotate around ={-45:(2.75,0.75)}, opacity = 0.35] (s1) arc (180:0:0.35 and 0.15) .. controls +(90:0.75) and \control{(s1)}{(90:0.75)};
\draw [rotate around ={-45:(2.75,0.75)}] (s2) .. controls +(90:0.75) and \control{(s1)}{(90:0.75)};

\draw [very thick, dashed] [rotate around ={-45:(2.75,0.75)}] (2.75,0.75) circle (0.35 and 0.15) (2.75,0.75) circle (1.05 and 0.35);
\fill [rotate around ={-45:(2.75,0.75)}, opacity = 0.25, even odd rule] (2.75,0.75) circle (0.35 and 0.15) (2.75,0.75) circle (1.05 and 0.35);
\draw [very thick, dashed] (s4) .. controls +(5:0.15) and \control{(s3)}{(-95:0.15)} .. controls +(-175:0.15) and \control{(s4)}{(85:0.15)};
\fill [opacity = 0.25] (s4) .. controls +(5:0.15) and \control{(s3)}{(-95:0.15)} .. controls +(-175:0.15) and \control{(s4)}{(85:0.15)};
\end{scope}

\begin{scope}
\clip (s0) .. controls +(135:2.5) and \control{(s5)}{(90:3)} .. controls +(-90:3) and \control{(s4)}{(-45:2)} .. controls +(135:0.5) and \control{(s3)}{(135:0.5)} .. controls +(-45:1.5) and \control{(s6)}{(-90:2)} .. controls +(90:1.75) and \control{(s0)}{(-45:2.5)};
\draw [dotted, blue] (0,0) circle (4 and 2.5);
\end{scope}

\path [name path = rsp] (-3.25,0) circle (2.5 and 1.75); 
\path [name intersections={of=torus and rsp,total=\tot}]
\foreach \s in {1,...,\tot}{coordinate (ot\s) at (intersection-\s)};
\path [name intersections={of=intorus and rsp,total=\tot}]
\foreach \s in {1,...,\tot}{coordinate (oi\s) at (intersection-\s)};
\begin{scope}[dotted]
\clip (-5.75,0) -- (ot1) -- (oi2) -- (ii1) .. controls +(150:0.5) and \control{(it2)}{(90:0.5)} -- cycle;
\draw [blue] (0,0) circle (4 and 2.5);
\draw [blue](r1) .. controls +(-90:0.11) and \control{(r2)}{(30:0.15)} .. controls +(-150:1) and \control{(r3)}{(-30:1)} .. controls +(150:0.15) and \control{(r4)}{(-90:0.11)};
\draw [blue] (r2) .. controls +(110:0.75) and \control{(r3)}{(70:0.75)};
\end{scope}
\begin{scope}[even odd rule]
\clip (-6,-2.75) rectangle (6.2,2.75) (-5.75,0) -- (ot1) -- (oi2) -- (ii1) .. controls +(150:0.5) and \control{(it2)}{(90:0.5)} -- cycle;
\clip (-6,-2.75) rectangle (6.2,2.75) (s0) .. controls +(135:2.5) and \control{(s5)}{(90:3)} .. controls +(-90:3) and \control{(s4)}{(-45:2)} .. controls +(135:0.5) and \control{(s3)}{(135:0.5)} .. controls +(-45:1.5) and \control{(s6)}{(-90:2)} .. controls +(90:1.75) and \control{(s0)}{(-45:2.5)};
\draw [blue] (0,0) circle (4 and 2.5);
\draw [blue](r1) .. controls +(-90:0.11) and \control{(r2)}{(30:0.15)} .. controls +(-150:1) and \control{(r3)}{(-30:1)} .. controls +(150:0.15) and \control{(r4)}{(-90:0.11)};
\draw [blue] (r2) .. controls +(110:0.75) and \control{(r3)}{(70:0.75)};
\end{scope}
\end{tikzpicture}
\caption{The green ball intersects the torus trivially while the red ball does not.}
\label{pic_trivial_intersections}
\end{center}
\end{figure}
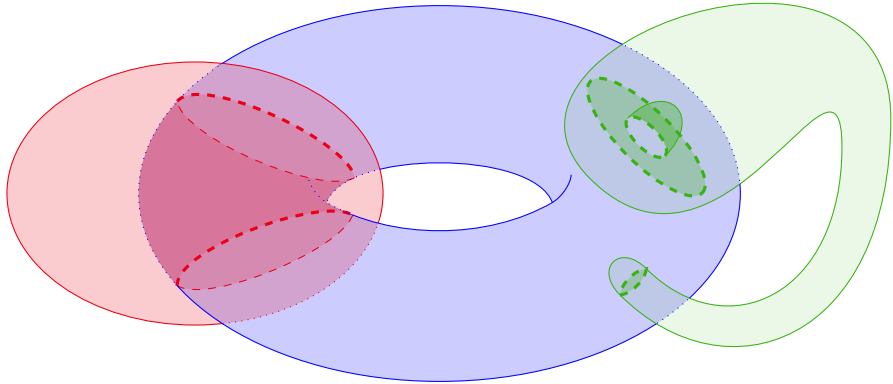

\begin{figure}[ht]
\begin{center}
\begin{tikzpicture}[scale = 0.42]
\def\p{13}
\def\b{7}
\def\c{0.25}
\def\l{0.5}
\def\diam{4}

\modulo{\b-1}{\b}{\q};
\pgfmathparse{360/\p}
\edef\a{\pgfmathresult}
\pgfmathparse{-\a}
\edef\d{\pgfmathresult}
\modulo{\q-1}{\q}{\qq}; 
\modulo{\p-1}{\p}{\pp}; 
\modulo{\qq-1}{\qq}{\qqq}; 

\foreach \i in {0,\qq}
{
\foreach \j in {0,1,...,\pp}
{
	\modulo{\i}{2}{\t};
	\coordinate (n-\i-\j) at (\a*\j+\t*\a/2+\d:1.5+\i);	
}
}
\foreach \j in {0,1,...,\pp}
{
	\coordinate (n-1-\j) at (\a*\j+\a/2+\d:2.5+\l);		
	\modulo{\qqq}{2}{\t}
	\coordinate (n-\qqq-\j) at (\a*\j+\t*\a/2+\d:1.5+\qqq-\l);
}

\foreach \j in {0,1,...,\pp}
{
\modulo{\j+1}{\p}{\jj};
\path (n-\qq-\j);
\pgfgetlastxy{\x}{\y};
\modulo{\qq}{2}{\t}
	\ifnum \t=1
		{		
		\modulo{\j-1+2*\t+\p}{\p}{\jj};
		\path (n-\qqq-\j);
		\pgfgetlastxy{\xx}{\yy};
		\path [reset cm] coordinate (co-\j+) at (2*\x-\xx,2*\y-\yy);
		\path (n-\qqq-\jj);
		\pgfgetlastxy{\xx}{\yy};
		\path [reset cm] coordinate (co-\j-) at (2*\x-\xx,2*\y-\yy);
		}
	\else	
		{
		\modulo{\j-1+2*\t+\p}{\p}{\jj};
		\path (n'-\qqq-\jj);
		\pgfgetlastxy{\xx}{\yy};		
		\path [reset cm] coordinate (co-\j+) at (2*\x-\xx,2*\y-\yy);
		\path (n-\qqq-\j);
		\pgfgetlastxy{\xx}{\yy};
		\path [reset cm] coordinate (co-\j-) at (2*\x-\xx,2*\y-\yy);
		}
	\fi
\path (n-0-\j);
\pgfgetlastxy{\x}{\y};
\modulo{\j-1+\p}{\p}{\jj};
\path (n-1-\jj);
\pgfgetlastxy{\xx}{\yy};
\path [reset cm] coordinate (ci-\j+) at (\x+\x*\c-\xx*\c,\y+\y*\c-\yy*\c);
\path (n-1-\j);
\pgfgetlastxy{\xx}{\yy};
\path [reset cm] coordinate (ci-\j-) at (\x+\x*\c-\xx*\c,\y+\y*\c-\yy*\c);
}
\modulo{\qq}{2}{\t}
\def\j{0}
\modulo{\j+1}{\p}{\jj};
\begin{scope}
\clip (0,0) -- (\a*\j+\t*\a/2+\d-\a/2:3.5+\qq) -- (\a*\j+\t*\a/2+\d:3.5+\qq) --cycle;
\modulo{\j+1}{\p}{\jj};
\end{scope}
\begin{scope}
\clip (0,0) -- (\a*\j+\t*\a/2+\d+\a/2:3.5+\qq) -- (\a*\j+\t*\a/2+\d:3.5+\qq) --cycle;
\modulo{\j+1}{\p}{\jj};
\end{scope}
\foreach \j in {0,...,\pp}
{
\begin{scope}
\clip (0,0) -- (\a*\j+\t*\a/2+\d+\a/2:3.5+\qq) -- (\a*\j+\t*\a/2+\d+\a:3.5+\qq) --cycle;
\modulo{\j+1}{\p}{\jj};
\draw [thick, opacity=0.3] (n-\qq-\j) .. controls (co-\j+) and (co-\jj-) .. (n-\qq-\jj);
\end{scope}
\begin{scope}
\clip (0,0) -- (\a*\j+\t*\a/2+\d+\a/2:3.5+\qq) -- (\a*\j+\t*\a/2+\d-\a:3.5+\qq) --cycle;
\modulo{\j+1}{\p}{\jj};
\draw [thick] (n-\qq-\j) .. controls (co-\j+) and (co-\jj-) .. (n-\qq-\jj);
\end{scope}
\begin{scope}
\clip (0,0) -- (\a*\j+\t*\a/2+\d-\a/2:3.5+\qq) -- (\a*\j+\t*\a/2+\d:3.5+\qq) --cycle;
\modulo{\j+1}{\p}{\jj};
\draw [thick, opacity=0.3] (n-0-\j) .. controls (ci-\j+) and (ci-\jj-) .. (n-0-\jj);
\end{scope}
\begin{scope}
\clip (0,0) -- (\a*\j+\t*\a/2+\d+\a/2:3.5+\qq) -- (\a*\j+\t*\a/2+\d:3.5+\qq) --cycle;
\modulo{\j+1}{\p}{\jj};
\draw [thick] (n-0-\j) .. controls (ci-\j+) and (ci-\jj-) .. (n-0-\jj);
\end{scope}
}

\euclidiv{\qq}{2}{\d};
\modulo{\qq}{2}{\t}	
\ifnum \t = 1
{
	\foreach \j in {0,...,\pp}
	{
	\foreach \i in {0,...,\qq}
	{
		\modulo{\j+\d}{\p}{\jj}
		\path [name path =b+] (n-0-\j) .. controls (n-1-\j) and (n-\qqq-\jj) .. (n-\qq-\jj);
		\modulo{\j+\i}{\p}{\i}
		\modulo{\i-\d-1+\p}{\p}{\ii}
		\modulo{\i-1+\p}{\p}{\iii}
		\modulo{\ii+1}{\p}{\iip}
		\path [name path =b-] (n-0-\i) .. controls (n-1-\iii) and (n-\qqq-\iip) .. (n-\qq-\ii);
		\path [name intersections={of=b+ and b-,total=\tot}]
		\foreach \s in {1,...,\tot}{coordinate (i-\j-\i) at (intersection-\s)};
	}
	}
}
\else
{
	\foreach \j in {0,...,\pp}
	{
	\foreach \i in {0,...,\qqq}
	{
		\modulo{\j+\d-1}{\p}{\jj}
		\modulo{\jj-1+\p}{\p}{\jjj}
		\path [name path =b+] (n-0-\j) .. controls (n-1-\j) and (n-\qqq-\jjj) .. (n-\qq-\jj);
		\modulo{\j+\i}{\p}{\i}
		\modulo{\i-\d+\p}{\p}{\ii}
		\modulo{\i-1+\p}{\p}{\im}
		\path [name path =b-] (n-0-\i) .. controls (n-1-\im) and (n-\qqq-\ii) .. (n-\qq-\ii);			
		\path [name intersections={of=b+ and b-,total=\tot}]
		\foreach \s in {1,...,\tot}{coordinate (i-\j-\i) at (intersection-\s)};
	}
	}
		
}
\fi

\ifnum \t = 1
{	
	\foreach \j in {0,...,\pp}
	{
		\modulo{\j-\d-1+\p}{\p}{\jj}
		\modulo{\jj+1}{\p}{\jjp}
		\modulo{\j-1+\p}{\p}{\jm}
		\draw [thick, opacity = 0.3] (n-0-\j) .. controls (n-1-\jm) and (n-\qqq-\jjp) .. (n-\qq-\jj);
	}
	\foreach \j in {0,...,\pp}
	{
		\modulo{\j+\d}{\p}{\jj}
		\draw [thick] (n-0-\j) .. controls (n-1-\j) and (n-\qqq-\jj) .. (n-\qq-\jj);
	}
}
\else
{
	\foreach \j in {0,...,\pp}
	{
		\modulo{\j-\d+\p}{\p}{\jj}
		\modulo{\j-1+\p}{\p}{\jm}
		\draw [thick, opacity = 0.3] (n-0-\j) .. controls (n-1-\jm) and (n-\qqq-\jj) .. (n-\qq-\jj);
	}
	\foreach \j in {0,...,\pp}
	{
		\modulo{\j+\d-1}{\p}{\jj}
		\modulo{\jj-1+\p}{\p}{\jjj}
		\draw [thick] (n-0-\j) .. controls (n-1-\j) and (n-\qqq-\jjj) .. (n'-\qq-\jj);
	}	
}
\fi


\begin{scope}[xshift=17 cm]
\modulo{\b-1}{\b}{\q};
\modulo{\q-1}{\q}{\qq}; 
\modulo{\p-1}{\p}{\pp}; 
\modulo{\qq-1}{\qq}{\qqq}; 
\pgfmathparse{360/\p}
\edef\a{\pgfmathresult}
\pgfmathparse{\a/2+\a*\pp/2-\a}
\edef\d{\pgfmathresult}

\foreach \i in {0,\qq}
{
\foreach \j in {0,1,...,\pp}
{
	\modulo{\i}{2}{\t};
	\coordinate (n'-\i-\j) at (\a*\j+\t*\a/2+\d:1.5+\i);	
}
}
\foreach \j in {0,1,...,\pp}
{
	\coordinate (n'-1-\j) at (\a*\j+\a/2+\d:2.5+\l);		
	\modulo{\qqq}{2}{\t}
	\coordinate (n'-\qqq-\j) at (\a*\j+\t*\a/2+\d:1.5+\qqq-\l);
}

\foreach \j in {0,1,...,\pp}
{
\modulo{\j+1}{\p}{\jj};
\path (n'-\qq-\j);
\pgfgetlastxy{\x}{\y};
\modulo{\qq}{2}{\t}
	\ifnum \t=1
		{		
		\modulo{\j-1+2*\t+\p}{\p}{\jj};
		\path (n'-\qqq-\j);
		\pgfgetlastxy{\xx}{\yy};
		\path [reset cm] coordinate (co'-\j+) at (2*\x-\xx,2*\y-\yy);
		\path (n'-\qqq-\jj);
		\pgfgetlastxy{\xx}{\yy};
		\path [reset cm] coordinate (co'-\j-) at (2*\x-\xx,2*\y-\yy);
		}
	\else	
		{
		\modulo{\j-1+2*\t+\p}{\p}{\jj};
		\path (n'-\qqq-\jj);
		\pgfgetlastxy{\xx}{\yy};		
		\path [reset cm] coordinate (co'-\j+) at (2*\x-\xx,2*\y-\yy);
		\path (n'-\qqq-\j);
		\pgfgetlastxy{\xx}{\yy};
		\path [reset cm] coordinate (co'-\j-) at (2*\x-\xx,2*\y-\yy);
		}
	\fi
\path (n'-0-\j);
\pgfgetlastxy{\x}{\y};
\modulo{\j-1+\p}{\p}{\jj};
\path (n'-1-\jj);
\pgfgetlastxy{\xx}{\yy};
\path [reset cm] coordinate (ci'-\j+) at (\x+\x*\c-\xx*\c,\y+\y*\c-\yy*\c);
\path (n'-1-\j);
\pgfgetlastxy{\xx}{\yy};
\path [reset cm] coordinate (ci'-\j-) at (\x+\x*\c-\xx*\c,\y+\y*\c-\yy*\c);
}
\def\j{0}
\modulo{\j+1}{\p}{\jj};
\begin{scope}
\clip (0,0) -- (\a*\j+\t*\a/2+\d-\a/2:3.5+\qq) -- (\a*\j+\t*\a/2+\d:3.5+\qq) --cycle;
\modulo{\j+1}{\p}{\jj};
\draw [thick, opacity=0.3] (n'-0-\j) .. controls (ci'-\j+) and (ci'-\jj-) .. (n'-0-\jj);
\end{scope}
\begin{scope}
\clip (0,0) -- (\a*\j+\t*\a/2+\d+\a/2:3.5+\qq) -- (\a*\j+\t*\a/2+\d:3.5+\qq) --cycle;
\modulo{\j+1}{\p}{\jj};
\draw [thick] (n'-0-\j) .. controls (ci'-\j+) and (ci'-\jj-) .. (n'-0-\jj);
\end{scope}
\foreach \j in {1,...,\pp}
{
\begin{scope}
\clip (0,0) -- (\a*\j+\t*\a/2+\d+\a/2:3.5+\qq) -- (\a*\j+\t*\a/2+\d+\a:3.5+\qq) --cycle;
\modulo{\j+1}{\p}{\jj};
\draw [thick, opacity=0.3] (n'-\qq-\j) .. controls (co'-\j+) and (co'-\jj-) .. (n'-\qq-\jj);
\end{scope}
\begin{scope}
\clip (0,0) -- (\a*\j+\t*\a/2+\d+\a/2:3.5+\qq) -- (\a*\j+\t*\a/2+\d-\a:3.5+\qq) --cycle;
\modulo{\j+1}{\p}{\jj};
\draw [thick] (n'-\qq-\j) .. controls (co'-\j+) and (co'-\jj-) .. (n'-\qq-\jj);
\end{scope}
\begin{scope}
\clip (0,0) -- (\a*\j+\t*\a/2+\d-\a/2:3.5+\qq) -- (\a*\j+\t*\a/2+\d:3.5+\qq) --cycle;
\modulo{\j+1}{\p}{\jj};
\draw [thick, opacity=0.3] (n'-0-\j) .. controls (ci'-\j+) and (ci'-\jj-) .. (n'-0-\jj);
\end{scope}
\begin{scope}
\clip (0,0) -- (\a*\j+\t*\a/2+\d+\a/2:3.5+\qq) -- (\a*\j+\t*\a/2+\d:3.5+\qq) --cycle;
\modulo{\j+1}{\p}{\jj};
\draw [thick] (n'-0-\j) .. controls (ci'-\j+) and (ci'-\jj-) .. (n'-0-\jj);
\end{scope}
}

\euclidiv{\qq}{2}{\d};
\modulo{\qq}{2}{\t}

\ifnum \t = 1
{	
	\foreach \j in {0,...,\pp}
	{
		\modulo{\j-\d-1+\p}{\p}{\jj}
		\modulo{\jj+1}{\p}{\jjp}
		\modulo{\j-1+\p}{\p}{\jm}
		\draw [thick, opacity = 0.3] (n'-0-\j) .. controls (n'-1-\jm) and (n'-\qqq-\jjp) .. (n'-\qq-\jj);
	}
	\foreach \j in {0,...,\pp}
	{
		\modulo{\j+\d}{\p}{\jj}
		\draw [thick] (n'-0-\j) .. controls (n'-1-\j) and (n'-\qqq-\jj) .. (n'-\qq-\jj);
	}
}
\else
{
	\foreach \j in {0,...,\pp}
	{
		\modulo{\j-\d+\p}{\p}{\jj}
		\modulo{\j-1+\p}{\p}{\jm}
		\draw [thick, opacity = 0.3] (n'-0-\j) .. controls (n'-1-\jm) and (n'-\qqq-\jj) .. (n'-\qq-\jj);
	}
	\foreach \j in {0,...,\pp}
	{
		\modulo{\j+\d-1}{\p}{\jj}
		\modulo{\jj-1+\p}{\p}{\jjj}
		\draw [thick] (n'-0-\j) .. controls (n'-1-\j) and (n'-\qqq-\jjj) .. (n'-\qq-\jj);
	}	
}
\fi
\end{scope}

\begin{scope}
\clip (0,0) rectangle (17, 8.5);
\draw [thick] (n'-\qq-\j) .. controls (co'-\j+) and ($(90:1.5)+(7,0)$) .. (7,0) .. controls +(90:-1.5) and (co'-\jj-) .. (n'-\qq-\jj);
\end{scope}
\begin{scope}
\clip (0,0) rectangle (17, -8.5);
\draw [thick, opacity = 0.3] (n'-\qq-\j) .. controls (co'-\j+) and ($(90:1.5)+(7,0)$) .. (7,0) .. controls +(90:-1.5) and (co'-\jj-) .. (n'-\qq-\jj);
\end{scope}

\path (n-\qq-0) .. controls (co-0+) and (co'-1-) .. (n'-\qq-1) node (mlow) [pos = 0.5] {};
\path (n-\qq-1) .. controls (co-1-) and (co'-0+) .. (n'-\qq-0) node (mup) [pos = 0.5] {};
\coordinate (cmlow) at (mlow);
\coordinate (cmup) at (mup);


\filldraw [purple!50!red, fill opacity = 0.2, even odd rule] ($(17,0)+(160:7.72)$) arc (160:-160:7.72) .. controls + (110:1) and \control{(7,0)}{(-90:1.75)} .. controls +(90:1.5) and \control{($(17,0)+(160:7.72)$)}{(-110:1)} (17,0) circle (1.23);
\filldraw [blue, fill opacity = 0.2, even odd rule] (0,0) circle (7.72) (0,0) circle (1.23);

\filldraw [thick, green!50!black,fill opacity = 0.1] (0,0) circle (9.5cm);
\draw [green!50!black, opacity = 0.4] (30:9.5) .. controls + (-60:2.5) and \control{(-150:9.5)}{(-60:2.5)};
\draw [green!50!black, dotted, opacity = 0.4] (30:9.5) .. controls + (120:2.5) and \control{(-150:9.5)}{(120:2.5)};

\path [name path = left] (0,0) circle (9.5cm);
\path [name path = right] ($(17,0)+(160:7.72)$) arc (160:-160:7.72) .. controls + (110:1) and \control{(7,0)}{(-90:1.75)} .. controls +(90:1.5) and \control{($(17,0)+(160:7.72)$)}{(-110:1)};
\path [name intersections={of=right and left,total=\tot}]
\foreach \s in {1,...,\tot}{coordinate (i\s) at (intersection-\s)};
\draw [green!50!black, opacity = 0.8, dashed, thick] (i1) .. controls + (85:0.5) and \control{(i2)}{(-85:0.5)};
\draw [green!50!black, opacity = 0.4, dashed, thick] (i1) .. controls + (120:1.25) and \control{(i2)}{(-120:1.25)};

\node [thick,  green!50!black] at (35:10.2cm) {$S_0$};
\node at ($(-135:8.5)+(17,0)$) {$M^2$};
\node at ($(-45:8.5)$) {$M^1$};
\end{tikzpicture}
\caption{Link components $M^1$ and $M^2$ embedded on tori, and sphere $S_0$.}
\label{pic_link_embed}
\end{center}
\end{figure}
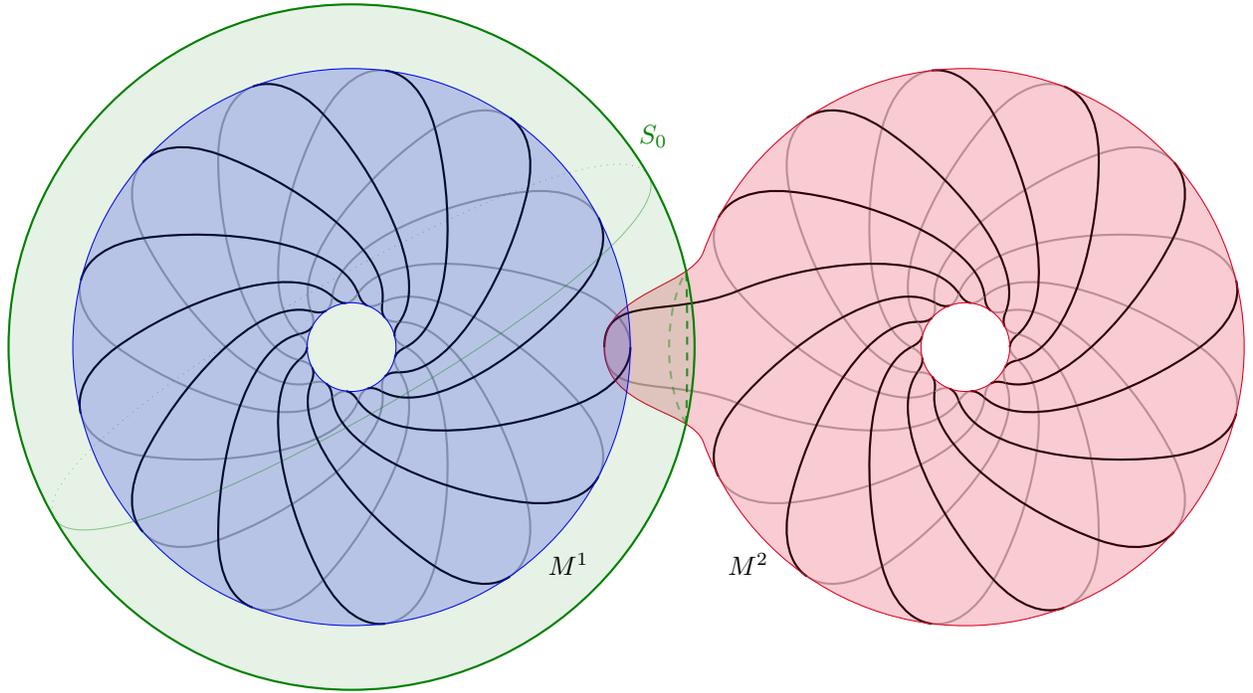

\subsection{Proof of \Cref{th_unlink_analysed}}\label{sec_proof}

In order to prove \Cref{th_unlink_analysed}, we assume that $\addcross(\mathcal{D}(n,n+1), R)< k = \frac{2}{3}n-2$. By virtue of \Cref{lem_inter_sphere}, it follows that $|M^i_t \cap S_t| \leq |M_t \cap S_t| < \frac{2}{3}n$ for $i \in \{1,2\}$. Now, changes to $|M_t \cap S_t|$ can only happen at critical times where the number of crossings involving $\mathcal{U}_t$ in $\mathcal{D}_t$ increases or decreases. In particular, we can ignore RIII moves between $\mathcal{U}_t$ and $\mathcal{M}_t$. Since we only consider moves involving both $\mathcal{U}_t$ and $\mathcal{M}_t$, only RII moves are relevant for our study. We denote the critical times of RII moves involving both $\mathcal{U}_t$ and $\mathcal{M}_t$ by $(t_j)_{1\leq j \leq s} \subset [0,1]$. These times are the only times where $S_t$ is not transverse to $M_t$, but \textbf{finitely tangent}, meaning that $M_t$ is tangent to $S_t$, but their number of intersections is still finite.

Recall that, up to applying $\phi'^{-1}$, we assume that for all $t \in [0,1]$ $M_t$ is fixed; we denote it by $M$. By \ref{def_T2}, for all times $t \in [0,1] \smallsetminus \{t_j\}_{1 \leq j \leq s}$, $S_t$ has a small side corresponding to a $3$-ball $B_t^i \in \T^i$. We study how these small sides evolve throughout the sweep-out. Special attention must be paid to tangencies at times $\{t_j\}_{1 \leq j \leq s}$, where small sides are not defined. For convenience, we denote by $\Theta: \Sp^3 \times [0,1] \to \Sp^3$ an isotopy describing the evolution of $S_t$, i.e., such that $\Theta(S_0,t)=S_t$.

\medskip

We start by defining the \textbf{agreement} between bubble tangles $\T^1$ and $\T^2$ as a map $a : [0,1] \smallsetminus \{t_j\}_{1 \leq j \leq s} \rightarrow \{0,1\}$, such that $a(t) = 1$ if $B_t^1 = B_t^2$ and $0$ otherwise. With this definition, and under the assumption that $\addcross(\mathcal{D}(n,n+1), R)< k$, we can infer that $a(0) = 0$ and $a(1) =1$: first, note that $a(1)=1$. Indeed, $\mathcal{U}_1$ is disjoint from $\mathcal{M}$ so that $S_1$ is disjoint from $M$. It follows that one side of $S_1$ is an empty ball with respect to both $M^1$ and $M^2$ and thus $a(1)=1$ by \ref{def_T4}. Second, the small side of $S_0$ in $\T^1$ is, by \ref{def_T4}, the ball not containing $M^1$. For $\T^2$, as it is illustrated by \Cref{pic_link_embed}, $S_0$ contains an $M^2$-trivial ball and $M^1$ on its small side, and the remaining of $M^2$ on the other side. It follows that the small sides disagree at time $0$, and hence $a(0)=0$.

We want to show that $a$ stays constant on its domain of definition. For this, we introduce the following definition: Two disjoint spheres $S, S' \subset \Sp^3$ are \textbf{braid-equivalent} with respect to some link $L \subset \Sp^3$, if $L$ forms a braid in the product region between $S$ and $S'$, or, equivalently, if the product region between $S$ and $S'$ is homeomorphic to $S_\ell \times [0,1]$ where $S_\ell$ is the $2$-sphere with $\ell$ holes, $\ell \geq 0$, see \Cref{pic_braid_equi} for an illustration.
 
\begin{figure}[ht]
\begin{center}
\begin{tikzpicture}
\begin{scope}[red!50!purple]
\draw (0,0) circle (1.15);
\draw (0,0) circle (1.5);
\draw (0,0) circle (1.85);
\end{scope}
\draw (170:2) -- (160:1.6) -- (150:1);
\draw (115:2) -- (125:1.6) -- (125:1.3) -- (115:1);
\draw (135:2) -- (135:1);
\draw (5:0.95) -- (0:1.3) -- (-10:1.25) -- (-15:0.95);
\end{tikzpicture}
\caption{A cross-section view of three spheres and a link. The two outermost spheres are braid-equivalent, but the innermost one is not braid-equivalent to the others.}
\label{pic_braid_equi}
\end{center}
\end{figure}
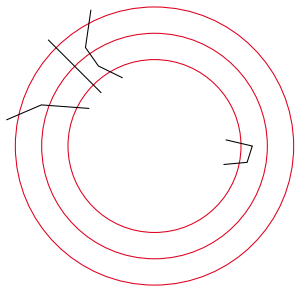

We have the following statement for two disjoint braid-equivalent spheres. 

\begin{lemma}[\cite{Lunel_spherewidth_out}, Lemma 3.2]\label{lem_braid_equi}
Let $\T$ be a bubble tangle and $S,S' \subset \Sp^3$ be two braid-equivalent spheres. Define $\Sp^3 \smallsetminus S = \{B_1,B_2\}$ and $\Sp^3 \smallsetminus S' = \{B'_1,B'_2\}$ such that $B_1 \subset B'_1$. If $B_1 \in \T$ then $B'_1 \in \T$.
\end{lemma}

The compression bubble tangles furthermore have the following property.

\begin{lemma}\label{lem_bubble_tangle_monot}
Let $\T$ be a compression bubble tangle of $L$, induced by an embedding of $L$ into some surface $\Sigma \subset \Sp^3$, and let $k$ be the order of $\T$. Then, for two closed $3$-balls $A,B \subset \Sp^3$, if $A \in \T$, $B \subset A$, and $|\partial B \cap L| < k$, then $B \in \T$. That is, $\T$ is \emph{stable by inclusion} up to \ref{def_T1}.
\end{lemma}

\begin{proof}
Let $A$ and $B$ be two closed balls of $\Sp^3$ such that $|\partial A \cap L|, |\partial B \cap L| < k$, and $A \in \T$. By definition, $A \cap \Sigma$ is $\pi_1$-trivial. Since $B \subset A $, we have $B \cap \Sigma \subset A \cap \Sigma$ so that the inclusion morphisms $i_*$ satisfy $\pi_1 (B \cap \Sigma) \rightarrow \pi_1 (A \cap \Sigma) \rightarrow 0$. Hence, $B \cap \Sigma$ is $\pi_1$-trivial and $B \in \T$.
\end{proof}

To handle the back-and-forths of our sweep-out, we introduce a new equivalence relation on spheres. Two intersecting spheres $S,S' \subset \Sp^3$ are \textbf{intersection-equivalent} if there exists an isotopy between them which stays constant on their intersection $S \cap S'$, see \Cref{pic_intersection_equivalent}. Note that, by this definition, a sphere $S$ is intersection-equivalent with itself. 

The structure of the remainder of the proof is to show that the agreement $a$ is constant on $[0,1]$ by showing that it is locally constant via Lemmas~\ref{lem_gen_orient_local} and~\ref{lem_gen_orient_local_bis}.

\begin{figure}[ht]
\begin{center}
\begin{tikzpicture}[scale=1.5]
\clip (-2,-1.6) rectangle (7,2.25);
\fill[opacity = 0.2, blue] (0,0) circle (1.5cm);
\fill [opacity = 0.2, purple] (-130:1.5) .. controls +(-145:1) and \control{(-170:1.5)}{(-170:1)} .. controls +(10:1) and \control{(150:1.5)}{(-20:1)} node (n1) [pos = 0.15] {} node (n2)[pos = 0.85] {} .. controls +(160:1.5) and \control{(110:1.5)}{(100:1.5)} .. controls +(-80:1.5) and \control{(45:1.5)}{(-120:1.5)} node (n3) [pos = 0.15] {} node (n4)[pos = 0.75] {} .. controls +(60:1.75) and \control{(-45:1.5)}{(-30:1.75)} node (n5) [pos = 0.3] {} node (n6)[pos = 0.75] {}  .. controls +(150:1.25) and \control{(-130:1.5)}{(50:1.25)} node (n7) [pos = 0.15] {} node (n8)[pos = 0.85] {} ;
\foreach \i in {1,...,8}{\coordinate (c\i) at (n\i);};

\begin{scope}
\clip (c1) -- (c2) -- (c3) -- (c4) -- (c5) -- (c6) -- (c7) -- (c8) --cycle;
\draw [purple, dotted] (-130:1.5) .. controls +(-145:1) and \control{(-170:1.5)}{(-170:1)} .. controls +(10:1) and \control{(150:1.5)}{(-20:1)} .. controls +(160:1.5) and \control{(110:1.5)}{(100:1.5)} .. controls +(-80:1.5) and \control{(45:1.5)}{(-120:1.5)} .. controls +(60:1.75) and \control{(-45:1.5)}{(-30:1.75)} .. controls +(150:1.25) and \control{(-130:1.5)}{(50:1.25)};
\end{scope}

\begin{scope}[even odd rule]
\clip (-2,-2) rectangle (2,2.2) (c1) -- (c2) -- (c3) -- (c4) -- (c7) -- (c8) --cycle;
\draw [purple] (-130:1.5) .. controls +(-145:1) and \control{(-170:1.5)}{(-170:1)} .. controls +(10:1) and \control{(150:1.5)}{(-20:1)} .. controls +(160:1.5) and \control{(110:1.5)}{(100:1.5)} .. controls +(-80:1.5) and \control{(45:1.5)}{(-120:1.5)} .. controls +(60:1.75) and \control{(-45:1.5)}{(-30:1.75)} .. controls +(150:1.25) and \control{(-130:1.5)}{(50:1.25)};
\end{scope}

\begin{scope}[dashed, purple]
\draw (c2) .. controls +(160:0.25) and \control{(c3)}{(100:0.25)}; 
\draw (c4) .. controls +(45:0.75) and \control{(c7)}{(-40:0.75)}; 
\draw (c8) .. controls +(-130:0.35) and \control{(c1)}{(-160:0.35)}; 

\draw (c2) .. controls +(-40:0.25) and \control{(c3)}{(-80:0.25)}; 
\draw (c4) .. controls +(-140:0.35) and \control{(c7)}{(150:0.35)}; 
\draw (c8) .. controls +(30:0.25) and \control{(c1)}{(0:0.25)}; 
\end{scope}

\begin{scope}
\clip (-130:1) -- (-130:2) -- (-170:2) -- (-170:1) -- (150:1) -- (150:2) -- (110:2) -- (110:1) -- (45:1) -- (45:2.5) -- (-45:2.5) -- (-45:1) -- cycle;
\draw [dotted, blue] (0,0) circle (1.5cm);
\end{scope}
\begin{scope}[even odd rule]
\clip (-2,-2) rectangle (2,2.2) (-130:1) -- (-130:2) -- (-170:2) -- (-170:1) -- (150:1) -- (150:2) -- (110:2) -- (110:1) -- (45:1) -- (45:2.5) -- (-45:2.5) -- (-45:1) -- cycle;
\draw [blue] (0,0) circle (1.5cm);
\end{scope}

\begin{scope}[xshift=5cm]
\fill[opacity = 0.2, blue] (0,0) circle (1.5cm);
\fill [opacity = 0.2, purple] (-130:1.5) .. controls +(-145:1) and \control{(-170:1.5)}{(-170:1)} .. controls +(10:1) and \control{(150:1.5)}{(-20:1)} node (n1) [pos = 0.15] {} node (n2)[pos = 0.85] {} .. controls +(160:1.5) and \control{(110:1.5)}{(100:1.5)} .. controls +(-80:1.5) and \control{(45:1.5)}{(-120:1.5)}node (n3) [pos = 0.15] {} node (n4)[pos = 0.75] {} .. controls +(60:0.45) and \control{(70:1.5)}{(80:0.5)} .. controls +(-100:0.55) and \control{(95:1.5)}{(-110:0.55)} node (n5) [pos = 0.15] {} node (n6) [pos = 0.85] {} .. controls +(70:3) and \control{(-45:1.5)}{(-30:2.5)} .. controls +(150:1.25) and \control{(-130:1.5)}{(50:1.25)} node (n7) [pos = 0.15] {} node (n8)[pos = 0.85] {};
\foreach \i in {1,...,8}{\coordinate (c\i) at (n\i);};

\begin{scope}
\clip (c1) -- (c2) -- (c3) -- (c6) -- (c5) -- (c7) -- (c8) --cycle;
\draw [purple, dotted] (-130:1.5) .. controls +(-145:1) and \control{(-170:1.5)}{(-170:1)} .. controls +(10:1) and \control{(150:1.5)}{(-20:1)} .. controls +(160:1.5) and \control{(110:1.5)}{(100:1.5)} .. controls +(-80:1.5) and \control{(45:1.5)}{(-120:1.5)} .. controls +(60:0.45) and \control{(70:1.5)}{(80:0.5)} .. controls +(-100:0.55) and \control{(95:1.5)}{(-110:0.55)} .. controls +(70:3) and \control{(-45:1.5)}{(-30:2.5)} .. controls +(150:1.25) and \control{(-130:1.5)}{(50:1.25)};
\end{scope}

\begin{scope}[even odd rule]
\clip (-2,-2) rectangle (2,2.3) (c1) -- (c2) -- (c3) -- (c6) -- (c5) -- (c7) -- (c8) --cycle;
\draw [purple] (-130:1.5) .. controls +(-145:1) and \control{(-170:1.5)}{(-170:1)} .. controls +(10:1) and \control{(150:1.5)}{(-20:1)} .. controls +(160:1.5) and \control{(110:1.5)}{(100:1.5)} .. controls +(-80:1.5) and \control{(45:1.5)}{(-120:1.5)} .. controls +(60:0.45) and \control{(70:1.5)}{(80:0.5)} .. controls +(-100:0.55) and \control{(95:1.5)}{(-110:0.55)} .. controls +(70:3) and \control{(-45:1.5)}{(-30:2.5)} .. controls +(150:1.25) and \control{(-130:1.5)}{(50:1.25)};
\end{scope}

\begin{scope}[dashed, purple]
\draw (c2) .. controls +(160:0.25) and \control{(c3)}{(100:0.25)}; 
\draw (c4) .. controls +(45:0.75) and \control{(c7)}{(-40:0.75)}; 
\draw (c8) .. controls +(-130:0.35) and \control{(c1)}{(-160:0.35)}; 
\draw (c2) .. controls +(-40:0.25) and \control{(c3)}{(-80:0.25)}; 
\draw (c4) .. controls +(-140:0.35) and \control{(c7)}{(150:0.35)}; 
\draw (c8) .. controls +(30:0.25) and \control{(c1)}{(0:0.25)}; 
\draw (c6) .. controls +(90:0.15) and \control{(c5)}{(75:0.15)};
\draw (c6) .. controls +(-90:0.1) and \control{(c5)}{(-105:0.1)};  
\end{scope}

\begin{scope}
\clip (-130:1) -- (-130:2) -- (-170:2) -- (-170:1) -- (150:1) -- (150:2) -- (110:2) -- (110:1) -- (95:1) -- (95:2) -- (70:2) -- (70:1) -- (45:1) -- (45:2.5) -- (-45:2.5) -- (-45:1) -- cycle;
\draw [dotted, blue] (0,0) circle (1.5cm);
\end{scope}
\begin{scope}[even odd rule]
\clip (-2,-2) rectangle (2,2.2) (-130:1) -- (-130:2) -- (-170:2) -- (-170:1) -- (150:1) -- (150:2) -- (110:2) -- (110:1) -- (95:1) -- (95:2) -- (70:2) -- (70:1) -- (45:1) -- (45:2.5) -- (-45:2.5) -- (-45:1) -- cycle;
\draw [blue] (0,0) circle (1.5cm);
\end{scope}
\end{scope}
\end{tikzpicture}
\caption{Left: two intersection-equivalent spheres. Right: two spheres that are not intersection-equivalent: the red sphere has an annulus component that cannot be mapped to a component of the blue one. Indeed, the components of the blue sphere are discs and a sphere with $4$ punctures.}
\label{pic_intersection_equivalent}
\end{center}
\end{figure}
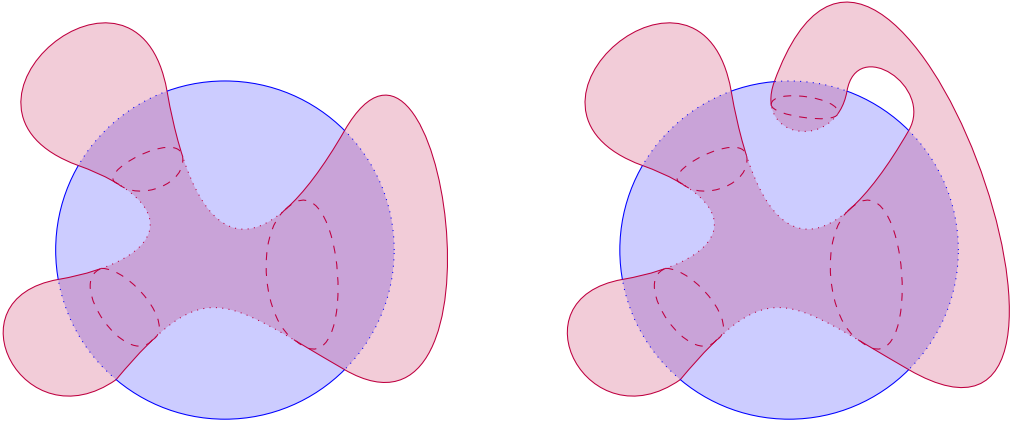 

Since we work in the piecewise-linear setting, we have the following observation:

\begin{lemma}\label{lem_gen_close_enough}
Let $t \in [0,1]$. There exists a neighbourhood $V \subset [0,1]$ of $t$ such that for all spheres $S_v$, $v \in V$, $S_v$ and $S_t$ are either disjoint or intersection-equivalent.
\end{lemma}

\begin{lemma}\label{lem_gen_orient_local}
Let $t \in [0,1] \smallsetminus \{t_j\}_{1 \leq j \leq s}$ be a non-critical time. There exists a neighbourhood $V \subset [0,1] \smallsetminus \{t_j\}_{1 \leq j \leq s}$ of $t$ such that $a$ is constant on $V$.
\end{lemma}

\begin{proof}[Proof of \Cref{lem_gen_orient_local}]
Let $V$ be an open connected neighbourhood of $t\in [0,1] \smallsetminus \{t_j\}_{1\leq j \leq s}$ such that the spheres $S_v$ for $v \in V$ are either disjoint or intersection-equivalent to $S_t$. Such a neighbourhood is the connected component of $t$ in the neighbourhood provided by \Cref{lem_gen_close_enough} intersected with $[0,1] \smallsetminus \{t_j\}_{1\leq j \leq s}$.

\medskip

\noindent
{\bf Case 1:} $S_v \cap S_t = \emptyset$. First assume that $S_v \subset B_t^i$ for some $i \in \{1,2\}$. It follows that one of the two $3$-balls $\Sp^3 \smallsetminus S_v$ is a subset of $B^i_t$. Denote this ball by $B^i_v$. Since no critical time is contained in $V$, $|S_v \cap M| = |S_t \cap M|$. Hence, we can apply \Cref{lem_bubble_tangle_monot} to conclude that $B^i_v \in \T^i$. If this situation applies for both $i=1$ and $i=2$ and $a(t)=1$, we therefore have $a(v) = 1$.

If $S_v \subset (B_t^i)^{c}$, and we denote the component of $\Sp^3 \smallsetminus S_v$ containing $B^i_t$ by $B^i_v$, we notice that since there is no critical time between $t$ and $v$, $S_v$ and $S_t$ are braid-equivalent. Therefore, we can apply \Cref{lem_braid_equi} to conclude that $B_v^i \in \T^i$. Again, if $a(t) = 1$ then $a(v) = 1$. 

If $a(t) = 0$, combining the first argument for one of the $i$, and the second argument for the other one together imply that $a(v) = 0$.

\medskip

\noindent
{\bf Case 2:} $S_v$ and $S_t$ are intersection-equivalent. The idea of the proof is to manage components of $S_v$ on the small side of $S_t$ using \Cref{lem_bubble_tangle_monot}, and components on the big side using \Cref{lem_braid_equi}, the idea for this case is illustrated in \Cref{pic_lem_15_ex}. For this, let $C(S)$ denote the set of connected components of a topological space $S$.

Let $\mathcal{I} = C( \mathring{B^i_t} \cap S_v)$ denote the connected components of $S_v \smallsetminus S_t$ within $B^i_t$. Since $S_t$ and $S_v$ are intersection-equivalent, there is a natural injection $\psi : \mathcal{I} \hookrightarrow C (S_t \smallsetminus S_v)$ mapping $I \in \mathcal{I}$ to a connected component of $S_t \smallsetminus S_v$ to which $I$ is isotopic while keeping $S_t \cap S_v$ fixed. 
Since there is no critical time between $t$ and $v$, $\mathcal{I}$ and $\psi (\mathcal{I})$ have the same number of intersections with $M$. Hence, $S'_v = (S_v \smallsetminus \psi(\mathcal{I})) \cup \mathcal{I}$ is a sphere ($S_t$ and $S'_v$ are isotopic via an isotopy that keeps $(B^i_t)^c \cap S_v$ fixed) such that $|S'_v \cap M^i| < k$ and $S'_v \subset B^i_t$. Hence, by \Cref{lem_bubble_tangle_monot}, $A^i \in \T^i$ where $A^i$ is the side of $S'_v$ included in $B^i_t$.

It remains to handle $C( S_v \smallsetminus B^i_t )$, the connected components of $S_v \smallsetminus S_t$ outside of $B^i_t$. We slightly push $S'_v$ into $\mathring{A^i}$ making $S'_v$ disjoint from $S_v$, and once again since there is no critical time between $t$ and $v$, we know that $S'_v$ and $S_v$ are braid-equivalent. Let $B^i$ be the side of $S_v$ containing $A^i$, by \Cref{lem_braid_equi}, $B^i \in \T^i$. 

\begin{figure}[ht]
\begin{center}
\begin{tikzpicture}[scale = 1.1]
\clip (-1.5,1.70) rectangle (11.75,-1.75);
\def\aa{95} \coordinate (sa) at (\aa:1);
\def\ab{135} \coordinate (sb) at (\ab:1);
\def\ac{-155} \coordinate (sc) at (\ac:0.85);
\def\ad{-115} \coordinate (sd) at (\ad:0.85);
\def\ae{-15} \coordinate (se) at (\ae:1);
\def\af{15} \coordinate (sf) at (\af:1);

\fill [fill opacity = 0.2, red!50!purple] (sa).. controls +(\aa:1) and \control{(sb)}{(\ab:1)} .. controls +(\ab+180:0.75) and \control{(sc)}{(\ac+180:0.75)} .. controls +(\ac:1) and \control{(sd)}{(\ad:1)} .. controls +(\ad+180:0.75) and \control{(se)}{(\ae+180:0.75)} .. controls +(\ae:1.5) and \control{(sf)}{(\af:1.5)} .. controls +(\af+180:0.75) and \control{(sa)}{(\aa+180:0.75)};
\draw [dotted, thick, red!50!purple] (sb) .. controls +(\ab+180:0.75) and \control{(sc)}{(\ac+180:0.75)} (sd) .. controls +(\ad+180:0.75) and \control{(se)}{(\ae+180:0.75)} (sf) .. controls +(\af+180:0.75) and \control{(sa)}{(\aa+180:0.75)};
\draw [thick, red!50!purple] (sa).. controls +(\aa:1) and \control{(sb)}{(\ab:1)} (sc) .. controls +(\ac:1) and \control{(sd)}{(\ad:1)} (se) .. controls +(\ae:1.5) and \control{(sf)}{(\af:1.5)};

\begin{scope}[even odd rule]
\clip (-2,2) rectangle (14,-2) (sa).. controls +(\aa:1) and \control{(sb)}{(\ab:1)} .. controls +(\ab+180:0.75) and \control{(sc)}{(\ac+180:0.75)} .. controls +(\ac:1) and \control{(sd)}{(\ad:1)} .. controls +(\ad+180:0.75) and \control{(se)}{(\ae+180:0.75)} .. controls +(\ae:1) and \control{(sf)}{(\af:1)} .. controls +(\af+180:0.75) and \control{(sa)}{(\aa+180:0.75)};
\draw [blue, thick] (0,0) circle (1.25);
\end{scope}
\begin{scope}
\clip (sa).. controls +(\aa:1) and \control{(sb)}{(\ab:1)} .. controls +(\ab+180:0.75) and \control{(sc)}{(\ac+180:0.75)} .. controls +(\ac:1) and \control{(sd)}{(\ad:1)} .. controls +(\ad+180:0.75) and \control{(se)}{(\ae+180:0.75)} .. controls +(\ae:1) and \control{(sf)}{(\af:1)} .. controls +(\af+180:0.75) and \control{(sa)}{(\aa+180:0.75)};
\draw [blue, thick, dotted] (0,0) circle (1.25);
\end{scope}
\fill [blue,opacity = 0.2] (0,0) circle (1.25);

\draw [thick, dashed, blue!40!red] (sa) .. controls +(\aa+20:0.25) and \control{(sb)}{(\ab-20:0.25)} .. controls + (\ab-160:0.25) and \control{(sa)}{(\aa+160:0.25)};
\draw [thick, dashed, blue!40!red] (sc) .. controls +(\ac+20:0.25) and \control{(sd)}{(\ad-20:0.25)} .. controls + (\ad-160:0.25) and \control{(sc)}{(\ac+160:0.25)};
\draw [thick, dashed, blue!40!red] (se) .. controls +(\ae+20:0.25) and \control{(sf)}{(\af-20:0.25)} .. controls + (\af-160:0.25) and \control{(se)}{(\ae+160:0.25)};

\draw [blue, -Stealth] (45:1.25) -- +(45:0.5);

\begin{scope}[xshift = 5cm]
\def\aa{95} \coordinate (sa) at (\aa:1);
\def\ab{135} \coordinate (sb) at (\ab:1);
\def\ac{-155} \coordinate (sc) at (\ac:0.85);
\def\ad{-115} \coordinate (sd) at (\ad:0.85);
\def\ae{-15} \coordinate (se) at (\ae:1);
\def\af{15} \coordinate (sf) at (\af:1);

\fill [fill opacity = 0.1, red!50!purple] (sa).. controls +(\aa+20:0.25) and \control{(sb)}{(\ab-20:0.25)} .. controls +(\ab+180:0.75) and \control{(sc)}{(\ac+180:0.75)} .. controls +(\ac+20:0.25) and \control{(sd)}{(\ad-20:0.25)} .. controls +(\ad+180:0.75) and \control{(se)}{(\ae+180:0.75)} .. controls +(\ae+20:0.25) and \control{(sf)}{(\af-20:0.25)} .. controls +(\af+180:0.75) and \control{(sa)}{(\aa+180:0.75)};
\fill [fill opacity = 0.1, red!50!purple] (sa).. controls +(\aa+160:0.25) and \control{(sb)}{(\ab-160:0.25)} .. controls +(\ab+180:0.75) and \control{(sc)}{(\ac+180:0.75)} .. controls +(\ac+160:0.25) and \control{(sd)}{(\ad-160:0.25)} .. controls +(\ad+180:0.75) and \control{(se)}{(\ae+180:0.75)} .. controls +(\ae+160:0.25) and \control{(sf)}{(\af-160:0.25)} .. controls +(\af+180:0.75) and \control{(sa)}{(\aa+180:0.75)};
\draw [thick, red!50!purple] (sb) .. controls +(\ab+180:0.75) and \control{(sc)}{(\ac+180:0.75)};
\draw [thick, red!50!purple] (sd) .. controls +(\ad+180:0.75) and \control{(se)}{(\ae+180:0.75)};
\draw [thick, red!50!purple] (sf) .. controls +(\af+180:0.75) and \control{(sa)}{(\aa+180:0.75)} node (na) [pos = 0.4] {};
\coordinate (ca) at (na);

\fill [opacity = 0.1, blue] (sa) .. controls +(\aa+20:0.25) and \control{(sb)}{(\ab-20:0.25)} .. controls + (\ab-160:0.25) and \control{(sa)}{(\aa+160:0.25)};
\fill [opacity = 0.1, blue] (sc) .. controls +(\ac+20:0.25) and \control{(sd)}{(\ad-20:0.25)} .. controls + (\ad-160:0.25) and \control{(sc)}{(\ac+160:0.25)};
\fill [opacity = 0.1, blue] (se) .. controls +(\ae+20:0.25) and \control{(sf)}{(\af-20:0.25)} .. controls + (\af-160:0.25) and \control{(se)}{(\ae+160:0.25)};
\draw [blue!40!red] (sa) .. controls +(\aa+20:0.25) and \control{(sb)}{(\ab-20:0.25)} .. controls + (\ab-160:0.25) and \control{(sa)}{(\aa+160:0.25)};
\draw [blue!40!red] (sc) .. controls +(\ac+20:0.25) and \control{(sd)}{(\ad-20:0.25)} .. controls + (\ad-160:0.25) and \control{(sc)}{(\ac+160:0.25)};
\draw [blue!40!red] (se) .. controls +(\ae+20:0.25) and \control{(sf)}{(\af-20:0.25)} .. controls + (\af-160:0.25) and \control{(se)}{(\ae+160:0.25)};

\draw [red!50!purple, -Stealth] (ca) -- +(45:0.5);
\end{scope}

\begin{scope}[xshift = 9cm]
\def\aa{100} \coordinate (sa) at (\aa:1);
\def\ab{130} \coordinate (sb) at (\ab:1);
\def\ac{-150} \coordinate (sc) at (\ac:0.85);
\def\ad{-120} \coordinate (sd) at (\ad:0.85);
\def\ae{-10} \coordinate (se) at (\ae:1);
\def\af{10} \coordinate (sf) at (\af:1);

\fill [fill opacity = 0.1, red!50!purple] (sa).. controls +(\aa+20:0.25) and \control{(sb)}{(\ab-20:0.25)} .. controls +(\ab+180:0.8) and \control{(sc)}{(\ac+180:0.8)} .. controls +(\ac+20:0.25) and \control{(sd)}{(\ad-20:0.25)} .. controls +(\ad+180:0.8) and \control{(se)}{(\ae+180:0.8)} .. controls +(\ae+20:0.25) and \control{(sf)}{(\af-20:0.25)} .. controls +(\af+180:0.8) and \control{(sa)}{(\aa+180:0.8)};
\fill [fill opacity = 0.1, red!50!purple] (sa).. controls +(\aa+160:0.25) and \control{(sb)}{(\ab-160:0.25)} .. controls +(\ab+180:0.8) and \control{(sc)}{(\ac+180:0.8)} .. controls +(\ac+160:0.25) and \control{(sd)}{(\ad-160:0.25)} .. controls +(\ad+180:0.8) and \control{(se)}{(\ae+180:0.8)} .. controls +(\ae+160:0.25) and \control{(sf)}{(\af-160:0.25)} .. controls +(\af+180:0.8) and \control{(sa)}{(\aa+180:0.8)};
\begin{scope}[densely dotted]
\draw [thick, red!50!purple] (sb) .. controls +(\ab+180:0.8) and \control{(sc)}{(\ac+180:0.8)};
\draw [thick, red!50!purple] (sd) .. controls +(\ad+180:0.8) and \control{(se)}{(\ae+180:0.8)};
\draw [thick, red!50!purple] (sf) .. controls +(\af+180:0.8) and \control{(sa)}{(\aa+180:0.8)} node (na) [pos = 0.4] {};
\coordinate (ca) at (na);
\end{scope}

\begin{scope}[densely dotted]
\draw [blue!40!red] (sa) .. controls +(\aa+20:0.25) and \control{(sb)}{(\ab-20:0.25)} .. controls + (\ab-160:0.25) and \control{(sa)}{(\aa+160:0.25)};
\draw [blue!40!red] (sc) .. controls +(\ac+20:0.25) and \control{(sd)}{(\ad-20:0.25)} .. controls + (\ad-160:0.25) and \control{(sc)}{(\ac+160:0.25)};
\draw [blue!40!red] (se) .. controls +(\ae+20:0.25) and \control{(sf)}{(\af-20:0.25)} .. controls + (\af-160:0.25) and \control{(se)}{(\ae+160:0.25)};
\end{scope}

\fill [opacity = 0.1, blue] (sa) .. controls +(\aa+20:0.25) and \control{(sb)}{(\ab-20:0.25)} .. controls + (\ab-160:0.25) and \control{(sa)}{(\aa+160:0.25)};
\fill [opacity = 0.1, blue] (sc) .. controls +(\ac+20:0.25) and \control{(sd)}{(\ad-20:0.25)} .. controls + (\ad-160:0.25) and \control{(sc)}{(\ac+160:0.25)};
\fill [opacity = 0.1, blue] (se) .. controls +(\ae+20:0.25) and \control{(sf)}{(\af-20:0.25)} .. controls + (\af-160:0.25) and \control{(se)}{(\ae+160:0.25)};

\end{scope}

\begin{scope}[xshift=9cm]
\def\aa{95} \coordinate (sa) at (\aa:1);
\def\ab{135} \coordinate (sb) at (\ab:1);
\def\ac{-155} \coordinate (sc) at (\ac:0.85);
\def\ad{-115} \coordinate (sd) at (\ad:0.85);
\def\ae{-15} \coordinate (se) at (\ae:1);
\def\af{15} \coordinate (sf) at (\af:1);

\filldraw [fill opacity = 0.2, thick, red!50!purple] (sa).. controls +(\aa:1) and \control{(sb)}{(\ab:1)} .. controls +(\ab+180:0.75) and \control{(sc)}{(\ac+180:0.75)} .. controls +(\ac:1) and \control{(sd)}{(\ad:1)} .. controls +(\ad+180:0.75) and \control{(se)}{(\ae+180:0.75)} .. controls +(\ae:1.5) and \control{(sf)}{(\af:1.5)} .. controls +(\af+180:0.75) and \control{(sa)}{(\aa+180:0.75)}  node (na) [pos = 0.5] {};
\coordinate (ca) at (na);
\draw [red!50!purple, -Stealth] (ca) -- +(45:0.5);
\end{scope}

\node at (-40:1.5) [blue] {$S_t$};
\node at (-17.5:1.85) [red!40!purple] {$S_v$};
\node at ($(5,0)+(-20:1.35)$) [red!40!purple] {$S'_v$};
\node at ($(9,0)+(-17.5:1.85)$) [red!40!purple] {$S_v$};
\node at (3,0) {\Large $\Rightarrow$};
\node at (3,0.15) [above] {Lemma 4.4};
\node at (7.25,0) {\Large $\Rightarrow$};
\node at (7.25,0.15) [above] {Lemma 4.3};
\end{tikzpicture}
\caption{The side of sphere in $\T^i$ is indicated by an arrow pointing outwards.}
\label{pic_lem_15_ex}
\end{center}
\end{figure}

In the above, if $a(t) = 1$, all the small sides coincide and hence we have $a(v) =1$. Otherwise, the closed balls $A^i \in \T^i$, $i \in \{1,2\}$, lie on different sides of $S_t$ throughout the construction. Then, each closed ball $B^i_t \in \T^i$ contains $A^i$, and it follows that $B^1_v \not = B^2_v$ implying $a(v)=0$.
\end{proof}

\begin{lemma}\label{lem_gen_orient_local_bis}
Let $t \in [0,1]$. Then there exists a neighbourhood $V \subset [0,1]$ of $t$ such that $a$ is constant on $V \smallsetminus \{t\}$.
\end{lemma}

\begin{proof}[Proof of \Cref{lem_gen_orient_local_bis}]
If $t$ is non-critical, the statement holds by \Cref{lem_gen_orient_local}. Otherwise, assume that $t = t_j$ for some $1\leq j \leq s$. Let $V$ be an open connected neighbourhood of $t$ in $[0,1]$ that does not connect the other critical times. Up to intersection with a neighbourhood provided by \Cref{lem_gen_close_enough}, we assume that the spheres $S_v$ for $v \in V$ are either disjoint or intersection-equivalent to $S_t$. We want to show that $a(u)= a(v)$ where $u$ and $v$ are in disjoint components of $V \smallsetminus t$.

Without loss of generality, let $p$ be the point of $M^1$ tangent to $S_t$ at time $t$ (we assume that $p \in M^1$ since the knots play symmetrical roles). Let $x$ be the point of $S_u$ such that $\Theta(x, t) =p$. Denote by $P$ the path followed by $x$ during the sweep-out by the spheres between $u$ and $v$, that is, $P = \Theta (x, [u,v])$. Cover $P$ by a closed ball $B$ that intersects $S_u$ and $S_v$ on a single disc each. This ball is $M^1$-trivial. This set-up is illustrated in the left and middle parts of \Cref{pic_lem_16_ex}.

Notice that the knot $M^2$ presents no tangency with $S_t$ in between $S_u$ and $S_v$. So \Cref{lem_gen_orient_local} applies to $\T^2$ and we remark that $p$ must be on different sides of $M^2$ with respect to $\T^2$, i.e., that $p \in B^2_u$ if and only if $p \not \in B^2_v$, because the tangency at $t$ stems from (the lift of) a RII move. Let us show that similarly, $p$ must be on different sides of $M^1$ with respect to $\T^1$ (see right part of \Cref{pic_lem_16_ex}).

\begin{figure}[ht]
\begin{center}
\begin{tikzpicture}[scale = 1.2]
\clip (-1.4,1.75) rectangle (11.15,-1.4);
\def\aa{95} \coordinate (sa) at (\aa:0.95);
\def\ab{135} \coordinate (sb) at (\ab:0.95);
\def\ac{-155} \coordinate (sc) at (\ac:0.85);
\def\ad{-115} \coordinate (sd) at (\ad:0.85);
\def\ae{-15} \coordinate (se) at (\ae:0.9);
\def\af{15} \coordinate (sf) at (\af:0.9);
\def\ag{80} \coordinate (sg) at (\ag:1);
\def\ah{150} \coordinate (sh) at (\ah:1);
\def\ai{-170} \coordinate (si) at (\ai:1);
\def\aj{-100} \coordinate (sj) at (\aj:1);
\def\ak{-30} \coordinate (sk) at (\ak:1);
\def\al{30} \coordinate (sl) at (\al:1);

\begin{scope}[even odd rule]
\clip (-2,2) rectangle (14,-2) (sg).. controls +(\ag:1) and \control{(sh)}{(\ab:1)} .. controls +(\ah+180:0.35) and \control{(si)}{(\ai+180:0.35)} .. controls +(\ai:1) and \control{(sj)}{(\aj:1)} .. controls +(\aj+180:0.35) and \control{(sk)}{(\ak+180:0.35)} .. controls +(\ak:1) and \control{(sl)}{(\al:1)} .. controls +(\al+180:0.35) and \control{(sg)}{(\ag+180:0.35)};
\draw [blue, thick] (0,0) circle (1.25);
\end{scope}
\begin{scope}
\clip (sg).. controls +(\ag:1) and \control{(sh)}{(\ab:1)} .. controls +(\ah+180:0.35) and \control{(si)}{(\ai+180:0.35)} .. controls +(\ai:1) and \control{(sj)}{(\aj:1)} .. controls +(\aj+180:0.35) and \control{(sk)}{(\ak+180:0.35)} .. controls +(\ak:1) and \control{(sl)}{(\al:1)} .. controls +(\al+180:0.35) and \control{(sg)}{(\ag+180:0.35)};
\draw [blue, thick, dotted] (0,0) circle (1.25);
\end{scope}
\fill [blue,opacity = 0.1] (0,0) circle (1.25);

\fill [red!50!blue, opacity = 0.1] (sg).. controls +(\ag:1) and \control{(sh)}{(\ab:1)} .. controls +(\ah+180:0.35) and \control{(si)}{(\ai+180:0.35)} .. controls +(\ai:1) and \control{(sj)}{(\aj:1)} .. controls +(\aj+180:0.35) and \control{(sk)}{(\ak+180:0.35)} .. controls +(\ak:1) and \control{(sl)}{(\al:1)} .. controls +(\al+180:0.35) and \control{(sg)}{(\ag+180:0.35)};
\begin{scope}[even odd rule]
\clip (-2,2) rectangle (14,-2) (sa).. controls +(\aa:1.35) and \control{(sb)}{(\ab:1.35)} .. controls +(\ab+180:0.75) and \control{(sc)}{(\ac+180:0.75)} .. controls +(\ac:1.25) and \control{(sd)}{(\ad:1.25)} .. controls +(\ad+180:0.75) and \control{(se)}{(\ae+180:0.75)} .. controls +(\ae:1.65) and \control{(sf)}{(\af:1.65)} .. controls +(\af+180:0.75) and \control{(sa)}{(\aa+180:0.75)};
\draw [red!50!blue,thick] (sg).. controls +(\ag:1) and \control{(sh)}{(\ab:1)} (si) .. controls +(\ai:1) and \control{(sj)}{(\aj:1)} (sk) .. controls +(\ak:1) and \control{(sl)}{(\al:1)};
\end{scope}
\begin{scope}
\clip (sa).. controls +(\aa:1.35) and \control{(sb)}{(\ab:1.35)} .. controls +(\ab+180:0.75) and \control{(sc)}{(\ac+180:0.75)} .. controls +(\ac:1.25) and \control{(sd)}{(\ad:1.25)} .. controls +(\ad+180:0.75) and \control{(se)}{(\ae+180:0.75)} .. controls +(\ae:1.65) and \control{(sf)}{(\af:1.65)} .. controls +(\af+180:0.75) and \control{(sa)}{(\aa+180:0.75)};
\draw [red!50!blue, thick, dotted] (sg).. controls +(\ag:1) and \control{(sh)}{(\ab:1)} (si) .. controls +(\ai:1) and \control{(sj)}{(\aj:1)} (sk) .. controls +(\ak:1) and \control{(sl)}{(\al:1)};
\end{scope}
\draw [red!50!blue, thick, dotted] (sh) .. controls +(\ah+180:0.35) and \control{(si)}{(\ai+180:0.35)} (sj) .. controls +(\aj+180:0.35) and \control{(sk)}{(\ak+180:0.35)} (sl) .. controls +(\al+180:0.35) and \control{(sg)}{(\ag+180:0.35)};
\path (sk) .. controls +(\ak:1) and \control{(sl)}{(\al:1)} node (np) [midway] {};
\coordinate (cp) at (np);

\def\ec{1.35 cm}
\fill [fill opacity = 0.1, red] (sa).. controls +(\aa:1.35) and \control{(sb)}{(\ab:1.35)} .. controls +(\ab+180:0.75) and \control{(sc)}{(\ac+180:0.75)} .. controls +(\ac:1.25) and \control{(sd)}{(\ad:1.25)} .. controls +(\ad+180:0.75) and \control{(se)}{(\ae+180:0.75)} .. controls +(\ae:1.65) and \control{(sf)}{(\af:1.65)} .. controls +(\af+180:0.75) and \control{(sa)}{(\aa+180:0.75)};
\path [name path = redsphere] (sa).. controls +(\aa:1.35) and \control{(sb)}{(\ab:1.35)} .. controls +(\ab+180:0.75) and \control{(sc)}{(\ac+180:0.75)} .. controls +(\ac:1.25) and \control{(sd)}{(\ad:1.25)} .. controls +(\ad+180:0.75) and \control{(se)}{(\ae+180:0.75)} .. controls +(\ae:1.65) and \control{(sf)}{(\af:1.65)} .. controls +(\af+180:0.75) and \control{(sa)}{(\aa+180:0.75)};
\path [name path = bluesphere] (0,0) circle (\ec);
\path [name intersections={of=redsphere and bluesphere,total=\tot}]
\foreach \s in {1,...,\tot}{coordinate (i\s) at (intersection-\s)};

\begin{scope}
\clip (0,0) circle (\ec);
\draw [red, dotted] (sa).. controls +(\aa:1.35) and \control{(sb)}{(\ab:1.35)} .. controls +(\ab+180:0.75) and \control{(sc)}{(\ac+180:0.75)} .. controls +(\ac:1.25) and \control{(sd)}{(\ad:1.25)} .. controls +(\ad+180:0.75) and \control{(se)}{(\ae+180:0.75)} .. controls +(\ae:1.65) and \control{(sf)}{(\af:1.65)} .. controls +(\af+180:0.75) and \control{(sa)}{(\aa+180:0.75)};
\end{scope}
\begin{scope}[even odd rule]
\clip (0,0) circle (\ec) (-2,2) rectangle (14,-2);
\draw [red, thick] (sa).. controls +(\aa:1.35) and \control{(sb)}{(\ab:1.35)} .. controls +(\ab+180:0.75) and \control{(sc)}{(\ac+180:0.75)} .. controls +(\ac:1.25) and \control{(sd)}{(\ad:1.25)} .. controls +(\ad+180:0.75) and \control{(se)}{(\ae+180:0.75)} .. controls +(\ae:1.65) and \control{(sf)}{(\af:1.65)} .. controls +(\af+180:0.75) and \control{(sa)}{(\aa+180:0.75)};
\end{scope}

\draw [thick, dashed, red] (sa) .. controls +(\aa+20:0.2) and \control{(sb)}{(\ab-20:0.25)} .. controls + (\ab-160:0.25) and \control{(sa)}{(\aa+160:0.25)};
\draw [thick, dashed, red] (sc) .. controls +(\ac+20:0.2) and \control{(sd)}{(\ad-20:0.25)} .. controls + (\ad-160:0.25) and \control{(sc)}{(\ac+160:0.25)};
\draw [thick, dashed, red] (se) .. controls +(\ae+20:0.2) and \control{(sf)}{(\af-20:0.2)} .. controls + (\af-160:0.2) and \control{(se)}{(\ae+160:0.2)};

\draw [thick, dashed, red!50!blue] (sg) .. controls +(\ag+20:0.4) and \control{(sh)}{(\ah-20:0.4)} .. controls + (\ah-160:0.3) and \control{(sg)}{(\ag+160:0.3)};
\draw [thick, dashed, red!50!blue] (si) .. controls +(\ai+20:0.3) and \control{(sj)}{(\aj-20:0.3)} .. controls + (\aj-160:0.5) and \control{(si)}{(\ai+160:0.5)};
\draw [thick, dashed, red!50!blue] (sk) .. controls +(\ak+20:0.3) and \control{(sl)}{(\al-20:0.3)} .. controls + (\al-160:0.4) and \control{(sk)}{(\ak+160:0.4)};

\coordinate (cp2) at ($(cp)+(0.4,-0.45)$);
\coordinate (cp3) at ($(cp)!0.4!(cp2)$);
\coordinate (cp4) at ($(cp)+(0.1,0.375)$);
\draw (2.25,-0.5) -- (cp2) -- (cp3);
\draw [opacity = 0.2] (cp3) -- (cp);
\begin{scope}
\clip (sa).. controls +(\aa:1.35) and \control{(sb)}{(\ab:1.35)} .. controls +(\ab+180:0.75) and \control{(sc)}{(\ac+180:0.75)} .. controls +(\ac:1.25) and \control{(sd)}{(\ad:1.25)} .. controls +(\ad+180:0.75) and \control{(se)}{(\ae+180:0.75)} .. controls +(\ae:1.65) and \control{(sf)}{(\af:1.65)} .. controls +(\af+180:0.75) and \control{(sa)}{(\aa+180:0.75)};
\draw [opacity = 0.2] (cp4) -- (cp);
\end{scope}
\begin{scope}[even odd rule]
\clip (-2,2) rectangle (14,-2) (sa).. controls +(\aa:1.35) and \control{(sb)}{(\ab:1.35)} .. controls +(\ab+180:0.75) and \control{(sc)}{(\ac+180:0.75)} .. controls +(\ac:1.25) and \control{(sd)}{(\ad:1.25)} .. controls +(\ad+180:0.75) and \control{(se)}{(\ae+180:0.75)} .. controls +(\ae:1.65) and \control{(sf)}{(\af:1.65)} .. controls +(\af+180:0.75) and \control{(sa)}{(\aa+180:0.75)};
\draw (cp) -- (cp4) --(2.25,0.5);
\end{scope}

\fill (cp) circle (0.05cm);
\node at (cp) [right] {$p$};
\node [blue] at (-50:1.5) {$S_u$};
\node [red!80!black] at (0:2.35) {$S_v$};
\node [red!50!blue] at (27.5:1.6) {$S_t$};
\node at (-17.5:2.35) {$M^1$};

\begin{scope}[xshift = 4.5cm]
\def\aa{95} \coordinate (sa) at (\aa:0.95);
\def\ab{135} \coordinate (sb) at (\ab:0.95);
\def\ac{-155} \coordinate (sc) at (\ac:0.85);
\def\ad{-115} \coordinate (sd) at (\ad:0.85);
\def\ae{-15} \coordinate (se) at (\ae:0.9);
\def\af{15} \coordinate (sf) at (\af:0.9);
\def\ag{80} \coordinate (sg) at (\ag:1);
\def\ah{150} \coordinate (sh) at (\ah:1);
\def\ai{-170} \coordinate (si) at (\ai:1);
\def\aj{-100} \coordinate (sj) at (\aj:1);
\def\ak{-30} \coordinate (sk) at (\ak:1);
\def\al{30} \coordinate (sl) at (\al:1);
\coordinate (u) at (10:0.9);
\coordinate (d) at (-10:0.9);
\coordinate (cpl) at (0:0.9);
\coordinate (ud) at ($(10:0.9)+(1.5,0)$);
\coordinate (dd) at ($(-10:0.9)+(1.5,0)$);
\coordinate (cpd) at ($(0:0.9)+(1.5,0)$);

\begin{scope}[even odd rule]
\clip (-2,2) rectangle (14,-2) (sa).. controls +(\aa:1.35) and \control{(sb)}{(\ab:1.35)} .. controls +(\ab+180:0.75) and \control{(sc)}{(\ac+180:0.75)} .. controls +(\ac:1.25) and \control{(sd)}{(\ad:1.25)} .. controls +(\ad+180:0.75) and \control{(se)}{(\ae+180:0.75)} .. controls +(\ae:1.65) and \control{(sf)}{(\af:1.65)} .. controls +(\af+180:0.75) and \control{(sa)}{(\aa+180:0.75)};
\draw [blue, thick] (0,0) circle (1.25);
\end{scope}
\begin{scope}
\clip (sa).. controls +(\aa:1.35) and \control{(sb)}{(\ab:1.35)} .. controls +(\ab+180:0.75) and \control{(sc)}{(\ac+180:0.75)} .. controls +(\ac:1.25) and \control{(sd)}{(\ad:1.25)} .. controls +(\ad+180:0.75) and \control{(se)}{(\ae+180:0.75)} .. controls +(\ae:1.65) and \control{(sf)}{(\af:1.65)} .. controls +(\af+180:0.75) and \control{(sa)}{(\aa+180:0.75)};
\draw [blue, thick, dotted] (0,0) circle (1.25);
\end{scope}
\fill [blue,opacity = 0.1] (0,0) circle (1.25);

\path (sk) .. controls +(\ak:1) and \control{(sl)}{(\al:1)} node (np) [midway] {};
\coordinate (cp) at (np);

\fill [fill opacity = 0.1, red] (sa).. controls +(\aa:1.35) and \control{(sb)}{(\ab:1.35)} .. controls +(\ab+180:0.75) and \control{(sc)}{(\ac+180:0.75)} .. controls +(\ac:1.25) and \control{(sd)}{(\ad:1.25)} .. controls +(\ad+180:0.75) and \control{(se)}{(\ae+180:0.75)} .. controls +(\ae:1.65) and \control{(sf)}{(\af:1.65)} .. controls +(\af+180:0.75) and \control{(sa)}{(\aa+180:0.75)};
\draw [dotted, thick, red] (sb) .. controls +(\ab+180:0.75) and \control{(sc)}{(\ac+180:0.75)} (sd) .. controls +(\ad+180:0.75) and \control{(se)}{(\ae+180:0.75)} (sf) .. controls +(\af+180:0.75) and \control{(sa)}{(\aa+180:0.75)};
\begin{scope}[even odd rule]
\clip (-2,2) rectangle (14,-2) (d) .. controls + (170:0.1) and \control{(u)}{(-170:0.1)} -- (ud) -- (dd) -- (d);
\draw [thick, red] (sa).. controls +(\aa:1.35) and \control{(sb)}{(\ab:1.35)}(sc) .. controls +(\ac:1.25) and \control{(sd)}{(\ad:1.25)} (se) .. controls +(\ae:1.65) and \control{(sf)}{(\af:1.65)};
\end{scope}

\draw [thick, dashed, red] (sa) .. controls +(\aa+20:0.2) and \control{(sb)}{(\ab-20:0.25)} .. controls + (\ab-160:0.25) and \control{(sa)}{(\aa+160:0.25)};
\draw [thick, dashed, red] (sc) .. controls +(\ac+20:0.2) and \control{(sd)}{(\ad-20:0.25)} .. controls + (\ad-160:0.25) and \control{(sc)}{(\ac+160:0.25)};
\draw [thick, dashed, red] (se) .. controls +(\ae+20:0.2) and \control{(sf)}{(\af-20:0.2)} .. controls + (\af-160:0.2) and \control{(se)}{(\ae+160:0.2)};

\begin{scope}[brown]
\clip (sa).. controls +(\aa:1.35) and \control{(sb)}{(\ab:1.35)} .. controls +(\ab+180:0.75) and \control{(sc)}{(\ac+180:0.75)} .. controls +(\ac:1.25) and \control{(sd)}{(\ad:1.25)} .. controls +(\ad+180:0.75) and \control{(se)}{(\ae+180:0.75)} .. controls +(\ae:1.65) and \control{(sf)}{(\af:1.65)} .. controls +(\af+180:0.75) and \control{(sa)}{(\aa+180:0.75)};
\draw [opacity = 0.35] (u) .. controls +(-10:0.1) and \control{(d)}{(10:0.1)};
\draw [opacity = 0.75] (d) .. controls + (170:0.1) and \control{(u)}{(-170:0.1)} -- (ud) -- (dd) -- (d);
\fill [opacity = 0.2] (d) .. controls + (170:0.1) and \control{(u)}{(-170:0.1)} -- (ud) -- (dd) -- (d);
\draw [black, opacity = 0.8] (cpl) -- (cpd);
\end{scope}
\begin{scope}
\clip (d) .. controls + (170:0.1) and \control{(u)}{(-170:0.1)} -- (ud) -- (dd) -- (d);
\draw [brown, thick] (sa).. controls +(\aa:1.35) and \control{(sb)}{(\ab:1.35)}(sc) .. controls +(\ac:1.25) and \control{(sd)}{(\ad:1.25)} (se) .. controls +(\ae:1.65) and \control{(sf)}{(\af:1.65)};
\end{scope}
\fill (cp) circle (0.05cm);

\path [name path = brownsphere] (d) .. controls + (170:0.1) and \control{(u)}{(-170:0.1)} -- (ud) -- (dd) -- (d);
\path [name path = redsphere] (sa).. controls +(\aa:1.35) and \control{(sb)}{(\ab:1.35)} .. controls +(\ab+180:0.75) and \control{(sc)}{(\ac+180:0.75)} .. controls +(\ac:1.25) and \control{(sd)}{(\ad:1.25)} .. controls +(\ad+180:0.75) and \control{(se)}{(\ae+180:0.75)} .. controls +(\ae:1.65) and \control{(sf)}{(\af:1.65)} .. controls +(\af+180:0.75) and \control{(sa)}{(\aa+180:0.75)};
\path [name intersections={of=redsphere and brownsphere,total=\tot}]
\foreach \s in {1,...,\tot}{coordinate (i\s) at (intersection-\s)};
\draw [opacity = 0.75, brown] (i1) .. controls + (170:0.1) and \control{(i2)}{(-170:0.1)};
\draw [opacity = 0.75, brown] (i2) .. controls +(-10:0.05) and \control{(i1)}{(10:0.05)};

\node [red!80!black] at (90:1.55) {$S_v$};
\node [blue] at (-50:1.5) {$S_u$};
\node [brown] at (-10:2.25) {$B$};
\node at (2.5:2.25) {$P$};
\end{scope}

\begin{scope}[xshift = 9cm]
\def\aa{95} \coordinate (sa) at (\aa:0.95);
\def\ab{135} \coordinate (sb) at (\ab:0.95);
\def\ac{-155} \coordinate (sc) at (\ac:0.85);
\def\ad{-115} \coordinate (sd) at (\ad:0.85);
\def\ae{-15} \coordinate (se) at (\ae:0.9);
\def\af{15} \coordinate (sf) at (\af:0.9);
\def\ag{80} \coordinate (sg) at (\ag:1);
\def\ah{150} \coordinate (sh) at (\ah:1);
\def\ai{-170} \coordinate (si) at (\ai:1);
\def\aj{-100} \coordinate (sj) at (\aj:1);
\def\ak{-30} \coordinate (sk) at (\ak:1);
\def\al{30} \coordinate (sl) at (\al:1);
\coordinate (u) at (10:0.9);
\coordinate (d) at (-10:0.9);
\coordinate (cpl) at (0:0.9);
\coordinate (ud) at ($(10:0.9)+(1.5,0)$);
\coordinate (dd) at ($(-10:0.9)+(1.5,0)$);
\coordinate (cpd) at ($(0:0.9)+(1.5,0)$);

\begin{scope}[even odd rule]
\clip (-2,2) rectangle (14,-2) (sa).. controls +(\aa:1.35) and \control{(sb)}{(\ab:1.35)} .. controls +(\ab+180:0.75) and \control{(sc)}{(\ac+180:0.75)} .. controls +(\ac:1.25) and \control{(sd)}{(\ad:1.25)} .. controls +(\ad+180:0.75) and \control{(se)}{(\ae+180:0.75)} .. controls +(\ae:1.65) and \control{(sf)}{(\af:1.65)} .. controls +(\af+180:0.75) and \control{(sa)}{(\aa+180:0.75)};
\draw [blue, thick] (0,0) circle (1.25);
\end{scope}
\begin{scope}
\clip (sa).. controls +(\aa:1.35) and \control{(sb)}{(\ab:1.35)} .. controls +(\ab+180:0.75) and \control{(sc)}{(\ac+180:0.75)} .. controls +(\ac:1.25) and \control{(sd)}{(\ad:1.25)} .. controls +(\ad+180:0.75) and \control{(se)}{(\ae+180:0.75)} .. controls +(\ae:1.65) and \control{(sf)}{(\af:1.65)} .. controls +(\af+180:0.75) and \control{(sa)}{(\aa+180:0.75)};
\draw [blue, thick, dotted] (0,0) circle (1.25);
\end{scope}
\fill [blue,opacity = 0.1] (0,0) circle (1.25);

\path (sk) .. controls +(\ak:1) and \control{(sl)}{(\al:1)} node (np) [midway] {};
\coordinate (cp) at (np);

\fill [fill opacity = 0.1, red] (sa).. controls +(\aa:1.35) and \control{(sb)}{(\ab:1.35)} .. controls +(\ab+180:0.75) and \control{(sc)}{(\ac+180:0.75)} .. controls +(\ac:1.25) and \control{(sd)}{(\ad:1.25)} .. controls +(\ad+180:0.75) and \control{(se)}{(\ae+180:0.75)} .. controls +(\ae:1.65) and \control{(sf)}{(\af:1.65)} .. controls +(\af+180:0.75) and \control{(sa)}{(\aa+180:0.75)};
\draw [dotted, thick, red] (sb) .. controls +(\ab+180:0.75) and \control{(sc)}{(\ac+180:0.75)} (sd) .. controls +(\ad+180:0.75) and \control{(se)}{(\ae+180:0.75)} (sf) .. controls +(\af+180:0.75) and \control{(sa)}{(\aa+180:0.75)};
\draw [thick, red] (sa).. controls +(\aa:1.35) and \control{(sb)}{(\ab:1.35)}(sc) .. controls +(\ac:1.25) and \control{(sd)}{(\ad:1.25)} (se) .. controls +(\ae:1.65) and \control{(sf)}{(\af:1.65)};

\draw [thick, dashed, red] (sa) .. controls +(\aa+20:0.2) and \control{(sb)}{(\ab-20:0.25)} .. controls + (\ab-160:0.25) and \control{(sa)}{(\aa+160:0.25)};
\draw [thick, dashed, red] (sc) .. controls +(\ac+20:0.2) and \control{(sd)}{(\ad-20:0.25)} .. controls + (\ad-160:0.25) and \control{(sc)}{(\ac+160:0.25)};
\draw [thick, dashed, red] (se) .. controls +(\ae+20:0.2) and \control{(sf)}{(\af-20:0.2)} .. controls + (\af-160:0.2) and \control{(se)}{(\ae+160:0.2)};

\node [red!80!black] at (90:1.55) {$S_v$};
\node [blue] at (-50:1.5) {$S_u$};
\fill (cp) circle (0.05cm);
\node at (cp) [right] {$p$};

\draw [blue, -Stealth] (45:1.25) -- +(45:0.5);
\path (se) .. controls +(\ae:1.65) and \control{(sf)}{(\af:1.65)} node [pos = 0.8] (na) {};
\coordinate (ca) at (na);
\draw [red, -Stealth] (ca) -- +(45:0.5);

\end{scope}
\end{tikzpicture}
\caption{All spheres are braid-equivalent with respect to $M^2$. Left: $p$ is the tangent point between $M^1$ and $S_t$. Middle: definition of the path $P$ and the ball $B$. Right: $p$ is on different sides of $B^2_u$ and $B^2_v$.}
\label{pic_lem_16_ex}
\end{center}
\end{figure}

\noindent
{\bf Case 1:} $p \not\in B^1_u$. Note that $|\partial(B^1_u \cup B) \cap M^1 | < k$ by construction of $B$ and assumptions on $S_t$. By \ref{def_T3}, $B^1_u \cup B \in \T^1$. Furthermore, by construction, $\partial (B^1_u \cup B)$ and $S_u$ are intersection-equivalent. The methods of \Cref{lem_gen_orient_local} apply and imply that the side of $S_v$ containing $p$ is in $\T^1$.

\noindent
{\bf Case 2:} $p \in B^1_u$. We still have $|\partial(B^1_u \smallsetminus B) \cap M^1 | < k$ by construction of $B$ and the assumptions on $S_t$. By \Cref{lem_bubble_tangle_monot}, $B^1_u \smallsetminus B \in \T^1$. By construction, $\partial (B^1_u \cup B)$ and $S_u$ are intersection-equivalent. The methods of \Cref{lem_gen_orient_local} apply and we can infer that the side of $S_v$ not containing $p$ is in $\T^1$.

Hence, for $i \in \{1,2\}$, $p \in B^i_u \Leftrightarrow p \not \in B^i_v$. This implies that $a(u) = a(v)$, and concludes our proof. 
\end{proof}

\begin{proposition}\label{prop_gen_orient_consistency}
The agreement $a$ is constant on $[0,1] \smallsetminus \{t_j\}_{1 \leq j\leq s}$ and hence $a(0)=a(1)$.
\end{proposition}

\begin{proof}
We can cover $[0,1]$ by open discs from \Cref{lem_gen_orient_local_bis} on which $a$ is constant. Since $[0,1]$ is compact, only finitely many of them are enough to cover it. On each connected component of $[0,1] \smallsetminus \{t_j\}_{1\leq j\leq s}$ the agreement function is constant by the continuity of $a$ ($a$ is locally constant). Moreover, we know from \Cref{lem_gen_orient_local_bis} that for $u < t_j < v$ close enough we have $a(u)=a(v)$. Hence, $a$ is constant on $[0,1] \smallsetminus \{t_j\}_{1 \leq j \leq s}$, and $a(0)=a(1)$. 
\end{proof}

\begin{proof}[Proof of \Cref{th_unlink_analysed}]
The initial discussion of this subsection states that $a(0) = 0$ and $a(1) = 1$. 
This contradicts \Cref{prop_gen_orient_consistency}. Thus, our assumption that $\addcross(\mathcal{D}(n,n+1), R)< \frac{2}{3}n-2$ does not hold. Therefore, during the sequence, at least one diagram has at least $2n^2 + \frac{2}{3}n$ crossings.
\end{proof}

\section{Other families of split links.}\label{appendix_other}
We specified our diagrams $\mathcal{D}_{p,q}$ with $p = n$ and $q = n+1$, but our proof can easily be adapted to handle any coprime $p,q$ and prove that $\mathcal{D}_{p,q}$ is a hard split link. However, our proof provides a lower bound for $\unlink(\mathcal{D}_{p,q})$ that only depends on $\min (p,q)$ while the number of crossings of the diagram is larger than $pq - \min (p,q)$. Hence, the lower bound on crossing-complexity depending on the number of crossings in the initial diagram is highest possible on the diagrams $\mathcal{D}(n,n+1)$.

\Cref{pic_entrelacs_4} shows a similar family of links for which our arguments also apply. In particular, this family of diagrams also contains hard split links with arbitrarily large crossing-complexity.

\begin{figure}[ht]
\begin{center}
\includegraphics[scale=1]{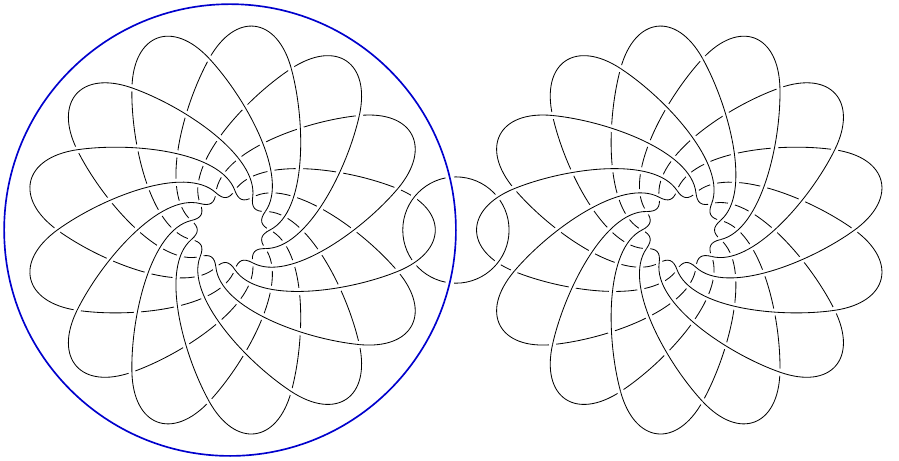}
\caption{Another construction where splitting the blue unknot requires a super-constant number of additional crossings.}
\label{pic_entrelacs_4}
\end{center}
\end{figure}

\subparagraph*{Acknowledgements.}
We would like to thank Cl\'ement Maria for helpful discussions and 
Stephan Tillmann for his comments and suggestions on an earlier version of the paper. 
Moreover, the authors want to thank the anonymous referees for their helpful comments.
Corentin Lunel was partially supported by the ANR project ANR-20-CE48-0007 (AlgoKnot) and was a PhD student at Université Gustave Eiffel for a significant part of this project. 
Jonathan Spreer was supported by the Australian Research Council under the Discovery Project scheme (grant number DP220102588).

\bibliographystyle{plainurl}
\bibliography{biblio}

\appendix

\section{Details of section \ref{sec_homo_iso}}
\label{sec_chambers_appendix}

In this section, we reprove \Cref{prop_iso_U} by following and detailing when necessary the proof of \Cref{Chambers_homo-isotopy} made in \cite[Section 2]{Chambers_homo-isotopy}. We try to follow their notations as much as possible. 

\begin{figure}[ht]
\begin{center}
\begin{tikzpicture}[scale = 0.8]
\draw (0,0) -- +(2,2) (2,0) -- +(-2,2);
\draw [dashed, opacity = 0.5] (1,1) circle (0.5cm);

\draw [-Stealth] (2,1) -- +(1,0);
\draw [-Stealth] (0,1) -- +(-1,0);
\path [opacity = 0.5] (1,1) -- +(0,-0.5) node [below] {$\epsilon$};  

\begin{scope}[xshift = 3cm]
\draw [dashed, opacity = 0.5] (1,1) circle (0.5cm);
\draw (0,0) -- +(0.5,0.5) .. controls (1,1) and (1,1) .. (0.5,1.5) -- (0,2);
\draw (2,0) -- +(-0.5,0.5) .. controls (1,1) and (1,1) .. (1.5,1.5) -- (2,2);
\end{scope}

\begin{scope}[xshift = -3cm]
\draw [dashed, opacity = 0.5] (1,1) circle (0.5cm);
\draw (0,0) -- +(0.5,0.5) .. controls (1,1) and (1,1) .. (1.5,0.5) -- (2,0);
\draw (0,2) -- +(0.5,-0.5) .. controls (1,1) and (1,1) .. (1.5,1.5) -- (2,2);
\end{scope}
\end{tikzpicture}
\caption{The two possible resolutions of a crossings on a ball of radius $\epsilon$.}
\label{pic_reso_def}
\end{center}
\end{figure}

We start with repeating the main ingredients of the proof of \Cref{Chambers_homo-isotopy}. We consider a closed curve $\gamma$ being homotoped between $\gamma_0$ and $\gamma_1$, and we denote the curves of the homotopy by $\gamma_t$ for $t\in [0,1]$. Similarly to the Reidemeister theorem and as described in \cite[Proposition 2.1]{Chambers_homo-isotopy}, such a homotopy can be discretised using projections of Reidemeister moves, which are denoted by R1, R2 and R3 and pictured on the left of \Cref{pic_resolution_moves}. We fix some small $\epsilon$ throughout the proof, and for any time $t_j$, we denote by $\{ G^j_i \}_i$ the set of connected resolutions of $\gamma_{t_j}$ on balls of radii $\epsilon$ (see \Cref{pic_reso_def} for the two possible resolutions of a crossing on a ball of radius $\epsilon$). By definition, all of these curves are $\epsilon$-image equivalent to $\gamma_{t_j}$. The proof of \Cref{Chambers_homo-isotopy} relies on proving that there exists a path between $\gamma_0$ and some $\overline{\gamma_1}$ within a certain graph $\Gamma$ containing resolutions of $\gamma$.

\begin{figure}[ht]
\begin{center}
\begin{tikzpicture}[scale = 0.9]
\def\e{0.05} \def\me{0.8} \def\ie{0.05}
\clip (-0.1,2.1) rectangle (15.6,-7.6);
\node  [circle, inner sep = \e cm] (n1) at (1.75,1) {};
\draw (0,0) -- +(0,2);
\draw (1.5,0) .. controls +(90:\me) and \control{(n1.south west)}{(-135:0.15)} -- (n1.north east) .. controls +(45:0.75) and \control{(n1.south east)}{(-45:0.75)} -- (n1.north west) .. controls +(135:0.15) and \control{(1.5,2)}{(-90:\me)};
\draw [Stealth-Stealth] (0.25,1) -- (1.25,1) node [midway, above] {R1};

\begin{scope}[xshift=7.5cm]
\node  [circle, inner sep = \e cm] (n1) at (1.75,1) {};
\draw (0,0) -- +(0,2);
\draw (1.5,0) .. controls +(90:\me) and \control{(n1.south west)}{(-135:0.15)} .. controls +(45:\ie) and \control{(n1.south east)}{(135:\ie)} .. controls +(-45:0.75) and \control{(n1.north east)}{(45:0.75)} .. controls +(-135:\ie) and \control{(n1.north west)}{(-45:\ie)} .. controls +(135:0.15) and \control{(1.5,2)}{(-90:\me)};
\draw [Stealth-Stealth] (0.25,1) -- (1.25,1) node [midway, above] {M1};
\end{scope}

\begin{scope}[yshift=-2.5cm]
\node  [circle, inner sep = \e cm] (n1) at (3,0.6) {};
\node  [circle, inner sep = \e cm] (n2) at (3,1.4) {};

\draw (0,0) -- ++(0,2);
\draw (1,0) -- ++(0,2);
\draw (2.5,0) .. controls +(90:0.15) and \control{(n1.south west)}{(-135:0.25)} -- (n1.north east) .. controls +(45:0.25) and \control{(n2.south east)}{(-45:0.25)} -- (n2.north west) .. controls +(135:0.25) and \control{(2.5,2)}{(-90:0.15)}; 
\draw (3.5,0) .. controls +(90:0.15) and \control{(n1.south east)}{(-45:0.25)} -- (n1.north west) .. controls +(135:0.25) and \control{(n2.south west)}{(-135:0.25)} -- (n2.north east) .. controls +(45:0.25) and \control{(3.5,2)}{(-90:0.15)}; 
\draw [Stealth-Stealth] (1.25,1) -- (2.25,1) node [midway, above] {R2};

\begin{scope}[xshift=7.5cm]
\node  [circle, inner sep = \e cm] (n1) at (3,0.6) {};
\node  [circle, inner sep = \e cm] (n2) at (3,1.4) {};

\draw (0,0) -- ++(0,2);
\draw (1,0) -- ++(0,2);
\draw (2.5,0) .. controls +(90:0.15) and \control{(n1.south west)}{(-135:0.25)}  .. controls +(45:\ie) and \control{(n1.north west)}{(-45:\ie)} .. controls +(135:0.25) and \control{(n2.south west)}{(-135:0.25)} .. controls +(45:\ie) and \control{(n2.north west)}{(-45:\ie)} .. controls +(135:0.25) and \control{(2.5,2)}{(-90:0.15)}; 
\draw (3.5,0) .. controls +(90:0.15) and \control{(n1.south east)}{(-45:0.25)}  .. controls +(135:\ie) and \control{(n1.north east)}{(-135:\ie)} .. controls +(45:0.25) and \control{(n2.south east)}{(-45:0.25)} .. controls +(135:\ie) and \control{(n2.north east)}{(-135:\ie)} .. controls +(45:0.25) and \control{(3.5,2)}{(-90:0.15)}; 
\draw [Stealth-Stealth] (1.25,1) -- (2.25,1) node [midway, above] {M2a};
\end{scope}

\begin{scope}[xshift=12cm]
\node [circle, inner sep = \e cm] (n1) at (3,0.6) {};
\node [circle, inner sep = \e cm] (n2) at (3,1.4) {};
\node [circle, inner sep = \e cm] (n3) at (0.5,0.6) {};
\node [circle, inner sep = \e cm] (n4) at (0.5,1.4) {};

\draw (0,0) .. controls +(90:0.15) and \control{(n3.south west)}{(-135:0.25)}  .. controls +(45:\ie) and \control{(n3.north west)}{(-45:\ie)} .. controls +(135:0.25) and \control{(n4.south west)}{(-135:0.25)} .. controls +(45:\ie) and \control{(n4.south east)}{(135:\ie)} .. controls +(-45:0.25) and \control{(n3.north east)}{(45:0.25)} .. controls +(-135:\ie) and \control{(n3.south east)}{(135:\ie)} .. controls +(-45:0.25) and \control{(1,0)}{(90:0.15)};
\draw (0,2) .. controls +(-90:0.15) and \control{(n4.north west)}{(135:0.25)}  .. controls +(-45:\ie) and \control{(n4.north east)}{(-135:\ie)} .. controls +(45:0.25) and \control{(1,2)}{(-90:0.15)}; 

\draw (2.5,2) .. controls +(-90:0.15) and \control{(n2.north west)}{(135:0.25)}  .. controls +(-45:\ie) and \control{(n2.south west)}{(45:\ie)} .. controls +(-135:0.25) and \control{(n1.north west)}{(135:0.25)} .. controls +(-45:\ie) and \control{(n1.north east)}{(-135:\ie)} .. controls +(45:0.25) and \control{(n2.south east)}{(-45:0.25)} .. controls +(135:\ie) and \control{(n2.north east)}{(-135:\ie)} .. controls +(45:0.25) and \control{(3.5,2)}{(-90:0.15)};
\draw (2.5,0) .. controls +(90:0.15) and \control{(n1.south west)}{(-135:0.25)}  .. controls +(45:\ie) and \control{(n1.south east)}{(135:\ie)} .. controls +(-45:0.25) and \control{(3.5,0)}{(90:0.15)}; 

\draw [Stealth-Stealth, blue!80!black] (1.25,1) -- (2.25,1) node [midway, above] {M2b};
\node at (-0.5,1) {\Large +};
\end{scope}
\end{scope}

\begin{scope}[yshift=-6.25cm]
\draw (0,0) -- ++(2,2);
\draw (2,0) -- ++(-2,2);
\draw (0,1.4) -- ++(2,0);

\draw [Stealth-Stealth] (2.25,1) -- (3.25,1) node [midway, above] {R3};

\draw (3.5,0) -- ++(2,2);
\draw (5.5,0) -- ++(-2,2);
\draw (3.5,0.6) -- ++(2,0);

\def\ee{0.075}
\begin{scope}[xshift=7.5cm, yshift = 1.25cm]
\node (n1) at (0.6,1.4) [circle, inner sep = \e cm] {};
\node (n2) at (1,1) [circle, inner sep = \e cm] {};
\node (n3) at (1.4,1.4) [circle, inner sep = \ee cm] {};

\draw (0,1.4) -- (n1.west) .. controls +(0:\ie) and \control{(n1.north west)}{(-45:\ie)} -- (0,2); 
\draw (0,0) -- (n2.south west) .. controls +(45:\ie) and \control{(n2.north west)}{(-45:\ie)} -- (n1.south east).. controls +(135:\ie) and \control{(n1.east)}{(180:\ie)} -- (n3.west) .. controls +(0:\ie) and \control{(n3.north east)}{(-135:\ie)} -- (2,2); 
\draw (2,0) -- (n2.south east) .. controls +(135:\ie) and \control{(n2.north east)}{(-135:\ie)} -- (n3.south west) .. controls +(45:\ie) and \control{(n3.east)}{(180:\ie)} -- (2,1.4);

\node (n1) at (4.1,0.6) [circle, inner sep = \ee cm] {};
\node (n2) at (4.5,1) [circle, inner sep = \e cm] {};
\node (n3) at (4.9,0.6) [circle, inner sep = \e cm] {};

\draw (5.5,0) -- (n3.south east) .. controls +(135:\ie) and \control{(n3.east)}{(180:\ie)} -- (5.5,0.6); 
\draw (3.5,0) -- (n1.south west) .. controls +(45:\ie) and \control{(n1.east)}{(180:\ie)} -- (n3.west) .. controls +(0:\ie) and \control{(n3.north west)}{(-45:\ie)} -- (n2.south east) .. controls +(135:\ie) and \control{(n2.north east)}{(-135:\ie)} -- (5.5,2); 
\draw (3.5,2) -- (n2.north west) .. controls +(-45:\ie) and \control{(n2.south west)}{(45:\ie)} -- (n1.north east) .. controls +(-135:\ie) and \control{(n1.west)}{(0:\ie)} -- (3.5,0.6); 
\draw [Stealth-Stealth] (2.25,1) -- (3.25,1) node [midway, above] {M3a};
\end{scope}

\begin{scope}[xshift=7.5cm, yshift = -1.25cm]
\node (n1) at (0.6,1.4) [circle, inner sep = \ee cm] {};
\node (n2) at (1,1) [circle, inner sep = \e cm] {};
\node (n3) at (1.4,1.4) [circle, inner sep = \ee cm] {};

\draw (0,1.4) -- (n1.west) .. controls +(0:\ie) and \control{(n1.south east)}{(135:\ie)} -- (n2.north west) .. controls +(-45:\ie) and \control{(n2.south west)}{(45:\ie)} -- (0,0); 
\draw (0,2) -- (n1.north west) .. controls +(-45:\ie) and \control{(n1.east)}{(180:\ie)} -- (n3.west) .. controls +(0:\ie) and \control{(n3.north east)}{(-135:\ie)} -- (2,2); 
\draw (2,0) -- (n2.south east) .. controls +(135:\ie) and \control{(n2.north east)}{(-135:\ie)} -- (n3.south west) .. controls +(45:\ie) and \control{(n3.east)}{(180:\ie)} -- (2,1.4); 

\node (n1) at (4.1,0.6) [circle, inner sep = \e cm] {};
\node (n2) at (4.5,1) [circle, inner sep = \e cm] {};
\node (n3) at (4.9,0.6) [circle, inner sep = \e cm] {};

\draw (3.5,0.6) -- (n1.west) .. controls +(0:\ie) and \control{(n1.south west)}{(45:\ie)} -- (3.5,0);
\draw (3.5,2) -- (n2.north west) .. controls +(-45:\ie) and \control{(n2.south west)}{(45:\ie)} -- (n1.north east) .. controls +(-135:\ie) and \control{(n1.east)}{(180:\ie)} -- (n3.west) .. controls +(0:\ie) and \control{(n3.north west)}{(-45:\ie)} -- (n2.south east) .. controls +(135:\ie) and \control{(n2.north east)}{(-135:\ie)} -- (5.5,2);
\draw (5.5,0.6) -- (n3.east) .. controls +(180:\ie) and \control{(n3.south east)}{(135:\ie)} -- (5.5,0);
\draw [Stealth-Stealth] (2.25,1) -- (3.25,1) node [midway, above] {M3b};
\end{scope}
\end{scope}

\draw [->, gray] (3,1) -- +(4,0);
\draw [->, gray] (4,-1.5) -- +(3,0);
\draw [->, gray] (6,-5) -- +(0.75,0.5);
\draw [->, gray] (6,-6) -- +(0.75,-0.5);
\end{tikzpicture}
\caption{Resolution moves.}
\label{pic_resolution_moves}
\end{center}
\end{figure}

The \textbf{graph of resolutions} $\Gamma$ is defined as follows. Let $\{ t_j \}_j$ be a family of times $t_j \in [0,1]$ alternating with critical times of the homotopy on $\gamma$. The set of vertices of $\Gamma$ is partitioned into $j$ layers, and for each layer $j$ there is a vertex for each resolution in $\{ G^j_i \}_i$. Between two times $t_j$ and $t_{j+1}$ there is exactly one critical time, which corresponds to a move R1, R2 and R3 on the curve $\gamma$. We put edges in $\Gamma$ between vertices in two consecutive layers $j$ and $j+1$ according to the following rules:
\begin{itemize}
\item If $\gamma_{t_j}$ and $\gamma_{t_{j+1}}$ differ by a move R1, then we put an edge between two resolutions $G^j_i$ and $G^{j+1}_{i'}$ whenever they differ by the move M1 pictured in \Cref{pic_resolution_moves}, top.

\item If $\gamma_{t_j}$ and $\gamma_{t_{j+1}}$ differ by a move R2, then we put an edge between two resolutions $G^j_i$ and $G^{j+1}_{i'}$ whenever they differ by the move M2a pictured in \Cref{pic_resolution_moves}, middle. 

\item If $\gamma_{t_j}$ and $\gamma_{t_{j+1}}$ differ by a move R3, then we put an edge between two resolutions $G^j_i$ and $G^{j+1}_{i'}$ whenever they differ by the moves M3a or M3b pictured in \Cref{pic_resolution_moves}, bottom.
  \end{itemize}

Additionally, we also add edges between vertices in a common layer $j$ according to the following rule:

\begin{itemize}
\item If $t_j$ follows or precedes an R2 move, we put an edge between two resolutions $G^j_i$ and $G^{j}_{i'}$ whenever they differ by the move M2b in \Cref{pic_graph_ex}, middle right.
\end{itemize}

Note that in this last case, the move M2a cannot be applied to $G^j_i$ and $G^j_{i'}$ since the resolutions around the two double points are not compatible. An example of a graph of resolutions $\Gamma$ is pictured in \Cref{pic_graph_ex}.

First, notice that by construction, two resolution curves $G^j_i$ and $G^{j'}_{i'}$ connected by an edge in $\Gamma$ are isotopic, and the isotopy has the property that the isotopy curves are all $\epsilon$-equivalent to curves $\gamma_{t}$ for $t$ in $[j,j']$. Note that this is also the case for $j=j'$ and the move M3b.

Furthermore, the degrees in $\Gamma$ are very constrained: when a $\gamma_{t_j}$ transforms to $\gamma_{t_{j+1}}$ via a move R1, then every resolution $G^j_i$ is incident to exactly one resolution $G^{j+1}_{i'}$ via a move M1. For a move R3, each resolution $G^j_i$ is incident to either exactly one resolution $G^{j+1}_{i'}$ via a move M3a or M3b, or to exactly three resolutions $G^{j+1}_{i_1}$, $G^{j+1}_{i_2}$, and $G^{j+1}_{i_3}$ due to the three-fold symmetry of the latter move. Finally, a move R2 between layers $j$ and $j+1$ induces for each resolution $G^j_i$ either an incidence to some $G^{j+1}_{i'}$ with a move M2a or to some  $G^j_{i'}$ with a move M2b. Therefore, by construction, each vertex outside of the first and last layers is incident to an even number of edges: either $2$, $4$, or $6$. We refer again to \Cref{pic_graph_ex} for an illustration of this behaviour.

\begin{figure}[ht]
\begin{center}
\begin{tikzpicture}[scale = 0.65]
\def\e{0.05} \def\se{0.3} \def\me{0.8} \def\le{1.5} \def\lle{2}

\clip (-1.25,3.75) rectangle (20,-14.75);
\begin{scope}[yshift = 3cm]
\draw (0.5,0) circle (1.25cm and 0.55cm);
\end{scope}

\coordinate (c0) at (0,0);
\coordinate (c1) at (1,0);
\foreach \i in {0,...,1}{\node (n\i) [circle, inner sep = \e cm] at (c\i) {};}

\draw (n0.west) .. controls +(180:\le) and \control{(n0.south)}{(-90:\le)} -- (n0.north) .. controls +(90:\me) and \control{(n1.north)}{(90:\me)} -- (n1.south) .. controls +(-90:\le) and \control{(n1.east)}{(0:\le)} -- (n1.west) -- (n0.west);

\begin{scope}[yshift = -3cm]
\coordinate (c0) at (0,0);
\coordinate (c1) at (1,0);
\coordinate (c2) at (0.5,-0.8);
\coordinate (c3) at (0.5,-1.8);
\foreach \i in {0,1,2,3}{\node (n\i) [circle, inner sep = \e cm] at (c\i) {};}

\draw (n0.west) .. controls +(180:\le) and \control{(n3.south west)}{(-135:\lle)} -- (n3.north east) .. controls +(45:\se) and \control{(n2.south east)}{(-45:\se)} -- (n2.north west) .. controls +(135:\se) and \control{(n0.south)}{(-90:\se)} -- (n0.north) .. controls +(90:\me) and \control{(n1.north)}{(90:\me)} -- (n1.south) .. controls +(-90:\se) and \control{(n2.north east)}{(45:\se)} -- (n2.south west) .. controls +(-135:\se) and \control{(n3.north west)}{(135:\se)} -- (n3.south east) .. controls + (-45:\lle) and \control{(n1.east)}{(0:\le)} -- (n0.west);
\end{scope}

\begin{scope}[yshift = -8cm]
\coordinate (c0) at (0.5,0.3);
\coordinate (c1) at (0,-0.8);
\coordinate (c2) at (1,-0.8);
\coordinate (c3) at (0.5,-1.8);
\foreach \i in {0,1,2,3}{\node (n\i) [circle, inner sep = \e cm] at (c\i) {};}

\draw (n0.north west) .. controls +(135:\lle) and \control{(n0.north east)}{(45:\lle)} -- (n0.south west) .. controls +(-135:\se) and \control{(n1.north)}{(90:\se)} -- (n1.south) .. controls +(-90:\se) and \control{(n3.north west)}{(135:\se)} -- (n3.south east) .. controls +(-45:\le) and \control{(n2.east)}{(0:\le)} -- (n1.west) .. controls +(180:\le) and \control{(n3.south west)}{(-135:\le)} -- (n3.north east) .. controls +(45:\se) and \control{(n2.south)}{(-90:\se)} -- (n2.north) .. controls +(90:\se) and \control{(n0.south east)}{(-45:\se)} -- (n0.north west);
\end{scope}

\begin{scope}[xshift = 0.5cm, yshift = -12.5cm]
\coordinate (c0) at (0,0);
\coordinate (c1) at (-0.5,-0.8);
\coordinate (c2) at (-0.5,-2);
\foreach \i in {0,1}{\node (n\i) [circle, inner sep = \e cm] at (c\i) {};}

\draw (n0.north west) .. controls +(135:\lle) and \control{(n0.north east)}{(45:\lle)} -- (n0.south west) .. controls +(-135:\se) and \control{(n1.north)}{(90:\se)} -- (n1.south) .. controls +(-90:\le) and \control{(n1.east)}{(0:\le)} -- (n1.west) .. controls +(180:\le) and \control{(c2)}{(180:1.5)} .. controls +(0:\lle) and \control{(n0.south east)}{(-45:\le)} -- (n0.north west);
\end{scope}

\draw [Stealth-Stealth] (0.5,1) -- +(0,1) node [midway, right] {R2};
\draw [Stealth-Stealth] (0.5,-2) -- +(0,1) node [midway, right] {R2};
\draw [Stealth-Stealth] (0.5,-6.25) -- +(0,1) node [midway, right] {R3};
\draw [Stealth-Stealth] (0.5,-11.125) -- +(0,1) node [midway, right] {R2};

\def\ie{0.05}
\begin{scope}[xshift = 11cm]

\begin{scope}[yshift = 3cm]
\draw (0.5,0) circle (1.25cm and 0.55cm);
\end{scope}

\coordinate (c0) at (0,0);
\coordinate (c1) at (1,0);
\foreach \i in {0,...,1}{\node (n\i) [circle, inner sep = \e cm] at (c\i) {};}

\draw (n0.west) .. controls +(180:\le) and \control{(n0.south)}{(-90:\le)} .. controls +(90:\ie) and \control{(n0.east)}{(180:\ie)} -- (n1.west) .. controls +(0:\ie) and \control{(n1.south)}{(90:\ie)} .. controls +(-90:\le) and \control{(n1.east)}{(0:\le)} .. controls +(180:\ie) and \control{(n1.north)}{(-90:\ie)} .. controls +(90:\me) and \control{(n0.north)}{(90:\me)} .. controls +(-90:\ie) and \control{(n0.west)}{(0:\ie)};

\begin{scope}[yshift = -3cm]
\coordinate (c0) at (0,0);
\coordinate (c1) at (1,0);
\coordinate (c2) at (0.5,-0.8);
\coordinate (c3) at (0.5,-1.8);
\foreach \i in {0,1,2,3}{\node (n\i) [circle, inner sep = \e cm] at (c\i) {};}

\draw (n0.west) .. controls +(180:\le) and \control{(n3.south west)}{(-135:\lle)} .. controls +(45:\ie) and \control{(n3.north west)}{(-45:\ie)} .. controls +(135:\se) and \control{(n2.south west)}{(-135:\se)} .. controls +(45:\ie) and \control{(n2.north west)}{(-45:\ie)}.. controls +(135:\se) and \control{(n0.south)}{(-90:\se)} .. controls +(90:\ie) and \control{(n0.east)}{(180:\ie)} -- (n1.west) .. controls +(0:\ie) and \control{(n1.south)}{(90:\ie)}.. controls +(-90:\se) and \control{(n2.north east)}{(45:\se)} .. controls +(-135:\ie) and \control{(n2.south east)}{(135:\ie)} .. controls +(-45:\se) and \control{(n3.north east)}{(45:\se)} .. controls +(-135:\ie) and \control{(n3.south east)}{(135:\ie)} .. controls + (-45:\lle) and \control{(n1.east)}{(0:\le)} .. controls +(180:\ie) and \control{(n1.north)}{(-90:\ie)} .. controls +(90:\me) and \control{(n0.north)}{(90:\me)} .. controls +(-90:\ie) and \control{(n0.west)}{(0:\ie)};
\end{scope}

\begin{scope}[xshift = -7cm, yshift = -3cm]
\coordinate (c0) at (0,0);
\coordinate (c1) at (1,0);
\coordinate (c2) at (0.5,-0.8);
\coordinate (c3) at (0.5,-1.8);
\foreach \i in {0,1,2,3}{\node (n\i) [circle, inner sep = \e cm] at (c\i) {};}

\draw (n0.west) .. controls +(180:\le) and \control{(n3.south west)}{(-135:\lle)} .. controls +(45:\ie) and \control{(n3.south east)}{(135:\ie)} .. controls + (-45:\lle) and \control{(n1.east)}{(0:\le)} .. controls +(180:\ie) and \control{(n1.south)}{(90:\ie)} .. controls +(-90:\se) and \control{(n2.north east)}{(45:\se)} .. controls +(-135:\ie) and \control{(n2.south east)}{(135:\ie)} .. controls +(-45:\se) and \control{(n3.north east)}{(45:\se)} .. controls +(-135:\ie) and \control{(n3.north west)}{(-45:\ie)} .. controls +(135:\se) and \control{(n2.south west)}{(-135:\se)} .. controls +(45:\ie) and \control{(n2.north west)}{(-45:\ie)} .. controls +(135:\se) and \control{(n0.south)}{(-90:\se)} .. controls +(90:\ie) and \control{(n0.east)}{(180:\ie)} -- (n1.west) .. controls +(0:\ie) and \control{(n1.north)}{(-90:\ie)} .. controls +(90:\me) and \control{(n0.north)}{(90:\me)} .. controls +(-90:\ie) and \control{(n0.west)}{(0:\ie)};
\end{scope}

\begin{scope}[xshift = -3.5cm, yshift = -3cm]
\coordinate (c0) at (0,0);
\coordinate (c1) at (1,0);
\coordinate (c2) at (0.5,-0.8);
\coordinate (c3) at (0.5,-1.8);
\foreach \i in {0,1,2,3}{\node (n\i) [circle, inner sep = \e cm] at (c\i) {};}

\draw (n0.west) .. controls +(180:\le) and \control{(n3.south west)}{(-135:\lle)} .. controls +(45:\ie) and \control{(n3.north west)}{(-45:\ie)} .. controls +(135:\se) and \control{(n2.south west)}{(-135:\se)} .. controls +(45:\ie) and \control{(n2.south east)}{(135:\ie)} .. controls +(-45:\se) and \control{(n3.north east)}{(45:\se)} .. controls +(-135:\ie) and \control{(n3.south east)}{(135:\ie)} .. controls + (-45:\lle) and \control{(n1.east)}{(0:\le)} .. controls +(180:\ie) and \control{(n1.south)}{(90:\ie)} .. controls +(-90:\se) and \control{(n2.north east)}{(45:\se)} .. controls +(-135:\ie) and \control{(n2.north west)}{(-45:\ie)} .. controls +(135:\se) and \control{(n0.south)}{(-90:\se)} .. controls +(90:\ie) and \control{(n0.east)}{(180:\ie)} -- (n1.west) .. controls +(0:\ie) and \control{(n1.north)}{(-90:\ie)} .. controls +(90:\me) and \control{(n0.north)}{(90:\me)} .. controls +(-90:\ie) and \control{(n0.west)}{(0:\ie)};
\end{scope}

\begin{scope}[xshift = 3.5cm, yshift = -3cm]
\coordinate (c0) at (0,0);
\coordinate (c1) at (1,0);
\coordinate (c2) at (0.5,-0.8);
\coordinate (c3) at (0.5,-1.8);
\foreach \i in {0,1,2,3}{\node (n\i) [circle, inner sep = \e cm] at (c\i) {};}

\draw (n0.west) .. controls +(180:\le) and \control{(n3.south west)}{(-135:\lle)} .. controls +(45:\ie) and \control{(n3.north west)}{(-45:\ie)} .. controls +(135:\se) and \control{(n2.south west)}{(-135:\se)} .. controls +(45:\ie) and \control{(n2.south east)}{(135:\ie)} .. controls +(-45:\se) and \control{(n3.north east)}{(45:\se)} .. controls +(-135:\ie) and \control{(n3.south east)}{(135:\ie)} .. controls + (-45:\lle) and \control{(n1.east)}{(0:\le)} .. controls +(180:\ie) and \control{(n1.north)}{(-90:\ie)} .. controls +(90:\me) and \control{(n0.north)}{(90:\me)} .. controls +(-90:\ie) and \control{(n0.east)}{(180:\ie)} -- (n1.west) .. controls +(0:\ie) and \control{(n1.south)}{(90:\ie)} .. controls +(-90:\se) and \control{(n2.north east)}{(45:\se)} .. controls +(-135:\ie) and \control{(n2.north west)}{(-45:\ie)} .. controls +(135:\se) and \control{(n0.south)}{(-90:\se)} .. controls +(90:\ie) and \control{(n0.west)}{(0:\ie)};
\end{scope}

\begin{scope}[xshift = 7cm, yshift = -3cm]
\coordinate (c0) at (0,0);
\coordinate (c1) at (1,0);
\coordinate (c2) at (0.5,-0.8);
\coordinate (c3) at (0.5,-1.8);
\foreach \i in {0,1,2,3}{\node (n\i) [circle, inner sep = \e cm] at (c\i) {};}

\draw (n0.west) .. controls +(180:\le) and \control{(n3.south west)}{(-135:\lle)} .. controls +(45:\ie) and \control{(n3.south east)}{(135:\ie)} .. controls + (-45:\lle) and \control{(n1.east)}{(0:\le)} .. controls +(180:\ie) and \control{(n1.north)}{(-90:\ie)} .. controls +(90:\me) and \control{(n0.north)}{(90:\me)} .. controls +(-90:\ie) and \control{(n0.east)}{(180:\ie)} -- (n1.west) .. controls +(0:\ie) and \control{(n1.south)}{(90:\ie)} .. controls +(-90:\se) and \control{(n2.north east)}{(45:\se)} .. controls +(-135:\ie) and \control{(n2.south east)}{(135:\ie)} .. controls +(-45:\se) and \control{(n3.north east)}{(45:\se)} .. controls +(-135:\ie) and \control{(n3.north west)}{(-45:\ie)} .. controls +(135:\se) and \control{(n2.south west)}{(-135:\se)} .. controls +(45:\ie) and \control{(n2.north west)}{(-45:\ie)} .. controls +(135:\se) and \control{(n0.south)}{(-90:\se)} .. controls +(90:\ie) and \control{(n0.west)}{(0:\ie)};
\end{scope}

\begin{scope}[xshift = -4cm, yshift = -8cm]
\coordinate (c0) at (0.5,0.3);
\coordinate (c1) at (0,-0.8);
\coordinate (c2) at (1,-0.8);
\coordinate (c3) at (0.5,-1.8);
\foreach \i in {0,1,2,3}{\node (n\i) [circle, inner sep = \e cm] at (c\i) {};}

\draw (n0.north west) .. controls +(135:\lle) and \control{(n0.north east)}{(45:\lle)} .. controls +(-135:\ie) and \control{(n0.south east)}{(135:\ie)} .. controls +(-45:\se) and \control{(n2.north)}{(90:\se)} .. controls +(-90:\ie) and \control{(n2.west)}{(0:\ie)} -- (n1.east) .. controls +(180:\ie) and \control{(n1.south)}{(90:\ie)} .. controls +(-90:\se) and \control{(n3.north west)}{(135:\se)} .. controls +(-45:\ie) and \control{(n3.north east)}{(-135:\ie)} .. controls +(45:\se) and \control{(n2.south)}{(-90:\se)} .. controls +(90:\ie) and \control{(n2.east)}{(180:\ie)} .. controls +(0:\le) and \control{(n3.south east)}{(-45:\le)} .. controls +(135:\ie) and \control{(n3.south west)}{(45:\ie)} .. controls +(-135:\le) and \control{(n1.west)}{(180:\le)} .. controls +(0:\ie) and \control{(n1.north)}{(-90:\ie)} .. controls +(90:\se) and \control{(n0.south west)}{(-135:\se)} .. controls +(45:\ie) and \control{(n0.north west)}{(-45:\ie)};
\end{scope}

\begin{scope}[yshift = -8cm]
\coordinate (c0) at (0.5,0.3);
\coordinate (c1) at (0,-0.8);
\coordinate (c2) at (1,-0.8);
\coordinate (c3) at (0.5,-1.8);
\foreach \i in {0,1,2,3}{\node (n\i) [circle, inner sep = \e cm] at (c\i) {};}

\draw (n0.north west) .. controls +(135:\lle) and \control{(n0.north east)}{(45:\lle)} .. controls +(-135:\ie) and \control{(n0.south east)}{(135:\ie)} .. controls +(-45:\se) and \control{(n2.north)}{(90:\se)} .. controls +(-90:\ie) and \control{(n2.east)}{(180:\ie)} .. controls +(0:\le) and \control{(n3.south east)}{(-45:\le)} .. controls +(135:\ie) and \control{(n3.north east)}{(-135:\ie)} .. controls +(45:\se) and \control{(n2.south)}{(-90:\se)} .. controls +(90:\ie) and \control{(n2.west)}{(0:\ie)} -- (n1.east) .. controls +(180:\ie) and \control{(n1.south)}{(90:\ie)} .. controls +(-90:\se) and \control{(n3.north west)}{(135:\se)} .. controls +(-45:\ie) and \control{(n3.south west)}{(45:\ie)} .. controls +(-135:\le) and \control{(n1.west)}{(180:\le)} .. controls +(0:\ie) and \control{(n1.north)}{(-90:\ie)} .. controls +(90:\se) and \control{(n0.south west)}{(-135:\se)} .. controls +(45:\ie) and \control{(n0.north west)}{(-45:\ie)};
\end{scope}

\begin{scope}[xshift = 4cm, yshift = -8cm]
\coordinate (c0) at (0.5,0.3);
\coordinate (c1) at (0,-0.8);
\coordinate (c2) at (1,-0.8);
\coordinate (c3) at (0.5,-1.8);
\foreach \i in {0,1,2,3}{\node (n\i) [circle, inner sep = \e cm] at (c\i) {};}

\draw (n0.north west) .. controls +(135:\lle) and \control{(n0.north east)}{(45:\lle)} .. controls +(-135:\ie) and \control{(n0.south east)}{(135:\ie)} .. controls +(-45:\se) and \control{(n2.north)}{(90:\se)} .. controls +(-90:\ie) and \control{(n2.east)}{(180:\ie)} .. controls +(0:\le) and \control{(n3.south east)}{(-45:\le)} .. controls +(135:\ie) and \control{(n3.south west)}{(45:\ie)} .. controls +(-135:\le) and \control{(n1.west)}{(180:\le)} .. controls +(0:\ie) and \control{(n1.south)}{(90:\ie)} .. controls +(-90:\se) and \control{(n3.north west)}{(135:\se)} .. controls +(-135:\ie) and \control{(n3.north east)}{(-45:\ie)} .. controls +(45:\se) and \control{(n2.south)}{(-90:\se)} .. controls +(90:\ie) and \control{(n2.west)}{(0:\ie)} -- (n1.east) .. controls +(180:\ie) and \control{(n1.north)}{(-90:\ie)} .. controls +(90:\se) and \control{(n0.south west)}{(-135:\se)} .. controls +(45:\ie) and \control{(n0.north west)}{(-45:\ie)};
\end{scope}

\begin{scope}[xshift = 0.5cm, yshift = -12.5cm]
\coordinate (c0) at (0,0);
\coordinate (c1) at (-0.5,-0.8);
\coordinate (c2) at (-0.5,-2);
\foreach \i in {0,1}{\node (n\i) [circle, inner sep = \e cm] at (c\i) {};}

\draw (n0.north west) .. controls +(135:\lle) and \control{(n0.north east)}{(45:\lle)} .. controls +(-135:\ie) and \control{(n0.south east)}{(135:\ie)} .. controls +(-45:\le) and \control{(c2)}{(0:1.5)} .. controls +(180:\le) and \control{(n1.west)}{(180:\le)} .. controls +(0:\ie) and \control{(n1.south)}{(90:\ie)} .. controls +(-90:\le) and \control{(n1.east)}{(0:\le)} .. controls +(180:\ie) and \control{(n1.north)}{(-90:\ie)} .. controls +(90:\se) and \control{(n0.south west)}{(-135:\se)} .. controls +(45:\ie) and \control{(n0.north west)}{(-45:\ie)};
\end{scope}

\draw [Stealth-Stealth] (0.5,1) -- +(0,1) node [midway, right] {M2a};
\draw [Stealth-Stealth] (0.5,-2) -- +(0,1) node [midway, right] {M2a};
\draw [Stealth-Stealth] (0.5,-6.25) -- +(0,1) node [midway, right] {M3b};
\draw [Stealth-Stealth] (-5.5,-2.75) -- +(1.5,0) node [midway, above] {M2b};
\draw [Stealth-Stealth] (5,-2.75) -- +(1.5,0) node [midway, above] {M2b};
\draw [Stealth-Stealth] (-2.25,-8) -- +(1.5,0) node [midway, above] {M2b};
\draw [Stealth-Stealth] (-5.75,-5.5) -- +(1.5,-1.5) node [midway, above right] {M3a};
\draw [Stealth-Stealth] (-1.75,-5.25) -- +(1.5,-1.5) node [midway, below left] {M3b};
\draw [Stealth-Stealth] (6.75,-5.5) -- +(-1.5,-1.5) node [midway, above left] {M3a};
\draw [Stealth-Stealth] (2.75,-5.25) -- +(-1.5,-1.5) node [midway, below right] {M3b};
\draw [Stealth-Stealth] (3.25,-10) -- +(-2,-1.5) node [midway, below right] {M2a};
\end{scope}
\end{tikzpicture}
\caption{An example of a graph of resolutions taken from \cite[Figure 7]{Chambers_homo-isotopy}.}
\label{pic_graph_ex}
\end{center}
\end{figure}

According to assumptions of \Cref{Chambers_homo-isotopy}, we now assume that $\gamma_0$ is simple so that the first layer $\{ G^0_i \}_i$ has only one vertex $G_0^0$. This vertex has degree one, while all the vertices in the graph outside of the first and last layer have even degree. Summing the degrees of vertices of a connected component of $\Gamma$ yields twice the number of edges of $\Gamma$: an even number. Hence, there is an even number of vertices of odd degrees in each connected component of $\Gamma$, and thus there exists a path in $\Gamma$ from $G^0_0$ to a vertex on the last layer. Gluing together the isotopies between resolutions of $\gamma_t$ corresponding to each edge, we obtain an isotopy from $\overline{\gamma_0}=\gamma_0$ to a curve $\overline{\gamma_1}$ that is $\epsilon$-image equivalent to $\gamma_1$, where each intermediate curve is $\epsilon$-equivalent to curves in $\gamma_t$.  Note that if $\gamma_1$ is a point, it can be replaced in the proof by a simple closed curve within a ball of radius $\epsilon$, that can then be isotoped to a point within the ball. This concludes the proof of \Cref{Chambers_homo-isotopy}.

Let us now complete our proof of \Cref{prop_iso_U}. We want to apply the same proof strategy to the curve $\mathcal{U}_t$ while keeping track of the number of intersections between $\mathcal{U}_t$, and $\mathcal{M}_t$ which serves as a discrete measure of length. The main difference with the previous proof is that this discrete metric evolves during the homotopy of curves $p \circ \phi$. Since this metric only takes integer values, we fix some $\epsilon \in ]0,1[$ that will be irrelevant with respect to the metric and which is only used to define $\epsilon$-image equivalent.

Let $\{ c_j \}_{1\leq j \leq s} \subset [0,1]$ be the critical times involving only $\mathcal{U}_t$. Let $\{ m_j \}_{1\leq j \leq r}$ be the critical times involving $\mathcal{M}_t$. We use the notation $\cross (\cdot, \cdot)$ from \Cref{sec:sweepout} to define the discrete length $ \| \mathcal{U}_t \| = \cross ( \mathcal{U}_t, \mathcal{M}_t )$. Outside of these critical times, we implicitly resolve all curves $\mathcal{U}_t$ with $\epsilon '$-image equivalent curves where $\epsilon ' \leq \epsilon$ so that all balls of radius $\epsilon '$ centred at crossings of $\mathcal{U}_t$ are disjoint from $\mathcal{M}_t$. Therefore, for $t_\ell$ disjoint from $\{ c_j \}_{1\leq j \leq s}$, the length of each curve in $\{ G^\ell_i\}_i$ is $\| \mathcal{U}_{t_\ell} \|$. We now define an increasing family of times $\{ t_j \}_{0 \leq j \leq s}$ satisfying $t_0 = 0$, $t_1 = 1$, and for $0 < j < s$, $ c_j < t_{j} < c_{j+1} $ and there are no $m_i$ between $c_j$ and $t_{j}$ (see \Cref{pic_def_times}). The graph $\Gamma$ is defined as previously by connecting resolutions of $\mathcal{U}_{t_j}$ between these times $t_j$.

\begin{figure}[ht]
\begin{center}
\begin{tikzpicture}[scale = 0.9]
\draw (-0.5,0) -- (10.5,0);
\draw (0,0) -- +(0,-0.3) node [below] {$0$};
\draw (10,0) -- +(0,-0.3) node [below] {$1$};

\begin{scope}[red!80!black]
\draw (1,0) -- +(0,-0.2) node [below] {$c_1$};
\draw (2.5,0) -- +(0,-0.2) node [below] {$c_2$};
\draw (4.5,0) -- +(0,-0.2) node [below] {$c_3$};
\draw (5.2,0) -- +(0,-0.2) node [below] {$c_4$};
\draw (9,0) -- +(0,-0.2) node [below] {$c_q$};
\path (6.9,0) -- +(0,-0.2) node [below] {$\cdot$};
\path (7.2,0) -- +(0,-0.2) node [below] {$\cdot$};
\path (7.5,0) -- +(0,-0.2) node [below] {$\cdot$};
\end{scope}

\begin{scope}[blue!80!black]
\draw (1.65,0) -- +(0,-0.1) node [below] {$m_1$};
\draw (3,0) -- +(0,-0.1) node [below] {$m_2$};
\path (3.4,0) -- +(0,-0.1) node [below] {\small $\cdot$};
\path (3.5,0) -- +(0,-0.1) node [below] {\small $\cdot$};
\path (3.6,0) -- +(0,-0.1) node [below] {\small $\cdot$};
\draw (4,0) -- +(0,-0.1) node [below] {$m_5$};
\draw (6,0) -- +(0,-0.1) node [below] {$m_5$};
\path (6.3,0) -- +(0,-0.1) node [below] {\small $\cdot$};
\path (6.4,0) -- +(0,-0.1) node [below] {\small $\cdot$};
\path (6.5,0) -- +(0,-0.1) node [below] {\small $\cdot$};
\draw (9.5,0) -- +(0,-0.1) node [below] {$m_r$};
\end{scope}

\begin{scope}[gray]
\path (6.9,0) -- +(0,0.2) node [above] {$\cdot$};
\path (7.2,0) -- +(0,0.2) node [above] {$\cdot$};
\path (7.5,0) -- +(0,0.2) node [above] {$\cdot$};
\draw (0,0) -- +(0,0.2) node [above] {$t_0$};
\draw (1.2,0) -- +(0,0.2) node [above] {$t_1$};
\draw (2.7,0) -- +(0,0.2) node [above] {$t_3$};
\draw (4.7,0) -- +(0,0.2) node [above] {$t_4$};
\draw (5.35,0) -- +(0,0.2) node [above] {$t_5$};
\draw (10,0) -- +(0,0.2) node [above] {$t_{s}$};
\end{scope}
\end{tikzpicture}
\caption{Definitions of the $t_j$.}
\label{pic_def_times}
\end{center}
\end{figure}

To conclude the proof, we proceed as for the proof of \Cref{Chambers_homo-isotopy} except that we additionally carry the transformations of $\mathcal{M}_t$ and perform them when the corresponding vertical edges are taken in $\Gamma$. At the level of curves, this means that when we go backwards in time following the path going through $\Gamma$, we also reverse the Reidemeister moves R1, R2, and R3 that were performed on $\mathcal{M}_t$, so that the discrete metric evolves alongside the isotopy of $\mathcal{U}_t$ that we are creating. This ensures that the length of a resolution $||G^j_i||$ is never larger than $\cross(\mathcal{U}_{t_j},\mathcal{M}_{t_j})$ and thus that we can apply the isotopies specified by the edges of $\Gamma$ while staying within the length budget. This step is illustrated in \Cref{pic_ex_graph_times}, which is essentially the combination of \Cref{pic_graph_ex} and \Cref{pic_def_times}. 

Following the hypotheses of \Cref{prop_iso_U}, we assume that for all $t \in [0,1]$, $\| \mathcal{U}_t \| \leq m$ where $m$ is a fixed integer. It follows that for $t \in [0,1]$ and all $i,j$, $\| G^j_i \| = \cross (G^j_i , \mathcal{M}_t) \leq m$, since $\mathcal{U}_{t_j}$ and $G^j_i$ are $\epsilon$-image equivalent. By construction, the length of curves does not change on the horizontal edges (those happen between changes in the number of intersections between $\mathcal{U}_t$ and $\mathcal{M}_t$). Hence, the path in $\Gamma$ yields an isotopy $h$ in $\Sp^2$ of $\mathcal{U}$ transforming $\mathcal{U}_0$ into $\mathcal{U}_1$ such that for all $t$, the isotopy at time $t$ is a resolution of some $\mathcal{U}_{t'}$ which is $\epsilon$-image equivalent to $ \mathcal{U}_{t'}$ and thus has less than $m$ intersections with $\mathcal{M}_{t'}$. 

Note that we did not modify the crossings of $\mathcal{M}_t$ and thus we preserve the decorations from the diagram $\mathcal{M}_t$ on these crossings. Therefore the diagrams $\mathcal{M}_t$ can be lifted to an isotopy $\phi' : \Sp^3 \times [0,1] \rightarrow \Sp^3$ satisfying the properties of \Cref{prop_iso_U}.

\begin{figure}[ht]
\begin{center}
\begin{tikzpicture}[scale = 0.65]
\def\e{0.05} \def\se{0.3} \def\me{0.8} \def\le{1.5} \def\lle{2}

\clip (-2.8,3.75) rectangle (20.1,-14.75);
\begin{scope}[yshift = 3cm]
\draw (0.5,0) circle (1.25cm and 0.55cm);
\end{scope}

\coordinate (c0) at (0,0);
\coordinate (c1) at (1,0);
\foreach \i in {0,...,1}{\node (n\i) [circle, inner sep = \e cm] at (c\i) {};}

\draw (n0.west) .. controls +(180:\le) and \control{(n0.south)}{(-90:\le)} -- (n0.north) .. controls +(90:\me) and \control{(n1.north)}{(90:\me)} -- (n1.south) .. controls +(-90:\le) and \control{(n1.east)}{(0:\le)} -- (n1.west) -- (n0.west);

\begin{scope}[yshift = -3cm]
\coordinate (c0) at (0,0);
\coordinate (c1) at (1,0);
\coordinate (c2) at (0.5,-0.8);
\coordinate (c3) at (0.5,-1.8);
\foreach \i in {0,1,2,3}{\node (n\i) [circle, inner sep = \e cm] at (c\i) {};}

\draw (n0.west) .. controls +(180:\le) and \control{(n3.south west)}{(-135:\lle)} -- (n3.north east) .. controls +(45:\se) and \control{(n2.south east)}{(-45:\se)} -- (n2.north west) .. controls +(135:\se) and \control{(n0.south)}{(-90:\se)} -- (n0.north) .. controls +(90:\me) and \control{(n1.north)}{(90:\me)} -- (n1.south) .. controls +(-90:\se) and \control{(n2.north east)}{(45:\se)} -- (n2.south west) .. controls +(-135:\se) and \control{(n3.north west)}{(135:\se)} -- (n3.south east) .. controls + (-45:\lle) and \control{(n1.east)}{(0:\le)} -- (n0.west);
\end{scope}

\begin{scope}[yshift = -8cm]
\coordinate (c0) at (0.5,0.3);
\coordinate (c1) at (0,-0.8);
\coordinate (c2) at (1,-0.8);
\coordinate (c3) at (0.5,-1.8);
\foreach \i in {0,1,2,3}{\node (n\i) [circle, inner sep = \e cm] at (c\i) {};}

\draw (n0.north west) .. controls +(135:\lle) and \control{(n0.north east)}{(45:\lle)} -- (n0.south west) .. controls +(-135:\se) and \control{(n1.north)}{(90:\se)} -- (n1.south) .. controls +(-90:\se) and \control{(n3.north west)}{(135:\se)} -- (n3.south east) .. controls +(-45:\le) and \control{(n2.east)}{(0:\le)} -- (n1.west) .. controls +(180:\le) and \control{(n3.south west)}{(-135:\le)} -- (n3.north east) .. controls +(45:\se) and \control{(n2.south)}{(-90:\se)} -- (n2.north) .. controls +(90:\se) and \control{(n0.south east)}{(-45:\se)} -- (n0.north west);
\end{scope}

\begin{scope}[xshift = 0.5cm, yshift = -12.5cm]
\coordinate (c0) at (0,0);
\coordinate (c1) at (-0.5,-0.8);
\coordinate (c2) at (-0.5,-2);
\foreach \i in {0,1}{\node (n\i) [circle, inner sep = \e cm] at (c\i) {};}

\draw (n0.north west) .. controls +(135:\lle) and \control{(n0.north east)}{(45:\lle)} -- (n0.south west) .. controls +(-135:\se) and \control{(n1.north)}{(90:\se)} -- (n1.south) .. controls +(-90:\le) and \control{(n1.east)}{(0:\le)} -- (n1.west) .. controls +(180:\le) and \control{(c2)}{(180:1.5)} .. controls +(0:\lle) and \control{(n0.south east)}{(-45:\le)} -- (n0.north west);
\end{scope}

\draw [Stealth-Stealth] (0.5,1) -- +(0,1) node [midway, right] {R2};
\draw [Stealth-Stealth] (0.5,-2) -- +(0,1) node [midway, right] {R2};
\draw [Stealth-Stealth] (0.5,-6.25) -- +(0,1) node [midway, right] {R3};
\draw [Stealth-Stealth] (0.5,-11.125) -- +(0,1) node [midway, right] {R2};

\def\ie{0.05}
\begin{scope}[xshift = 11cm]

\begin{scope}[yshift = 3cm]
\draw (0.5,0) circle (1.25cm and 0.55cm);
\end{scope}

\coordinate (c0) at (0,0);
\coordinate (c1) at (1,0);
\foreach \i in {0,...,1}{\node (n\i) [circle, inner sep = \e cm] at (c\i) {};}

\draw (n0.west) .. controls +(180:\le) and \control{(n0.south)}{(-90:\le)} .. controls +(90:\ie) and \control{(n0.east)}{(180:\ie)} -- (n1.west) .. controls +(0:\ie) and \control{(n1.south)}{(90:\ie)} .. controls +(-90:\le) and \control{(n1.east)}{(0:\le)} .. controls +(180:\ie) and \control{(n1.north)}{(-90:\ie)} .. controls +(90:\me) and \control{(n0.north)}{(90:\me)} .. controls +(-90:\ie) and \control{(n0.west)}{(0:\ie)};

\begin{scope}[yshift = -3cm]
\coordinate (c0) at (0,0);
\coordinate (c1) at (1,0);
\coordinate (c2) at (0.5,-0.8);
\coordinate (c3) at (0.5,-1.8);
\foreach \i in {0,1,2,3}{\node (n\i) [circle, inner sep = \e cm] at (c\i) {};}

\draw (n0.west) .. controls +(180:\le) and \control{(n3.south west)}{(-135:\lle)} .. controls +(45:\ie) and \control{(n3.north west)}{(-45:\ie)} .. controls +(135:\se) and \control{(n2.south west)}{(-135:\se)} .. controls +(45:\ie) and \control{(n2.north west)}{(-45:\ie)}.. controls +(135:\se) and \control{(n0.south)}{(-90:\se)} .. controls +(90:\ie) and \control{(n0.east)}{(180:\ie)} -- (n1.west) .. controls +(0:\ie) and \control{(n1.south)}{(90:\ie)}.. controls +(-90:\se) and \control{(n2.north east)}{(45:\se)} .. controls +(-135:\ie) and \control{(n2.south east)}{(135:\ie)} .. controls +(-45:\se) and \control{(n3.north east)}{(45:\se)} .. controls +(-135:\ie) and \control{(n3.south east)}{(135:\ie)} .. controls + (-45:\lle) and \control{(n1.east)}{(0:\le)} .. controls +(180:\ie) and \control{(n1.north)}{(-90:\ie)} .. controls +(90:\me) and \control{(n0.north)}{(90:\me)} .. controls +(-90:\ie) and \control{(n0.west)}{(0:\ie)};
\end{scope}

\begin{scope}[xshift = -7cm, yshift = -3cm]
\coordinate (c0) at (0,0);
\coordinate (c1) at (1,0);
\coordinate (c2) at (0.5,-0.8);
\coordinate (c3) at (0.5,-1.8);
\foreach \i in {0,1,2,3}{\node (n\i) [circle, inner sep = \e cm] at (c\i) {};}

\draw (n0.west) .. controls +(180:\le) and \control{(n3.south west)}{(-135:\lle)} .. controls +(45:\ie) and \control{(n3.south east)}{(135:\ie)} .. controls + (-45:\lle) and \control{(n1.east)}{(0:\le)} .. controls +(180:\ie) and \control{(n1.south)}{(90:\ie)} .. controls +(-90:\se) and \control{(n2.north east)}{(45:\se)} .. controls +(-135:\ie) and \control{(n2.south east)}{(135:\ie)} .. controls +(-45:\se) and \control{(n3.north east)}{(45:\se)} .. controls +(-135:\ie) and \control{(n3.north west)}{(-45:\ie)} .. controls +(135:\se) and \control{(n2.south west)}{(-135:\se)} .. controls +(45:\ie) and \control{(n2.north west)}{(-45:\ie)} .. controls +(135:\se) and \control{(n0.south)}{(-90:\se)} .. controls +(90:\ie) and \control{(n0.east)}{(180:\ie)} -- (n1.west) .. controls +(0:\ie) and \control{(n1.north)}{(-90:\ie)} .. controls +(90:\me) and \control{(n0.north)}{(90:\me)} .. controls +(-90:\ie) and \control{(n0.west)}{(0:\ie)};
\end{scope}

\begin{scope}[xshift = -3.5cm, yshift = -3cm]
\coordinate (c0) at (0,0);
\coordinate (c1) at (1,0);
\coordinate (c2) at (0.5,-0.8);
\coordinate (c3) at (0.5,-1.8);
\foreach \i in {0,1,2,3}{\node (n\i) [circle, inner sep = \e cm] at (c\i) {};}

\draw (n0.west) .. controls +(180:\le) and \control{(n3.south west)}{(-135:\lle)} .. controls +(45:\ie) and \control{(n3.north west)}{(-45:\ie)} .. controls +(135:\se) and \control{(n2.south west)}{(-135:\se)} .. controls +(45:\ie) and \control{(n2.south east)}{(135:\ie)} .. controls +(-45:\se) and \control{(n3.north east)}{(45:\se)} .. controls +(-135:\ie) and \control{(n3.south east)}{(135:\ie)} .. controls + (-45:\lle) and \control{(n1.east)}{(0:\le)} .. controls +(180:\ie) and \control{(n1.south)}{(90:\ie)} .. controls +(-90:\se) and \control{(n2.north east)}{(45:\se)} .. controls +(-135:\ie) and \control{(n2.north west)}{(-45:\ie)} .. controls +(135:\se) and \control{(n0.south)}{(-90:\se)} .. controls +(90:\ie) and \control{(n0.east)}{(180:\ie)} -- (n1.west) .. controls +(0:\ie) and \control{(n1.north)}{(-90:\ie)} .. controls +(90:\me) and \control{(n0.north)}{(90:\me)} .. controls +(-90:\ie) and \control{(n0.west)}{(0:\ie)};
\end{scope}

\begin{scope}[xshift = 3.5cm, yshift = -3cm]
\coordinate (c0) at (0,0);
\coordinate (c1) at (1,0);
\coordinate (c2) at (0.5,-0.8);
\coordinate (c3) at (0.5,-1.8);
\foreach \i in {0,1,2,3}{\node (n\i) [circle, inner sep = \e cm] at (c\i) {};}

\draw (n0.west) .. controls +(180:\le) and \control{(n3.south west)}{(-135:\lle)} .. controls +(45:\ie) and \control{(n3.north west)}{(-45:\ie)} .. controls +(135:\se) and \control{(n2.south west)}{(-135:\se)} .. controls +(45:\ie) and \control{(n2.south east)}{(135:\ie)} .. controls +(-45:\se) and \control{(n3.north east)}{(45:\se)} .. controls +(-135:\ie) and \control{(n3.south east)}{(135:\ie)} .. controls + (-45:\lle) and \control{(n1.east)}{(0:\le)} .. controls +(180:\ie) and \control{(n1.north)}{(-90:\ie)} .. controls +(90:\me) and \control{(n0.north)}{(90:\me)} .. controls +(-90:\ie) and \control{(n0.east)}{(180:\ie)} -- (n1.west) .. controls +(0:\ie) and \control{(n1.south)}{(90:\ie)} .. controls +(-90:\se) and \control{(n2.north east)}{(45:\se)} .. controls +(-135:\ie) and \control{(n2.north west)}{(-45:\ie)} .. controls +(135:\se) and \control{(n0.south)}{(-90:\se)} .. controls +(90:\ie) and \control{(n0.west)}{(0:\ie)};
\end{scope}

\begin{scope}[xshift = 7cm, yshift = -3cm]
\coordinate (c0) at (0,0);
\coordinate (c1) at (1,0);
\coordinate (c2) at (0.5,-0.8);
\coordinate (c3) at (0.5,-1.8);
\foreach \i in {0,1,2,3}{\node (n\i) [circle, inner sep = \e cm] at (c\i) {};}

\draw (n0.west) .. controls +(180:\le) and \control{(n3.south west)}{(-135:\lle)} .. controls +(45:\ie) and \control{(n3.south east)}{(135:\ie)} .. controls + (-45:\lle) and \control{(n1.east)}{(0:\le)} .. controls +(180:\ie) and \control{(n1.north)}{(-90:\ie)} .. controls +(90:\me) and \control{(n0.north)}{(90:\me)} .. controls +(-90:\ie) and \control{(n0.east)}{(180:\ie)} -- (n1.west) .. controls +(0:\ie) and \control{(n1.south)}{(90:\ie)} .. controls +(-90:\se) and \control{(n2.north east)}{(45:\se)} .. controls +(-135:\ie) and \control{(n2.south east)}{(135:\ie)} .. controls +(-45:\se) and \control{(n3.north east)}{(45:\se)} .. controls +(-135:\ie) and \control{(n3.north west)}{(-45:\ie)} .. controls +(135:\se) and \control{(n2.south west)}{(-135:\se)} .. controls +(45:\ie) and \control{(n2.north west)}{(-45:\ie)} .. controls +(135:\se) and \control{(n0.south)}{(-90:\se)} .. controls +(90:\ie) and \control{(n0.west)}{(0:\ie)};
\end{scope}

\begin{scope}[xshift = -4cm, yshift = -8cm]
\coordinate (c0) at (0.5,0.3);
\coordinate (c1) at (0,-0.8);
\coordinate (c2) at (1,-0.8);
\coordinate (c3) at (0.5,-1.8);
\foreach \i in {0,1,2,3}{\node (n\i) [circle, inner sep = \e cm] at (c\i) {};}

\draw (n0.north west) .. controls +(135:\lle) and \control{(n0.north east)}{(45:\lle)} .. controls +(-135:\ie) and \control{(n0.south east)}{(135:\ie)} .. controls +(-45:\se) and \control{(n2.north)}{(90:\se)} .. controls +(-90:\ie) and \control{(n2.west)}{(0:\ie)} -- (n1.east) .. controls +(180:\ie) and \control{(n1.south)}{(90:\ie)} .. controls +(-90:\se) and \control{(n3.north west)}{(135:\se)} .. controls +(-45:\ie) and \control{(n3.north east)}{(-135:\ie)} .. controls +(45:\se) and \control{(n2.south)}{(-90:\se)} .. controls +(90:\ie) and \control{(n2.east)}{(180:\ie)} .. controls +(0:\le) and \control{(n3.south east)}{(-45:\le)} .. controls +(135:\ie) and \control{(n3.south west)}{(45:\ie)} .. controls +(-135:\le) and \control{(n1.west)}{(180:\le)} .. controls +(0:\ie) and \control{(n1.north)}{(-90:\ie)} .. controls +(90:\se) and \control{(n0.south west)}{(-135:\se)} .. controls +(45:\ie) and \control{(n0.north west)}{(-45:\ie)};
\end{scope}

\begin{scope}[yshift = -8cm]
\coordinate (c0) at (0.5,0.3);
\coordinate (c1) at (0,-0.8);
\coordinate (c2) at (1,-0.8);
\coordinate (c3) at (0.5,-1.8);
\foreach \i in {0,1,2,3}{\node (n\i) [circle, inner sep = \e cm] at (c\i) {};}

\draw (n0.north west) .. controls +(135:\lle) and \control{(n0.north east)}{(45:\lle)} .. controls +(-135:\ie) and \control{(n0.south east)}{(135:\ie)} .. controls +(-45:\se) and \control{(n2.north)}{(90:\se)} .. controls +(-90:\ie) and \control{(n2.east)}{(180:\ie)} .. controls +(0:\le) and \control{(n3.south east)}{(-45:\le)} .. controls +(135:\ie) and \control{(n3.north east)}{(-135:\ie)} .. controls +(45:\se) and \control{(n2.south)}{(-90:\se)} .. controls +(90:\ie) and \control{(n2.west)}{(0:\ie)} -- (n1.east) .. controls +(180:\ie) and \control{(n1.south)}{(90:\ie)} .. controls +(-90:\se) and \control{(n3.north west)}{(135:\se)} .. controls +(-45:\ie) and \control{(n3.south west)}{(45:\ie)} .. controls +(-135:\le) and \control{(n1.west)}{(180:\le)} .. controls +(0:\ie) and \control{(n1.north)}{(-90:\ie)} .. controls +(90:\se) and \control{(n0.south west)}{(-135:\se)} .. controls +(45:\ie) and \control{(n0.north west)}{(-45:\ie)};
\end{scope}

\begin{scope}[xshift = 4cm, yshift = -8cm]
\coordinate (c0) at (0.5,0.3);
\coordinate (c1) at (0,-0.8);
\coordinate (c2) at (1,-0.8);
\coordinate (c3) at (0.5,-1.8);
\foreach \i in {0,1,2,3}{\node (n\i) [circle, inner sep = \e cm] at (c\i) {};}

\draw (n0.north west) .. controls +(135:\lle) and \control{(n0.north east)}{(45:\lle)} .. controls +(-135:\ie) and \control{(n0.south east)}{(135:\ie)} .. controls +(-45:\se) and \control{(n2.north)}{(90:\se)} .. controls +(-90:\ie) and \control{(n2.east)}{(180:\ie)} .. controls +(0:\le) and \control{(n3.south east)}{(-45:\le)} .. controls +(135:\ie) and \control{(n3.south west)}{(45:\ie)} .. controls +(-135:\le) and \control{(n1.west)}{(180:\le)} .. controls +(0:\ie) and \control{(n1.south)}{(90:\ie)} .. controls +(-90:\se) and \control{(n3.north west)}{(135:\se)} .. controls +(-135:\ie) and \control{(n3.north east)}{(-45:\ie)} .. controls +(45:\se) and \control{(n2.south)}{(-90:\se)} .. controls +(90:\ie) and \control{(n2.west)}{(0:\ie)} -- (n1.east) .. controls +(180:\ie) and \control{(n1.north)}{(-90:\ie)} .. controls +(90:\se) and \control{(n0.south west)}{(-135:\se)} .. controls +(45:\ie) and \control{(n0.north west)}{(-45:\ie)};
\end{scope}

\begin{scope}[xshift = 0.5cm, yshift = -12.5cm]
\coordinate (c0) at (0,0);
\coordinate (c1) at (-0.5,-0.8);
\coordinate (c2) at (-0.5,-2);
\foreach \i in {0,1}{\node (n\i) [circle, inner sep = \e cm] at (c\i) {};}

\draw (n0.north west) .. controls +(135:\lle) and \control{(n0.north east)}{(45:\lle)} .. controls +(-135:\ie) and \control{(n0.south east)}{(135:\ie)} .. controls +(-45:\le) and \control{(c2)}{(0:1.5)} .. controls +(180:\le) and \control{(n1.west)}{(180:\le)} .. controls +(0:\ie) and \control{(n1.south)}{(90:\ie)} .. controls +(-90:\le) and \control{(n1.east)}{(0:\le)} .. controls +(180:\ie) and \control{(n1.north)}{(-90:\ie)} .. controls +(90:\se) and \control{(n0.south west)}{(-135:\se)} .. controls +(45:\ie) and \control{(n0.north west)}{(-45:\ie)};
\end{scope}

\draw [Stealth-Stealth] (0.5,1) -- +(0,1) node [midway, right] {M2a};
\draw [Stealth-Stealth] (0.5,-2) -- +(0,1) node [midway, right] {M2a};
\draw [Stealth-Stealth] (0.5,-6.25) -- +(0,1) node [midway, right] {M3b};
\draw [Stealth-Stealth] (-5.5,-2.75) -- +(1.5,0) node [midway, above] {M2b};
\draw [Stealth-Stealth] (5,-2.75) -- +(1.5,0) node [midway, above] {M2b};
\draw [Stealth-Stealth] (-2.25,-8) -- +(1.5,0) node [midway, above] {M2b};
\draw [Stealth-Stealth] (-5.75,-5.5) -- +(1.5,-1.5) node [midway, above right] {M3a};
\draw [Stealth-Stealth] (-1.75,-5.25) -- +(1.5,-1.5) node [midway, below left] {M3b};
\draw [Stealth-Stealth] (6.75,-5.5) -- +(-1.5,-1.5) node [midway, above left] {M3a};
\draw [Stealth-Stealth] (2.75,-5.25) -- +(-1.5,-1.5) node [midway, below right] {M3b};
\draw [Stealth-Stealth] (3.25,-10) -- +(-2,-1.5) node [midway, below right] {M2a};
\end{scope}

\draw (-1.5,3.25) -- (-1.5,-14.25);
\draw (-1.5,3) -- +(-0.3,0) node [left] {$0$};
\draw (-1.5,-14) -- +(-0.3,0) node [left] {$1$};
\draw [gray] (-1.5,0) -- +(-0.2,0) node [left] {$t_1$};
\draw [gray] (-1.5,-4) -- +(-0.2,0) node [left] {$t_2$};
\draw [gray] (-1.5,-8.5) -- +(-0.2,0) node [left] {$t_3$};
\draw [gray] (-1.5,-12) -- +(-0.2,0) node [left] {$t_4$};

\begin{scope}[blue!80!black]
\draw (-1.5,2) -- +(-0.1,0) node [left] {$m_1$};
\path (-1.5,1.5) -- +(-0.1,0) node [left] {\small $\cdot$};
\path (-1.5,1.4) -- +(-0.1,0) node [left] {\small $\cdot$};
\path (-1.5,1.3) -- +(-0.1,0) node [left] {\small $\cdot$};
\draw (-1.5,1) -- +(-0.1,0) node [left] {$m_{i_1}$};
\end{scope}

\begin{scope}[blue!80!black, yshift = -3.1cm]
\draw (-1.5,2) -- +(-0.1,0) node [left] {$m_{i_2}$};
\path (-1.5,1.5) -- +(-0.1,0) node [left] {\small $\cdot$};
\path (-1.5,1.4) -- +(-0.1,0) node [left] {\small $\cdot$};
\path (-1.5,1.3) -- +(-0.1,0) node [left] {\small $\cdot$};
\draw (-1.5,1) -- +(-0.1,0) node [left] {$m_{i_3}$};
\end{scope}

\begin{scope}[blue!80!black, yshift = -7.2cm]
\draw (-1.5,2) -- +(-0.1,0) node [left] {$m_{i_4}$};
\path (-1.5,1.5) -- +(-0.1,0) node [left] {\small $\cdot$};
\path (-1.5,1.4) -- +(-0.1,0) node [left] {\small $\cdot$};
\path (-1.5,1.3) -- +(-0.1,0) node [left] {\small $\cdot$};
\draw (-1.5,1) -- +(-0.1,0) node [left] {$m_{i_5}$};
\end{scope}

\begin{scope}[blue!80!black, yshift = -12.1cm]
\draw (-1.5,2) -- +(-0.1,0) node [left] {$m_{i_6}$};
\path (-1.5,1.5) -- +(-0.1,0) node [left] {\small $\cdot$};
\path (-1.5,1.4) -- +(-0.1,0) node [left] {\small $\cdot$};
\path (-1.5,1.3) -- +(-0.1,0) node [left] {\small $\cdot$};
\draw (-1.5,1) -- +(-0.1,0) node [left] {$m_{i_7}$};
\end{scope}

\begin{scope}[blue!80!black, yshift = -14.6cm]
\draw (-1.5,2) -- +(-0.1,0) node [left] {$m_{i_8}$};
\path (-1.5,1.6) -- +(-0.1,0) node [left] {\small $\cdot$};
\path (-1.5,1.5) -- +(-0.1,0) node [left] {\small $\cdot$};
\path (-1.5,1.4) -- +(-0.1,0) node [left] {\small $\cdot$};
\draw (-1.5,1.2) -- +(-0.1,0) node [left] {$m_{i_9}$};
\end{scope}
\end{tikzpicture}
\caption{A graph of resolutions with times.}
\label{pic_ex_graph_times}
\end{center}
\end{figure}

\end{document}